\documentclass[11pt]{article}
\usepackage{amssymb,amsmath,amsthm,amsfonts,amssymb,graphicx,geometry,mathtools,mathrsfs,cases,color,cite,subcaption,hyperref}
\usepackage[T1]{fontenc} 
\usepackage{lmodern}
\usepackage[english]{babel}
\usepackage[utf8x]{inputenc}

\newtheorem{thm}{Theorem}[section]
\newtheorem{cor}[thm]{Corollary}
\newtheorem{lem}[thm]{Lemma}
\newtheorem{prop}[thm]{Proposition}

\newtheorem{rem}[thm]{Remark}

\allowdisplaybreaks

\def\LL{\left}
\def\RR{\right}
\def\EE{\mathbb{E}}

\begin{document}
	\title{Estimates for exponential functionals of continuous Gaussian processes with emphasis on fractional Brownian motion}
	\author{
		José Alfredo López-Mimbela 
		\footnote{Centro de Investigación en Matemáticas, Guanajuato, Mexico. jalfredo@cimat.mx} 
		\and  
		Gerardo Pérez-Suárez
		\footnote{Centro de Investigación en Matemáticas, Guanajuato, Mexico. gerardo.perez@cimat.mx
	}}
	\maketitle
	
	\begin{center}
		\textbf{Abstract}
	\end{center}

	 Our aim in this article is to provide explicit computable estimates for the cumulative distribution function (c.d.f.) and the $p$-th order moment of the exponential functional of a fractional Brownian motion (fBM) with drift.
	 Using elementary techniques, we prove general upper bounds for the c.d.f.\ of exponential functionals of  continuous Gaussian processes.
	 On the other hand, by applying classical results for extremes of Gaussian processes, we derive general lower bounds.
	 We also find estimates for the $p$-th order moment and the moment-generating function of such functionals.
	 As a consequence, we obtain explicit lower and upper bounds for the c.d.f.\ and the $p$-th order moment of the exponential functionals of a fBM, and of a series of independent fBMs. 
	 In addition, we  show the continuity in law of the exponential functional of a fBM with respect to the Hurst parameter.\newline
	 
	
	
	\section{Introduction and background}
	
	Let $H\in (0,1]$. A fractional Brownian motion (fBM for short) with Hurst parameter $H$  is a centered continuous Gaussian process $B^{H}\coloneqq \{ B_t^{H}: t\geq 0\} $ with covariance function 
	\begin{align*}
		\mathbb{E}\left[B^{H}_sB^{H}_t\right]=\frac{1}{2}\left(s^{2H}+t^{2H}-|s-t|^{2H}\right),\quad s,t\geq 0.
	\end{align*}
	The process $B^{H}$ is self-similar with index $H$. Additionally, $B^{H}$ has stationary increments and locally H\"older continuous sample paths with an arbitrary exponent smaller than $H$. For $H\in (0,1/2)$ the increments of $B^{H}$ are negatively correlated and exhibit short-range dependence, whereas for  $H\in (1/2,1)$ they are positively correlated and exhibit long-range dependence. On the other hand, $B^{H}$ is neither a semimartingale nor a Markov process if $H\in (0,1/2)\cup(1/2,1)$. For $H=1/2$ the process $B^{H}$ coincides with the standard Brownian motion, and for $H=1$ we have $B^{H}\stackrel{d}{=}\{tN: t\geq 0\}$, where $N$ is a standard normal random variable. For a detailed exposition of fBM, we refer the reader to Mishura \cite{mishu}, Nourdin \cite{nourd} and the references therein.
	
	In this paper,  by an exponential functional of $B^{H}$  we mean a random variable of the form
	\begin{align*}
		I_T^{\mu,\sigma,H}\coloneqq \int_0^{T}e^{\mu t +\sigma B_t^{H}}\,\mathrm{d}t,
	\end{align*}
	where $\mu\in\mathbb{R}$, $\sigma> 0$ and  $T\in (0,\infty]$ are constants. In the special case $H=1/2$,  the functional $I_T^{\mu,\sigma,H}$ is called Dufresne’s functional, and  plays an important role in several domains, e.g. continuous time finance models \cite{yor} and  one-dimensional disordered models \cite{comte}. The law of Dufresne's functional has been studied in the literature using different approaches. Pintoux and Privault  \cite{priv} proved an integral representation of the density function of $I_T^{\mu,\sigma,1/2}$ for $T<\infty$ by means of a Fokker–Planck equation for the density. In \cite{yor}, the law of $I_T^{\mu,\sigma,1/2}$ was studied replacing $T$ by an exponentially distributed random variable independent of $B^{1/2}$, and an explicit expression was obtained for the conditional density function of $I_T^{\mu,\sigma,1/2}$ given $B_T^{1/2}=x$ for $T<\infty$. 
	
	In the case where $T=\infty$,  the functional $I_T^{\mu,\sigma,1/2}$ is infinite a.s.\ if $\mu\geq 0$. On the other hand, if $\mu<0$, then  $I_\infty^{\mu,\sigma,1/2}\stackrel{d}{=}1/Y$, where $Y$ has Gamma distribution with shape parameter $-2\mu/\sigma^2$ and rate parameter $2/\sigma^{2}$. Equivalently, the cumulative distribution function (c.d.f.) of $I_\infty^{\mu,\sigma,1/2}$ is given by	
	\begin{align}\label{func_exp_brow}
			P\left[I_\infty^{\mu,\sigma,1/2}\leq x\right]=\int_0^{x}\frac{1}{\Gamma\left(-\frac{2\mu}{\sigma^{2}}\right) y}\left(\frac{2}{\sigma^{2}y}\right)^{-\frac{2\mu}{\sigma^{2}}} \exp\left(-\frac{2}{\sigma^{2}y}\right)\,\mathrm{d}y,\quad x>0,\quad \mu<0.
		\end{align}
	This formula was demonstrated originally by Dufresne \cite{duf} using weak convergence methods, and it was recovered in \cite{yor} by applying Lamperti's transformation. In \cite{salminen}  some integral functionals of $B^{1/2}$, including Dufresne's functional, were characterized as hitting times of a point for some diffusions. The research on exponential functionals of Brownian motion has been extended to other stochastic processes. In particular, the exponential functionals of Lévy processes have been investigated by several authors (see e.g. \cite{bert} and \cite{vos}). However, few results are known for the exponential functional of fBM.
	
 	To the best of our knowledge, the c.d.f.\ of $I_T^{\mu,\sigma,H}$ is unknown, and only a few estimates are available in the literature. Based on techniques from Malliavin calculus, Dung \cite{dung} obtained explicit upper bounds for the tail probabilities of a general class of exponential functionals that includes exponential functionals of fBM when $T<\infty$. In Dung \cite{dung1},  an upper bound for the c.d.f.\ of $I_\infty^{\mu,\sigma,H}$ was proved  for any $\mu<0$, $\sigma>0$ and $H\in (0,1)$. 
	%On the other hand, by using the law of the iterated logarithm for fBm, Dozzi, Kolkovska and López-Mimbela \cite{dozzi3} proved that $I_\infty^{\mu,\sigma,H}<\infty$ $P$-a.s.\  if $\mu<0$, and  $I_\infty^{\mu,\sigma,H}=\infty$ $P$-a.s.\  if $\mu>0$, for any $\sigma>0$ and $H\in(0,1)$. 
	In a recent paper, Dung, Hang and Thuy \cite{dung2} established a log-normal upper bound for the density function of $I_T^{\mu,\sigma,H}$ for any $\mu\in\mathbb{R}$, $\sigma>0$, $T<\infty$ and $H\in (1/2,1)$. 
	
	In this paper we estimate from above and from below the c.d.f.\ of the exponential functional of fBM. 
	In order to show the kind of estimates we obtain, let us designate by $\Phi$  the c.d.f.\ of the standard normal distribution.
	We first consider a general continuous Gaussian process $X\equiv\{X_t:t\ge0\}$, and 
	the family of density functions $f$ on $[0,T)$ for which
	$$
	\int_0^T\LL((\log f(t))^2 +\EE[X^2_t]\RR)f(t)\,\mathrm{d}t<\infty.
	$$
	We show that  $$J^{X}_T(f):=\int_0^T\LL(-\log f(t)+X_t\RR)f(t)\,\mathrm{d}t$$
	is a Gaussian random variable such that $e^{J^{X}_T(f)} \le \int_0^Te^{X_t}\,\mathrm{d}t$ a.s., hence
	$P\big[\int_0^Te^{X_t}\,\mathrm{d}t\le x\big]\le P\big[e^{J^{X}_T(f)}\le x\big]$. This allows us to get estimates of the form
	$$P\LL[\int_0^Te^{X_t}\,\mathrm{d}t\le x\RR] \le \Phi\LL(\frac{\log x-m^X_T(f)}{s^X_T(f)}\RR),$$ where $m^X_T(f)$ and $s^X_T(f)$ are defined in Section \ref{section-2}. By taking the infimum over all such densities $f$, we deduce a sharper estimate (Proposition \ref{general_bound}).

	To deal with the fBM functional $\int_0^T e^{\mu t +\sigma B^H_t}\,\mathrm{d}t$ when $T<\infty$, we have to work with a more restrictive set of density functions $f_{\lambda,T}$ given by (\ref{exponential-density}) below. We show that 
	\begin{align}\label{estimate1}
		P\LL[\int_0^T e^{\mu t +\sigma B^H_t}\,\mathrm{d}t\le x \RR]\leq \Phi\left(\frac{\sqrt{2H+2}}{\sigma T^{H}}\left(\log(x)-\log(T)-\frac{\mu T}{2}\right)\right),
	\end{align}
	and for the case $T=\infty$,
	\begin{align}\label{estimate2}
		P\LL[\int_0^\infty e^{\mu t +\sigma B^H_t}\mathrm{d}t\le x \RR]\leq \Phi\left(\inf_{\lambda\in (0,\infty)}\frac{\lambda^{H}}{\sigma}\sqrt{\frac{2}{\Gamma(2H+1)}}\left(\log(x)+\log(\lambda)-1-\frac{\mu}{\lambda}\right)\right),
	\end{align}
	for any $\mu\in\mathbb{R}$, $\sigma>0$, $H\in(0,1]$ and $x>0$. An explicit expression for the infimum in (\ref{estimate2}) is given in Theorem \ref{bounds_infinity}, which coincides with the exact c.d.f.\ of $I_\infty^{\mu,\sigma,H}$ in the case $H=1$. 
	
	The log-normal estimate (\ref{estimate1}) is consistent with the upper bound proved in \cite{dung2}. In contrast to the estimates obtained in \cite{dung1} and \cite{dung2},  our  bounds (\ref{estimate1}) and (\ref{estimate2})  are valid for any $\mu\in\mathbb{R}$ and $H\in (0,1]$, and do not depend on unknown constant parameters. Also, they are obtained by elementary arguments.
	We believe that this  approach will be useful for 
	potential practitioners.
	
	As a companion to estimate (\ref{estimate1}), we prove limit theorems of the form
	\begin{align*}
		\lim_{n\rightarrow\infty} \Phi\left(\frac{\sqrt{2H+2}}{\sigma_n T_n^{H}}\left(\log(x)-\log(T_n)-\frac{\mu_n T_n}{2}\right)\right)-P\left[\int_0^{T_n} e^{\mu_n t +\sigma_n B^H_t}\,\mathrm{d}t\leq x\right]=0,
	\end{align*}
	valid for all $x>0$ under suitable conditions on the sequences of parameters $\{\mu_n\}$, $ \{\sigma_n\}$ and $\{T_n\}$. In the case $T<\infty$ we also prove  that
	\begin{align*}
		\lim_{(\mu_n,\sigma_n,H_n)\rightarrow(\mu,\sigma,H)}\sup_{x\in\mathbb{R}} \left| P\left[I_{T}^{\mu_n,\sigma_n,H_n}\leq x\right]-P\left[I_{T}^{\mu,\sigma,H}\leq x\right]\right|=0.
	\end{align*}
	Additionally, we derive from (\ref{estimate1}) and (\ref{estimate2}) estimates for the moment-generating function of %fBM exponential functionals 
	$I_T^{\mu,\sigma,H}$ (see Corollary \ref{cor-gmf-finite} and Corollary \ref{cor-gmf-infinite}). 
	%For $T<\infty$, we obtain
	%\begin{align*}
	%	\mathbb{E}\left[e^{-\lambda I_T^{\mu,\sigma,H}}\right]\leq  \inf_{\varepsilon\in(0,1)} \left\lbrace  \varepsilon+(1-\varepsilon)\Phi\left(\frac{\sqrt{ 2H+2}}{\sigma T^{H}}\left(\log\left(\log\left(\frac{1}{\varepsilon}\right)\right)-\log(\lambda)-\log(T)-\frac{\mu T}{2}\right)\right)\right\rbrace,
	%\end{align*}
	%and for the case $T=\infty$,
	%\begin{align*}
	%	&\mathbb{E}\left[e^{-\lambda  I_\infty^{\mu,\sigma,H}}\right]\\
	%	&\leq \inf_{\substack{\varepsilon\in(0,1),\\ \delta>0}} \left\lbrace \varepsilon+(1-\varepsilon)\Phi\left(\frac{\delta^{H}}{\sigma}\sqrt{\frac{2}{\Gamma(2H+1)}}\left(\log\left(\log\left(\frac{1}{\varepsilon}\right)\right)-\log(\lambda)+\log(\delta)-1
	%	-\frac{\mu}{\delta}\right)\right)\right\rbrace,
	%\end{align*} 
	%for any $\lambda>0$. Also, in both cases, $\mathbb{E}\LL[e^{-\lambda I_T^{\mu,\sigma,H}}\RR]=\infty$ for all $\lambda<0$. 
	Finiteness and lower bounds for the $p$-th order moments of fBM exponential functionals are also explored (Corollary \ref{corolario-momento-caso-finito} and Corollary \ref{cor-moment-func-infinte}). 
	
	We estimate from below the c.d.f.\ of exponential functionals as well. 
	Lower bounds for the c.d.f.\ of $\int_{0}^{T}e^{X_t}\,\mathrm{d}t$ are obtained by using the inequality
	\begin{align*}
		P\left[\int_{0}^{T}e^{X_t}\,\mathrm{d}t\leq x\right]\geq P\left[\sup_{t\in [0,T)}\left(X_t-\mathbb{E}[X_t]-f(t)\right)\leq \log\left(\frac{x}{\int_0^{T}e^{\mathbb{E}[X_t]+f(t)}\,\mathrm{d}t}\right) \right],\quad x>0,
	\end{align*}
	for a suitable  continuous function $f$ that satisfies $\int_0^{T}e^{\mathbb{E}[X_t]+f(t)}\,\mathrm{d}t<\infty$. By applying standard results for extremes of Gaussian processes, we are able to prove lower bounds for $	P\LL[\int_0^T e^{\mu t +\sigma B^H_t}\,\mathrm{d}t\le x \RR]$ for  any $\mu\neq 0$ and $x$ large enough. In particular, for $\mu\neq 0$,  $H\in\left[1/2,1\right]$ and $T<\infty$, we get
	\begin{align*}
		P\left[\int_0^T e^{\mu t +\sigma B^H_t}\mathrm{d}t> x\right]\leq 2\Phi\left(\frac{1}{\sigma T^{H}}\log\left(\frac{e^{\mu T}-1}{\mu x}\right)\right),\quad x>\frac{e^{\mu T}-1}{\mu}.
	\end{align*}
	This estimate improves the bound given in \cite{dung} (Remark \ref{improved-bound}), showing that our approach allows to obtain similar or better estimates than those obtained previously by using Malliavin calculus.

	In the final part of the paper, we investigate the exponential functional of the process $Z\equiv\{Z_t:t\ge0\}$ defined by
	$$Z_t=\sum_{n=1}^{\infty}\sigma_nB_n^{H_n}(t),
	\quad\mbox{with}\quad\mbox{$\sum_{n=1}^{\infty}\sigma^2_n\in(0,\infty),$}$$
	where $\{B^{H_n}_n\}$ is a sequence of independent fBMs such that the Hurst parameters satisfy $H_n\in(0,1]$ and $H_0\coloneqq \inf_{n\ge1}H_n>0$.
	The  process $Z$ is a centered continuous Gaussian process with stationary increments, and it has locally  H\"older continuous trajectories with an arbitrary exponent smaller than $H_0$.
	It is a natural generalization of the mixed fractional Brownian motion (mfBM), which was introduced by Cheridito \cite{che} to model the discounted stock price in some arbitrage-free and complete financial markets.  We refer the interested reader to Mishura and Zili \cite{mishu2}, and the references therein,  for definitions and main properties of the mfBM.
	
	As far as we know, the exponential functionals of $Z$ have not been investigated in the literature. By using the estimates (\ref{estimate1}) and (\ref{estimate2}), we get upper bounds for the c.d.f.\ of the exponential functionals of $Z$, which now involve $H_0$ and $ H_{\infty}\coloneqq\sup_{n\ge1}H_n$. In addition, we investigate the finiteness of such functionals. We prove that $\int_0^{\infty}e^{\mu t+Z_t}\,\mathrm{d}t$ is infinite a.s.\ if $\mu>0$. On the other hand, if $\mu<0$, we show that $\int_0^{\infty}e^{\mu t+Z_t}\,\mathrm{d}t$ is finite a.s.\ provided $H_\infty<(1+H_0)/2$. We estimate the $p$-th order moments of these functionals as well.
	
	Another of our motivations to investigate the fBM exponential functionals is to obtain a lower bound for the blowup probability  %the positive solutions 
	of positive solutions of stochastic partial differential equations (SPDEs) of the prototype
	\begin{align}
		\mathrm{d}u(t,x)&=\left(\Delta u(t,x)+\gamma u(t,x)+C(u(t,x))^{1+\beta}\right)\,\mathrm{d}t+\kappa u(t,x)\,\mathrm{d}B^{H}_t, \quad \textrm{with}\quad H\in[1/2,1).\label{original_equation}
	\end{align}
 	 This application of the fBM exponential functionals is described in detail and extended 
 	 in a separated paper, where we study a semilinear SPDE of the form (\ref{original_equation}) but replacing $B^{H}$ by a centered Gaussian process $\tilde{Z}$ with H\"older continuous sample paths with exponent greater than $1/3$. The bounds obtained in this article for the exponential functional of $Z$ are  then applied to estimate from below the blowup probability of positive solutions of such SPDE in the case where $\tilde{Z}\equiv Z$ and $H_0>1/3$.
 	 We refer the reader to Dozzi, Kolkovska and López-Mimbela \cite{dozzi2}, who initially investigated (\ref{original_equation}),  and the related papers \cite{dozzi,dozzi3,guerrero} for applications of exponential functionals to semilinear SPDEs.
 	 
	This article is organized as follows. In Section 2, we prove a general log-normal type upper bound and a lower bound for the c.d.f.\ of the exponential functionals of  continuous Gaussian processes.  We also prove bounds for the moments and the moment-generating function of such functionals. These estimates are used in Section 3 to derive explicit upper and lower bounds for the c.d.f.\ and the moments of the exponential functional of fBM. In addition, we show the continuity in law of $I_T^{\mu,\sigma,H}$ with respect to parameters $\mu,\sigma,H$. In Section 4, we consider the exponential functional of a series of independent fBMs and obtain upper and lower bounds for its c.d.f.\ and moments.
	Finally, in Section 5 we present some plots of our estimates for the c.d.f. of $I_T^{\mu,\sigma,H}$.
	
	\section{Exponential functionals of continuous Gaussian processes}\label{section-2}
	
	In this section, we establish general upper and lower bounds for the c.d.f.\ of the exponential functional of a continuous Gaussian process.  We derive estimates for its $p$-th order moment and moment-generating function as well.
	 
	Let  $T\in (0,\infty]$, and  let $X\coloneqq\left\lbrace X_t: t\geq 0 \right\rbrace$ be a  continuous  Gaussian process defined on a probability space $(\Omega,\mathcal{F},P)$. We introduce the random variable

	$$I_{T}^{X}\coloneqq\int_0^{T}e^{X_t}\,\mathrm{d}t$$

	and call it  the exponential functional of $X$ on $[0,T)$.
	
	\subsection{Upper bounds} 
	We denote by $\mathcal{M}_T$ the set of all continuous functions $f\colon[0,T)\rightarrow (0,\infty)$ such that $\int_0^{T}f(t)\,\mathrm{d}t=1$, and define 
	\begin{align*}
		\mathcal{M}_T^{X}\coloneqq \left\lbrace f\in\mathcal{M}_T: \int_0^{T}\left((\log (f(t)))^{2}+\mathbb{E}[X_t^{2}]\right)f(t)\,\mathrm{d}t<\infty\right\rbrace.
	\end{align*}
		 For  each $f\in\mathcal{M}_T^{X}$, we set
	\begin{align*}
		m_{T}^{X}(f)&\coloneqq\int_0^{T}\left(-\log (f(t))+\mathbb{E}[X_t]\right)f(t)\,\mathrm{d}t,\\
		s_{T}^{X}(f)&\coloneqq\left(\int_0^{T}\int_0^{T}\mathrm{Cov}( X_s, X_t)f(s)f(t)\,\mathrm{d}s\,\mathrm{d}t\right)^{\frac{1}{2}},
	\end{align*}
	and define the random variable
	\begin{align*}
		J_{T}^{X}(f)\coloneqq\int_0^{T}\left(- \log(f(t))+ X_t\right)f(t)\,\mathrm{d}t.
	\end{align*}
	\begin{lem}\label{ine_basic}
		Let $f\in \mathcal{M}_T^{X}$. The random variable $J_{T}^{X}(f)$ has normal distribution with mean $m_T^{X}(f)$ and variance $(s_T^{X}(f))^2$, and it satisfies
		\begin{align}\label{log_norm}
			e^{J_{T}^{X}(f)}\leq I_T^{X}\quad P\textrm{-a.s.\ }
		\end{align}
	\end{lem}

	\begin{proof}
			Let $f\in\mathcal{M}_T ^{X}$ . From Jensen's inequality it follows that
		\begin{align*}
			\mathbb{E}\left[\left(\int_0^{T}\left|- \log(f(t))+ X_t\right|f(t)\,\mathrm{d}t\right)^{2}\right]
			&\leq \mathbb{E}\left[	\int_0^{T}\left(-\log(f(t))+ X_t\right)^{2}f(t)\,\mathrm{d}t\right]\\ 
			&\leq  	2 \int_0^{T}\left( (\log(f(t)))^{2}+\mathbb{E}[ X_t^{2}]\right)f(t)\,\mathrm{d}t\\ 
			&<\infty.
		\end{align*}
		Therefore, $J_{T}^{X}(f)$ is  $P$-a.s.\  well-defined.  Moreover, since $X$ is a Gaussian process, we deduce that $J_T^{X}(f)$ has normal distribution with mean
		\begin{align*}
			\mathbb{E}\left[J_{T}^{X}(f)\right]=\int_0^{T}\left(- \log(f(t))+\mathbb{E}[X_t]\right)f(t)\,\mathrm{d}t=m_{T}^{X}(f),
		\end{align*}
		and variance
		\begin{align*}
			\mathbb{E}\left[\left(J_{T}^{X}(f)-m_{T}^{X}(f)\right)^{2}\right]
			&=\mathbb{E}\left[\left(\int_0^{T}\left( X_t-\mathbb{E}[X_t]\right)f(t)\,\mathrm{d}t\right)^{2}\right]\\
			&=\int_0^{T}\int_0^{T}\mathrm{Cov}( X_s, X_t)f(s)f(t)\,\mathrm{d}s\,\mathrm{d}t \\ 
			&=\left(s_{T}^{X}(f)\right)^{2}.
		\end{align*}  
	Using again Jensen's inequality, we get
	\begin{align*}
		I_T^{X}=\int_0^{T}e^{-\log(f(t))+X_t}f(t)\,\mathrm{d}t\geq \exp\left(\int_0^{T}\left(-\log(f(t))+X_t\right) f(t)\,\mathrm{d}t\right)=e^{J_T^{X}(f)}.
	\end{align*}
 	This completes the proof.
	\end{proof}
	Let us denote the c.d.f.\ of the standard normal distribution  by $\Phi$. We have the following upper bounds for the c.d.f.\ of $I_T^{X}$.
	
	\begin{lem}
		Let $x>0$ and $f\in\mathcal{M}_T^{X}$. It holds that
		\begin{align}\label{log_norm_proba}
			P\left[I_T^{X}\leq x\right]\leq P\left[ e^{J_T^{X}(f)}\leq x\right].
		\end{align}
	If $s_{T}^X(f)>0$, then
	\begin{align}\label{general-estimate-positive}
			P\left[I_T^{X}\leq x\right]\leq
			\Phi\left(\frac{\log(x)-m_T^{X}(f)}{s_T^{X}(f)}\right).
	\end{align}
	On the other hand, if $s_T^{X}(f)=0$, then $P\big[I_T^{X}\leq x\big]=0$ for all $x<e^{m_T^{X}(f)}$.
	\end{lem}

	\begin{proof}
		The inequality (\ref{log_norm_proba}) follows from (\ref{log_norm}). If $s_T^X(f)>0$, then
		\begin{align*}
			P\left[ e^{J_T^{X}(f)}\leq x\right]=P\left[ J_T^{X}(f)\leq \log(x)\right]=
			\Phi\left(\frac{\log(x)-m_T^{X}(f)}{s_T^{X}(f)}\right).
		\end{align*}
	On the other hand, if $s_T^X(f)=0$, we have $P\big[ e^{J_T^{X}(f)}\leq x\big]=\textbf{1}_{\{e^{m_T^{X}(f)}\leq x\}}$.
	\end{proof}
	The estimate (\ref{log_norm_proba}) means that the c.d.f.\ of the exponential functional $I_T^{X}$ is upper bounded by the c.d.f.\ of the log-normal random variable $e^{J_T^{X}(f)}$ for each $f\in\mathcal{M}_{T}^{X}$. This implies the following improved upper bound.

	\begin{prop} \label{general_bound}
		   Let $x>0$. It holds that
	 \begin{align}
	 		P\left[ I_T^{X}\leq x\right]\leq  \Phi\left(\inf_{\substack{f\in\mathcal{M}_T^{X}:\\s_T^{X}(f)>0}}\frac{\log\left(x\right)-m_{T}^{X}(f)}{s_{T}^{X}(f)}\right)\label{general_estimate}.
	 \end{align}
	\end{prop}
	\begin{proof}
    Taking the infimum over all $f\in \mathcal{M}_{T}^{X}$ such that $s_T^{X}(f)>0$ in (\ref{general-estimate-positive}), and using the continuity and monotonicity of $\Phi$, we obtain the desired inequality.
	\end{proof}

	\begin{rem}\emph{Notice that to obtain non-trivial upper bounds for the c.d.f.\ of $I_T^{X}$, we need $\mathcal{M}_T^{X}\neq \emptyset$. In the case $T<\infty$, we observe that $\mathcal{M}_T^X$ contains all bounded functions in $\mathcal{M}_T$. In particular, the constant function $f\equiv \frac{1}{T}$ belongs to $\mathcal{M}_T^X$. On the other side, in the case $T=\infty$, the set $\mathcal{M}_T^{X}$ depends strongly on the covariance function of $X$. Nevertheless, it always contains all functions of the form
			\begin{align*}
				h_\lambda(t)\coloneqq C_\lambda e^{-(\lambda t+\mathbb{E}[X_t^{2}])},\quad t\geq 0,
			\end{align*}
	where $C_\lambda\coloneqq\big(\int_0^{\infty}e^{-(\lambda t+\mathbb{E}[X_t^{2}])}\,\mathrm{d}t\big)^{-1}$, for each $\lambda>0$. Thus, $\mathcal{M}_T^{X}$ is never empty.} 
	\end{rem}

	\begin{rem}\emph{The estimate (\ref{general_estimate}) remains true if we  replace $\mathcal{M}_T^{X}$ by any subset 					$\widetilde{\mathcal{M}}_T^{X}\subset\mathcal{M}_T^{X}$.
			In order to obtain explicit upper bounds  for the c.d.f.\ of $I_T^{X}$, we need to choose $\widetilde{\mathcal{M}}_T^{X}$  such that both $m_T^{X}(f)$ and $s_T^X(f)$ are computable for each $f\in\widetilde{\mathcal{M}}_T^{X}$. 
			 Suppose that $\widetilde{\mathcal{M}}_T^{X}$ is indexed by an interval $I\subset\mathbb{R}$, i.e.,  $\widetilde{\mathcal{M}}_T^{X}=\{p_\lambda:\lambda\in I\}$. If the functions $\lambda\mapsto m_T^{X}(p_\lambda)$ and $\lambda\mapsto s_T^{X}(p_\lambda)$  are differentiable, and $s_T^{X}(p_\lambda)>0$ for all $\lambda\in I$, then we can apply basic calculus methods to find the value of
			 $$
			 \Phi\left(\inf_{f\in\widetilde{\mathcal{M}}_T^{X}}\frac{\log\left(x\right)-m_{T}^{X}(f)}{s_{T}^{X}(f)}\right)=\Phi\left(\inf_{\lambda\in I}\frac{\log\left(x\right)-m_{T}^{X}(p_\lambda)}{s_{T}^{X}(p_\lambda)}\right),
			 $$
			 which is an upper bound for $P\left[ I_T^{X}\leq x\right]$.
			 We will implement this strategy to find estimates for the c.d.f.\ of exponential functionals of fBM in Section 3.
		 }
	\end{rem}

	\begin{rem}\emph{
			We can use Proposition \ref{general_bound} to estimate from above  the c.d.f.\ of exponential functionals of self-similar continuous Gaussian processes as follows.
			Suppose that $X$ is, in addition, self-similar with index $H>0$.
			Then $\mathbb{E}[X_t]=t^{H}\mathbb{E}[X_1]$ and  $\mathbb{E}[X_t^{2}]=t^{2H}\mathbb{E}[X_1^{2}]$ for each $t\geq 0$. 
			Let $f\colon [0,\infty)\rightarrow\mathbb{R}$ be a continuous function  with at most polynomial growth.
			Denote by $X+f$ the process $\{X_t+f(t): t\geq 0\}$.
			Let $\mathcal{N}_\infty$ be the set of all densities on $[0,\infty)$ of the form $f_\lambda(t)\coloneqq  \lambda e^{-\lambda t}$, $t\geq 0$, for each $\lambda>0$. Notice that
			 $$
			 \int_0^{\infty} \left( \left(\log (f_\lambda(t))\right)^{2}+\mathbb{E}[(X_t+f(t))^{2}]\right) f_\lambda(t)\,\mathrm{d}t<\infty
			 $$
			 for all $\lambda>0$. Therefore, $\mathcal{N}_\infty\subset\mathcal{M}_\infty^{X+f}$.  We deduce from Proposition \ref{general_bound} that
			 $$
			 	P\left[ I_\infty^{X+f}\leq x\right]\leq\Phi\left(\inf_{\lambda\in (0,\infty)}\frac{\log\left(x\right)-m_{T}^{X+f}(f_\lambda)}{s_{T}^{X+f}(f_\lambda)}\right),
			 $$	
			 where	
			 \begin{align*}
			 	 m_{\infty}^{X+f}(f_\lambda)&=\int_0^{\infty}\left(-\log (f_\lambda(t))+\mathbb{E}[X_t+f(t)]\right)f_\lambda(t)\,\mathrm{d}t\\
			 	 &=-\log(\lambda)+1+\frac{\Gamma(H+1)}{\lambda^{H}}\mathbb{E}[X_1]+\lambda\int_0^{\infty}f(t)e^{-\lambda t}\,\mathrm{d}t,\\
			 	 \left(s_\infty^{X+f}(f_\lambda)\right)^{2}&=\int_0^{\infty}\int_0^{\infty}\mathrm{Cov}( X_s+f(s), X_t+f(t))f_\lambda(s)f_\lambda(t)\,\mathrm{d}s\,\mathrm{d}t\\
			 	 &=\lambda^{2}\int_0^{\infty}\int_0^{\infty}\mathrm{Cov}( X_s, X_t)e^{-\lambda(t+s)}\,\mathrm{d}s\,\mathrm{d}t.
			 \end{align*}
		The class of self-similar Gaussian processes  includes, as particular cases, the fBM,  sub-fractional Brownian motion and  bifractional Brownian motion. We refer the reader to Tudor \cite{tudor} for definitions and properties of these processes.
		It also includes the weighted
		sub-fractional Brownian motion \cite{jose}.
		} 
	\end{rem}

	\begin{rem}\emph{
		Adapting the proof of Lemma \ref{ine_basic}, we can obtain upper estimates for the c.d.f.\ of other integral functionals of continuous Gaussian processes.
		Let $f\colon\mathbb{R}\rightarrow[0,\infty)$ be a  convex strictly increasing function,
		and let $\mu$ be a probability measure on $[0,T)$.
		If  $\int_0^{T}\mathbb{E}[X_t^{2}]\,\mathrm{d}\mu(t)<\infty$, then
		\begin{align*}
			f\left(\int_0^{T} X_t\,\mathrm{d}\mu(t)\right)\leq \int_0^{T} f(X_t)\,\mathrm{d}\mu(t)\quad P\textrm{-a.s.}
		\end{align*}
		Moreover, if $s_T^{X}>0$, then 
		\begin{align*}
			P\left[ \int_0^{T} f(X_t)\,\mathrm{d}\mu(t)\leq x\right]\leq  \Phi\left(\frac{f^{-1}(x)-m_T^{X}}{s_T^{X}}\right),\quad x>0,
		\end{align*}
	where
	\begin{align*}
		m_T^{X}\coloneqq\int_0^{T}\mathbb{E}[X_t]\,\mathrm{d}\mu(t),\quad \left(s_T^{X}\right)^{2}\coloneqq\int_0^{T}\int_0^{T}\textrm{Cov}(X_s, X_t)\,\mathrm{d}\mu(s)\,\mathrm{d}\mu(t).
	\end{align*}
	}
	\end{rem}
	
	\begin{rem}\emph{
			Roughly speaking, the estimate in Proposition \ref{general_bound} states that the tail of $I_T^{X}$  decays no faster than the tail of a log-normal distribution.
			In fact, it is possible that the tail of $I_T^{X}$ decays as fast as the tail of a log-normal distribution.
			As an example, consider the process $X_t\coloneqq \mu t+\sigma N$, $t\geq 0$, where $\mu<0$, $\sigma>0$ and $N$ is a standard normal random variable.
			We have
			\begin{align*}
				I_\infty^{X}=\int_0^{\infty}e^{\mu t+\sigma N}\,\mathrm{d}t=\frac{e^{\sigma N}}{-\mu},
			\end{align*}
			so $I_\infty^{X}$ has log-normal distribution.
			In this case, the right-hand side of the estimate \eqref{general_estimate} coincides with the exact c.d.f.\ of $I_\infty^{X}$.
			Taking $f(t)=(-\mu)e^{\mu t}$, $t\geq 0$, in \eqref{general-estimate-positive} we obtain
			\begin{align*}
				P\left[I_\infty^{X}\leq x\right]\leq \Phi\left(\frac{\log(-\mu x)}{\sigma}\right),\quad x>0,
			\end{align*}
			which is in fact an equality.
			On the other hand, notice that $I_\infty^{-X}\equiv \infty$, so the tail of $I_\infty^{-X}$ does not decay as a log-normal tail.
			From this example, we learn two facts.
			First, it is reasonable to estimate from above the c.d.f.\ of exponential functionals using log-normal type bounds.
			Secondly, it is possible that the c.d.f.\ of exponential functionals can not be lower bounded by a log-normal c.d.f.}\end{rem}

	\subsection{Lower bounds}
	
	In order to estimate from below the c.d.f.\ of $I_T^{X}$, we now define the set
		 
	\begin{align*}
		\mathcal{S}_T^{X}&\coloneqq  \left\lbrace  f\colon [0,T)\rightarrow \mathbb{R}\,\big|\, f\textrm{ is continuous and }  \int_0^{T}e^{\mathbb{E}[X_t]+f(t)}\,\mathrm{d}t<\infty \right\rbrace.
	\end{align*}

	In the following lemma, we obtain a lower bound for the c.d.f.\ of $I_T^{X}$ in terms of the c.d.f.\ of the supremum  over $[0,T)$ of the process $X$ with a drift in $\mathcal{S}_T^{X}$. 
	
	\begin{lem}\label{prop-general-lower-bound}
		Let $x>0$. Then
		\begin{align*}
			P\left[I_T^{X}\leq x\right]\geq \sup_{f\in\mathcal{S}_T^{X}} P\left[\sup_{t\in [0,T)}\left(X_t-\mathbb{E}[X_t]-f(t)\right)\leq \log\left(\frac{x}{\int_0^{T}e^{\mathbb{E}[X_t]+f(t)}\,\mathrm{d}t}\right) \right].
		\end{align*}
	\end{lem}
	\begin{proof}
		Fix $f\in\mathcal{S}_T^{X}$. Then we have
		\begin{align*}
			P\left[ I_T^{X}\leq x\right]&=P\left[\int_0^{T}e^{(X_t-\mathbb{E}[X_t]-f(t))+(f(t)+\mathbb{E}[X_t])}\,\mathrm{d}t\leq x\right]\\
			&\geq P\left[\exp\left(\sup_{t\in [0,T)}\left(X_t-\mathbb{E}[X_t]-f(t)\right)\right)\int_0^{T}e^{\mathbb{E}[X_t]+f(t)}\,\mathrm{d}t\leq x\right]\\
			&= P\left[\sup_{t\in [0,T)}\left(X_t-\mathbb{E}[X_t]-f(t)\right)\leq\log\left( \frac{x}{\int_0^{T}e^{\mathbb{E}[X_t]+f(t)}\,\mathrm{d}t}\right)\right].
		\end{align*}
	\end{proof}
	Notice that $\mathcal{S}_T^{X}$ contains any bounded continuous function on $[0,T)$ if $T<\infty$.  In this case, we can estimate from below the c.d.f.\ of $I_T^{X}$ 
	using the Borell-TIS inequality (see Theorem 4.2 in \cite{nourd}) as follows.
	\begin{lem}[Borell-TIS inequality]
		Let $Y\coloneqq \{Y_t: t\in[0,1]\}$ be a centered continuous Gaussian process, and let
		$\sigma^{2}\coloneqq \sup_{t\in[0,1]}\emph{\textrm{Var}}(Y_t)$.
		Then $\mu\coloneqq \mathbb{E}[\sup_{t\in[0,1]}Y_t]$ is finite. Moreover,
		\begin{align*}
			P\left[\sup_{t\in[0,1]} Y_t\geq x\right]\leq \exp\left(-\frac{(x-\mu)^{2}}{2\sigma^{2}}\right) \quad\textrm{for all}\quad x>\mu.
		\end{align*}
	\end{lem}

	\begin{prop}\label{cor-lower-bound-finite-t}
		Suppose that $T<\infty$. Let $f\colon[0,T)\rightarrow\mathbb{R}$ be a bounded continuous function. Then
		\begin{align*}
		P\left[ I_T^{X}\leq x\right]\geq 1-\exp\left(-\frac{1}{2\left(\sigma_T^{X}\right)^{2}}\left(\log\left( \frac{x}{\int_0^{T}e^{\mathbb{E}[X_t]+f(t)}\,\mathrm{d}t}\right)+f_{\inf}-\mu_T^{X}\right)^{2}\right)
		\end{align*}
	for all $x>e^{\mu_T^{X}-f_{\inf}}\int_0^{T}e^{\mathbb{E}[X_t]+f(t)}\,\mathrm{d}t$, where
	\begin{align*}
		\mu_T^{X}\coloneqq \mathbb{E}\left[\sup_{t\in[0,T]} \left(X_t-\mathbb{E}[X_t]\right)\right],\quad \left(\sigma_T^{X}\right)^{2}\coloneqq\sup_{t\in[0,T]}\emph{\textrm{Var}}(X_t),\quad f_{\inf}\coloneqq\inf_{t\in[0,T)}f(t).
	\end{align*}
	\end{prop}
	
	\begin{proof}
		It follows from Lemma \ref{prop-general-lower-bound} that
		\begin{align*}
			P\left[ I_T^{X}\leq x\right]&\geq P\left[\sup_{t\in [0,T)}\left(X_t-\mathbb{E}[X_t]\right)+\sup_{t\in[0,T)}(-f(t))\leq\log\left( \frac{x}{\int_0^{T}e^{\mathbb{E}[X_t]+f(t)}\,\mathrm{d}t}\right)\right]\\
				&= P\left[\sup_{t\in [0,T]}\left(X_t-\mathbb{E}[X_t]\right)\leq\log\left( \frac{x}{\int_0^{T}e^{\mathbb{E}[X_t]+f(t)}\,\mathrm{d}t}\right)+f_{\inf}\right]\\
		\end{align*}
	for all $x>0$. By applying the  Borell-TIS inequality, we get 
		\begin{align*}
			&P\left[\sup_{t\in [0,T]}\left(X_t-\mathbb{E}[X_t]\right)\leq\log\left( \frac{x}{\int_0^{T}e^{\mathbb{E}[X_t]+f(t)}\,\mathrm{d}t}\right)+f_{\inf}\right]\\
			&\geq 1-\exp\left(-\frac{1}{2\left(\sigma_T^{X}\right)^{2}}\left(\log\left( \frac{x}{\int_0^{T}e^{\mathbb{E}[X_t]+f(t)}\,\mathrm{d}t}\right)+f_{\inf}-\mu_T^{X}\right)^{2}\right)
		\end{align*}
	for all $x>e^{\mu_T^{X}-f_{\inf}}\int_0^{T}e^{\mathbb{E}[X_t]+f(t)}\,\mathrm{d}t$. This concludes the proof.
	\end{proof}

	\begin{rem}\emph{
		In the case $T=\infty$, in order to obtain lower bounds for the c.d.f.\ of $I_T^{X}$ by applying Lemma \ref{prop-general-lower-bound}, we need to estimate from below the c.d.f.\ of $\sup_{t\in [0,\infty)}\left(X_t-\mathbb{E}[X_t]-f(t)\right)$ for  $f\in\mathcal{S}_\infty^{X}$. In particular, this requires to know the growth rate of $X$. The asymptotic behaviour of Gaussian processes has been investigated in numerous research articles.  For instance, bounds for the tails of the supremum  over $[0,\infty)$ of Gaussian processes with drift, such as fBM, scaled Brownian motion and some integrated stationary Gaussian processes, are proved in \cite{debicki}.
		Asymptotic estimates for the tails of extremes of Gaussian processes with stationary increments, and of self-similar Gaussian processes, are studied in \cite{dieker} and the references therein.
	}\end{rem}
	
	\subsection{Moment and moment-generating function estimates}  
	
	Our last objective in this section is to estimate the $p$-th order moment and the moment generating function of $I_T^{X}$.
	By exploiting the convexity of the exponential function, 
	we get the following bounds for the moments of $I_T^{X}$.
	
	\begin{prop}\label{moments_functional}
		The following statements hold:	
		\begin{enumerate}
			\item[(i)]If $p\in(0,1)$, then
			\begin{align*}
				\sup_{f\in\mathcal{M}_T^{X}}\int_0^{T}e^{p\mathbb{E}[X_t]+\frac{1}{2}p^{2}\emph{\textrm{Var}}(X_t)}\left(f(t)\right)^{1-p}\,\mathrm{d}t \leq \mathbb{E}\left[\left(I_T^{X}\right)^{p}\right]\leq \left(\int_0^{T}e^{\mathbb{E}[X_t]+\frac{1}{2}\emph{\textrm{Var}}(X_t)}\,\mathrm{d}t\right)^{p}.
			\end{align*}
			\item[(ii)] If $p\geq 1$, then
			\begin{align*}
				\!\!\!\sup_{f\in\mathcal{M}_T^{X}}\exp\left(pm_T^{X}(f)+\frac{1}{2}p^{2}\left(s_T^{X}(f)\right)^{2} \right)\leq \mathbb{E}\left[\left(I_T^{X}\right)^{p}\right]\leq  \inf_{f\in\mathcal{M}_T^{X}}\int_0^{T}e^{p\mathbb{E}[X_t]+\frac{1}{2}p^{2}\emph{\textrm{Var}}(X_t)}\left(f(t)\right)^{1-p}\,\mathrm{d}t.
			\end{align*}
		
	\end{enumerate}
	\end{prop}		
	
	\begin{proof}
		From Lemma \ref{ine_basic} we have
		\begin{align*}
			\mathbb{E}\left[\left(I_T^{X}\right)^{p}\right]\geq \mathbb{E}\left[e^{p J_T^{X}(f)}\right]=\exp\left(pm_T^{X}(f)+\frac{1}{2}p^{2}\left(s_T^X(f)\right)^{2}\right)
		\end{align*}
		for all $p\geq 0$ and $f\in\mathcal{M}_T^{X}$, which implies the left-hand side inequality in statement (ii). If $p\geq1$ and $f\in\mathcal{M}_T^{X}$, it follows from Jensen's inequality that
		\begin{align*}
			\mathbb{E}\left[\left(I_T^{X}\right)^{p}\right]=\mathbb{E}\left[\left(\int_0^{T}e^{-\log(f(t))+X_t}f(t)\,\mathrm{d}t\right)^{p}\right]\leq \int_0^{T} \mathbb{E}\left[e^{p(-\log(f(t))+X_t)}\right] f(t)\,\mathrm{d}t,
		\end{align*}
		which implies the right-hand side inequality in (ii). The estimates in (i) follow by applying Jensen's inequality for concave function.
	\end{proof}
	
	\begin{cor}\label{cor-moment-bound}
	If $T<\infty$, then $I_T^{X}\in L^{p}(\Omega)$ for all $p\geq 1$.
	\end{cor}

	Now, by using the previous bounds for the c.d.f.\ of $I_T^{X}$, we can estimate the moment-generating function of $I_T^{X}$ as follows.
	
	\begin{prop}\label{general-ch-funct}
		The following statements hold:
		\begin{enumerate}
			\item[(i)] Suppose that there exists $f\in\mathcal{M}_T^{X}$ such that $s_{T}^{X}(f)>0$. If $\lambda>0$, then
			\begin{align*}
				\mathbb{E}\left[e^{-\lambda  I_T^{X}}\right]\leq \inf_{\varepsilon\in(0,1)}\left\lbrace \varepsilon+(1-\varepsilon)\Phi\left(\frac{\log(\log(\frac{1}{\varepsilon}))-\log(\lambda)-m_T^{X}(f)}{s_T^{X}(f)}\right)\right\rbrace.
			\end{align*} 
			Also, $\mathbb{E}\big[e^{-\lambda  I_T^{X}}\big]=\infty$ for all $\lambda<0$.
			\item[(ii)] Suppose that $T<\infty$ and $f\colon[0,T)\rightarrow\mathbb{R}$ is bounded and continuous. Then for any $\lambda>0$,
			\begin{align*}
				\!\!\!\!&\mathbb{E}\left[e^{-\lambda  I_T^{X}}\right]\\
				\!\!\!\!&\geq  \int_0^{\exp\left(-\lambda e^{\mu_T^{X}-f_{\inf}}\int_0^{T}e^{\mathbb{E}[X_t]+f(t)}\,\mathrm{d}t\right)}\left( 1-\exp\left(-\frac{\left(\log\left( \frac{\log\left(\frac{1}{x}\right)}{\lambda\int_0^{T}e^{\mathbb{E}[X_t]+f(t)}\,\mathrm{d}t}\right)+f_{\inf}-\mu_T^{X}\right)^{2}}{2\left(\sigma_T^{X}\right)^{2}}\right)\right)\,\mathrm{d}x,
			\end{align*}
			where $\mu_T^{X}$, $\sigma_T^{X}$ and $f_{\inf}$ are as in Proposition \ref{cor-lower-bound-finite-t}.
		\end{enumerate}
	\end{prop}
	\begin{proof}
		We first prove (i). Let $\lambda>0$, $\varepsilon\in(0,1)$ and $f\in\mathcal{M}_T^{X}$ with $s_T^{X}(f)>0$. From estimate (\ref{general-estimate-positive}) we have
		\begin{align*}
			\mathbb{E}\left[e^{-\lambda  I_T^{X}}\right]&=
			\int_0^{1}P\left[ e^{-\lambda  I_T^{X}}\geq x\right]\,\mathrm{d}x\\
			&=\int_0^{1}P\left[I_T^{X}\leq -\frac{\log(x)}{\lambda} \right]\,\mathrm{d}x\\
			&\leq \int_0^{1}\Phi\left(\frac{\log(\log(\frac{1}{x}))-\log(\lambda)-m_T^{X}(f)}{s_T^{X}(f)}\right)\,\mathrm{d}x\\
			&\leq \varepsilon+ \int_\varepsilon^{1}\Phi\left(\frac{\log(\log(\frac{1}{x}))-\log(\lambda)-m_T^{X}(f)}{s_T^{X}(f)}\right)\,\mathrm{d}x\\
			&\leq  \varepsilon+(1-\varepsilon)\Phi\left(\frac{\log(\log(\frac{1}{\varepsilon}))-\log(\lambda)-m_T^{X}(f)}{s_T^{X}(f)}\right),
			\end{align*}
		where we have used the monotonicity of the function $x\mapsto \log(\log(1/x))$ on $(0,1)$. Now, using the fact that the moment-generating function of a log-normal random variable is infinite on $(0,\infty)$, we obtain that 
		%$\mathbb{E}[e^{\lambda  I_T^{X}}]\geq \mathbb{E}[\exp(\lambda  e^{J_T^{X}(f)})]=\infty$ for all $\lambda>0$. 
		\begin{align*}
			\mathbb{E}\left[e^{\lambda  I_T^{X}}\right]\geq \mathbb{E}\left[\exp\left(\lambda  e^{J_T^{X}(f)}\right)\right]=\infty \quad \textrm{for all}\quad \lambda>0.
		\end{align*}
	
		If $T<\infty$ and $f\colon[0,T)\rightarrow\mathbb{R}$ is bounded and continuous, it follows from Proposition \ref{cor-lower-bound-finite-t} that
		\begin{align*}
			&\mathbb{E}\left[e^{-\lambda  I_T^{X}}\right]\\
			&\geq \int_0^{\exp\left(-\lambda e^{\mu_T^{X}-f_{\inf}}\int_0^{T}e^{\mathbb{E}[X_t]+f(t)}\,\mathrm{d}t\right)}P\left[I_T^{X}\leq -\frac{\log(x)}{\lambda} \right]\,\mathrm{d}x\\
			&\geq \int_0^{\exp\left(-\lambda e^{\mu_T^{X}-f_{\inf}}\int_0^{T}e^{\mathbb{E}[X_t]+f(t)}\,\mathrm{d}t\right)}\left( 1-\exp\left(-\frac{\left(\log\left( \frac{-\log\left(x\right)}{\lambda\int_0^{T}e^{\mathbb{E}[X_t]+f(t)}\,\mathrm{d}t}\right)+f_{\inf}-\mu_T^{X}\right)^{2}}{2\left(\sigma_T^{X}\right)^{2}}\right)\right)\,\mathrm{d}x.
			\end{align*}
		This completes the proof.
	\end{proof}

	\begin{rem}\emph{
		Suppose that there exists $f\in\mathcal{M}_T^{X}$ such that $s_T^{X}(f)>0$.
		As a result of Proposition \ref{general-ch-funct} (i), we have $\lim_{x\rightarrow\infty}e^{\lambda x}P\left[I_T^{\mu,\sigma,H}>x\right]=\infty$ for all $\lambda>0$, i.e., $I_T^{\mu,\sigma,H}$ has a heavy-tailed distribution.}
	\end{rem}

	\section{Exponential functionals of fBM }
	In this section we focus on the special case of fBM. Our aim is to provide explicit computable estimates for the c.d.f.\  and the $p$-th order moment of the exponential functional of a fBM with drift.
	
	Let $\mu\in\mathbb{R}$, $\sigma\in (0,\infty)$, $T\in (0,\infty]$ and $H\in (0,1]$. Let $B^{H}\coloneqq \{ B_t^{H}: t\geq 0\} $ be a fBM with Hurst parameter $H$, and denote by $B^{\mu,\sigma,H}\coloneqq \{ B_t^{\mu,\sigma,H}: t\geq 0\} $ the process given by $B^{\mu,\sigma,H}_t\coloneqq \mu t+\sigma B_t^{H}$ for each $t\geq 0$. Notice that $B^{\mu,\sigma,H}$ is a continuous Gaussian process with covariance function
	\begin{align*}
		\mathbb{E}\left[B^{\mu,\sigma,H}_sB^{\mu,\sigma,H}_t\right]=\frac{\sigma^{2}}{2}\left(s^{2H}+t^{2H}-|s-t|^{2H}\right)\geq 0,\quad s,t\geq 0,
	\end{align*} 
	and mean function $\mathbb{E}\big[B_t^{\mu,\sigma,H}\big]=\mu t$ for each $t\geq 0$. We write the exponential functional of $B^{\mu,\sigma,H}$ on $[0,T)$ as
	\begin{align*}
		I_{T}^{\mu,\sigma,H}\coloneqq I_T^{B^{\mu,\sigma,H}}= \int_0^T e^{B_t^{\mu,\sigma,H}}\,\mathrm{d}t.
	\end{align*}
	We distinguish the cases where $T<\infty$ and $T=\infty$.
	
	\subsection{Case $T\in(0,\infty)$}

	In order to apply Proposition \ref{general_bound} to estimate from above the c.d.f.\ of $I_T^{\mu,\sigma,H}$, we need to replace $\mathcal{M}_T^{B^{\mu,\sigma,H}}$ by a more suitable set of densities on $[0,T)$.
	One simple choice is the set $\{f_{0,T}\}$, where $f_{0,T}$ denotes the uniform density on $[0,T)$.
	This choice leads to the estimate in \eqref{finite_log} below.
	We can improve this estimate by taking a larger set of densities.
	We consider the set
	$\mathcal{N}_{T}\coloneqq\left\lbrace  f_{\lambda,T}: \lambda\in\mathbb{R}\right\rbrace$, where $f_{\lambda,T}\colon[0,T)\rightarrow (0,\infty)$ is the function given by
	\begin{align}\label{exponential-density}
		f_{\lambda,T}(t)\coloneqq\begin{cases}
			\frac{\lambda e^{\lambda t}}{e^{\lambda T}-1}&\textrm{if }\lambda\neq 0,\\
			\frac{1}{T}&\textrm{if }\lambda=0,
		\end{cases} 
	\end{align}
	for each $\lambda\in\mathbb{R}$. 
	Since $\mathcal{M}_T^{B^{\mu,\sigma,H}}$ contains all bounded continuous positive densities on $[0,T)$, in particular we have $\mathcal{N}_{T}\subset\mathcal{M}_T^{B^{\mu,\sigma,H}}$.
	For this family of densities we can obtain a collection of explicit upper estimates for the c.d.f.\ of $I_T^{\mu,\sigma,H}$.

	We proceed to compute $m_{T}^{\mu,\sigma,H}(\lambda)\coloneqq m_T^{B^{\mu,\sigma,H}}(f_{\lambda,T})$ and $s_T^{\mu,\sigma,H}(\lambda)\coloneqq s_T^{B^{\mu,\sigma,H}}(f_{\lambda,T})$. 
	For $\lambda\neq 0$, we have
	\begin{align*}
		m_{T}^{\mu,\sigma,H}(\lambda)&=\int_0^{T}\left(-\log\left(\frac{\lambda}{e^{\lambda T}-1}\right)-\lambda t+\mu t\right) \left(\frac{\lambda e^{\lambda t}}{e^{\lambda T}-1}\right)\,\mathrm{d}t\\
		&=\log\left(\frac{e^{\lambda T}-1}{\lambda}\right)+\frac{(\mu-\lambda)(\lambda Te^{\lambda T}-e^{\lambda T}+1)}{\lambda(e^{\lambda T}-1)},\\
		\left(s_T^{\mu,\sigma,H}(\lambda)\right)^{2}
		&=\int_0^{T}\int_0^{T}\frac{\sigma^{2}}{2}\left(s^{2H}+t^{2H}-|s-t|^{2H}\right)\left(\frac{\lambda e^{\lambda s}}{e^{\lambda T}-1}\right)\left(\frac{\lambda e^{\lambda t}}{e^{\lambda T}-1}\right)\,\mathrm{d}s\,\mathrm{d}t \\
		&=\frac{\lambda\sigma^{2}}{e^{\lambda T}-1}\int_0^{T}t^{2H}e^{\lambda t}\,\mathrm{d}t-\frac{\sigma^{2}}{2}\int_{-T}^{T}|t|^{2H}\left(\frac{\lambda}{2(e^{\lambda T}-1)^{2}}e^{\lambda |t|}\left(e^{2\lambda(T-|t|)}-1\right)\right)\,\mathrm{d}t\\	
		&=\frac{\lambda\sigma^{2}}{e^{\lambda T}-1}\int_0^{T}t^{2H}e^{\lambda t}\,\mathrm{d}t-\frac{\sigma^{2}}{2}\left(\frac{\lambda e^{2\lambda T}}{(e^{\lambda T}-1)^{2}}\int_0^{T}t^{2H}e^{-\lambda t}\,\mathrm{d}r+\frac{\lambda}{(e^{\lambda T}-1)^{2}}\int_0^{T}t^{2H}e^{\lambda t}\,\mathrm{d}t\right)\\	&=\frac{\lambda \sigma^{2}}{2\left(e^{\lambda T}-1\right)^{2}}\left(\left(2e^{\lambda T}-1\right)\int_{0}^{T}t^{2H}e^{\lambda t}\,\mathrm{d}t-e^{2\lambda T}\int_0^{T}t^{2H}e^{-\lambda t}\,\mathrm{d}t\right).
	\end{align*}
 	On the other hand, for the case $\lambda=0$ we have
	\begin{align*}
		m_T^{\mu,\sigma,H}(0)&=\int_0^{T}\left(-\log\left(\frac{1}{T}\right)+\mu t\right)\frac{1}{T} \,\mathrm{d}t=\log(T)+\frac{\mu T}{2},\\
		\left(s_T^{\mu,\sigma,H}(0)\right)^{2}&=\frac{\sigma^{2}}{T^{2}}\int_0^{T}\int_0^{T}\frac{1}{2}\left(s^{2H}+t^{2H}-|s-t|^{2H}\right)\,\mathrm{d}s\,\mathrm{d}t 
		=\frac{\sigma^{2}}{T^{2}}\left( \frac{T^{2H+2}}{2H+2}\right)=\frac{\sigma^{2}T^{2H}}{2H+2}.
	\end{align*}
	By applying directly Proposition \ref{general_bound}, we obtain the following upper bounds for the c.d.f.\ of $I_T^{\mu,\sigma,H}$.

	\begin{thm}\label{cor-upper-bound-finite}
		Let $x>0$. It holds that
		\begin{align*}
			P\left[I_T^{\mu,\sigma,H}\leq x\right]\leq  \Phi\left(\frac{\log(x)-m_{T}^{\mu,\sigma,H}(\lambda)}{s_T^{\mu,\sigma,H}(\lambda)}\right)
		\end{align*}
	for all $\lambda\in\mathbb{R}$. In particular, for $\lambda=0$ we have
	\begin{align}\label{finite_log}
			P\left[I_T^{\mu,\sigma,H}\leq x\right]\leq \Phi\left(\frac{\sqrt{2H+2}}{\sigma T^{H}}\left(\log(x)-\log(T)-\frac{\mu T}{2}\right)\right).
	\end{align}
	\end{thm}
	
		\begin{figure}[htbp]
		\centering
		\includegraphics[width=0.4\textwidth]{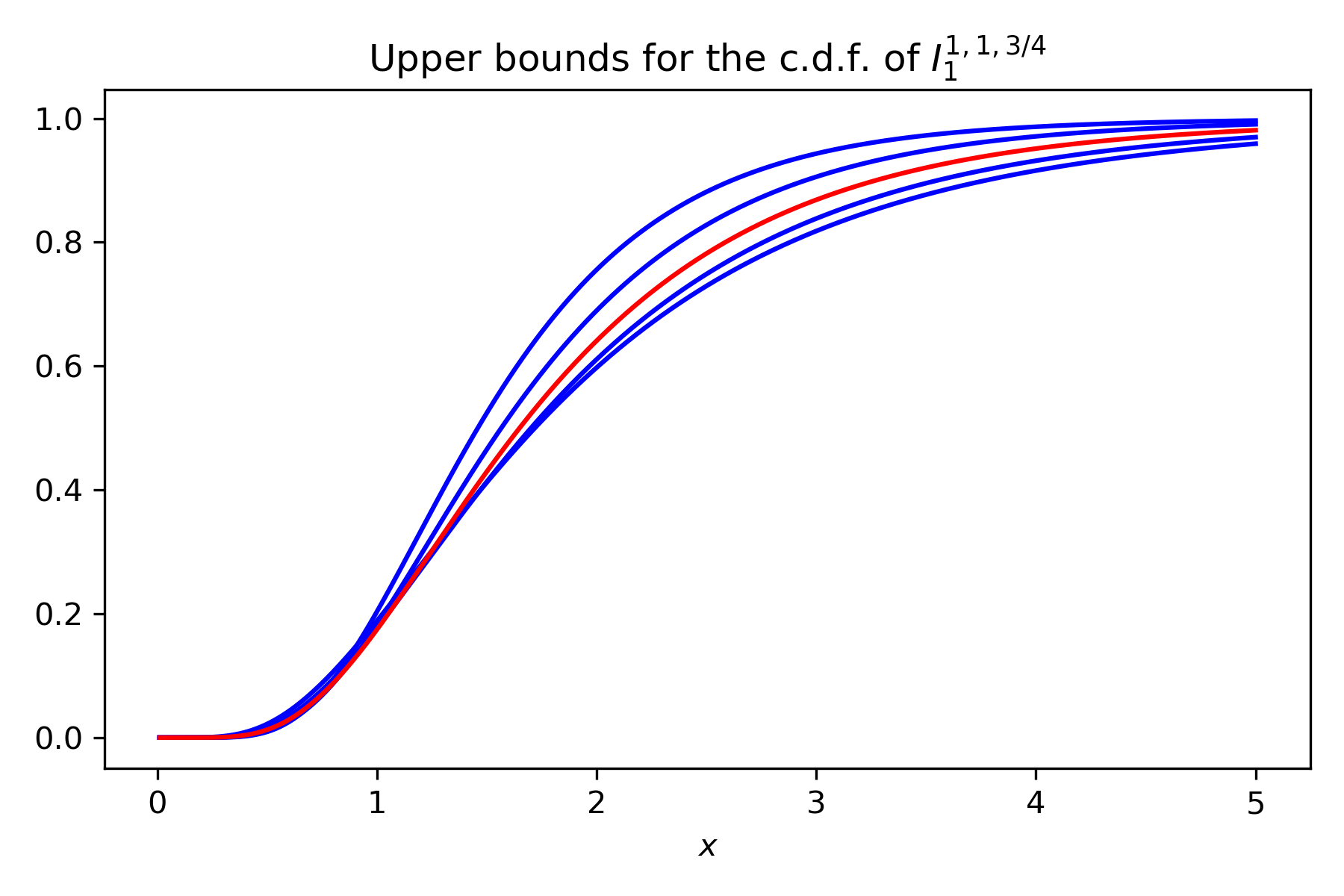}
		\caption{Upper bounds for the c.d.f.\ of $I_{1}^{1,1,3/4}$ obtained in Theorem \ref{cor-upper-bound-finite} with $\lambda=-1,0.5,0,0.5,1$. The red line corresponds to the case $\lambda=0$.}
		\label{figure:several_bounds}
	\end{figure}

	We appreciate in Figure \ref{figure:several_bounds} that the upper bounds obtained in Theorem \ref{cor-upper-bound-finite} improve the estimate in \eqref{finite_log} for some values of $\lambda$.
	The optimal value  of $\lambda$ may depend on $x,\mu,\sigma,H,T$.
	However, there remains as a challenge to obtain explicitly this value.
	One advantage of the estimate in (\ref{finite_log}) is that it has a tractable expression.	
	We have the following convergence result for this estimate.

	\begin{prop}\label{prop:limit}
		Suppose that $H\in (0,1)$. Let $\{\mu_n\}_{n\geq 1 }, \{\sigma_n\}_{n\geq 1}, \{T_n\}_{n\geq 1}$ be sequences that satisfy one of the following conditions:
		\begin{enumerate}
			\item[(i)] $\mu_n T_n\rightarrow 0$, $\sigma_n T_n^{H}\rightarrow 0$, and $T_n\rightarrow L\in [0,\infty)$;
			\item [(ii)] $\mu_n T_n\rightarrow \infty$, $\{\sigma_n T_n^{H}\}_{n\geq 1}$ is bounded,  and $\{T_n\}_{n\geq 1}$ is bounded away from zero;
			\item[(iii)] $\mu_n T_n\rightarrow -\infty$, and $\{\sigma_n T_n^{H}\}_{n\geq 1},\{T_n\}_{n\geq 1}$ are bounded.
		\end{enumerate}
	Then 
	\begin{align*}
		\lim_{n\rightarrow\infty} \Phi\left(\frac{\sqrt{2H+2}}{\sigma_n T_n^{H}}\left(\log(x)-\log(T_n)-\frac{\mu_n T_n}{2}\right)\right)-P\left[I_{T_n}^{\mu_n,\sigma_n,H}\leq x\right]=0
	\end{align*}
	for all $x>0$, $x\neq L$.
	\end{prop}

	\begin{proof}
	Let
	\begin{align*}
		J_{T_n}^{\mu_n,\sigma_n,H}(0)\coloneqq J_{T_n}^{B^{\mu_n,\sigma_n,H}}(f_{0,T_n})=\int_0^{T_n}\left(\log(T_n)+\mu_n t+\sigma_n B_t^{H}\right)\frac{1}{T_n}\,\mathrm{d}t.
	\end{align*}
	By the self-similarity of $B^{H}$, we have
	\begin{align*}
		I_{T_n}^{\mu_n,\sigma_n,H}&=T_n\int_0^{1}e^{\mu_n (T_n s)+\sigma_n B_{T_n s}^{H}}\,\mathrm{d}s=T_n\int_0^{1}e^{(\mu_n T_n)s+\sigma_n T_n^{H} B_{s}^{H}}\,\mathrm{d}s=T_n I_1^{\mu_n T_n, \sigma_n T_n^{H},H},\\
		\exp\left(J_{T_n}^{\mu_n,\sigma_n,H}(0)\right)&=T_ne^{\frac{\mu_n T_n}{2}}e^{\sigma_n\int_0^{1} B_{T_n s}^{H}\,\mathrm{d}s}=T_ne^{\frac{\mu_n T_n}{2}}e^{\sigma_n T_n^{H}\int_0^{1}B_s^{H}\,\mathrm{d}s}=T_n e^{\frac{\mu_n T_n}{2}}\exp\left(J_1^{0,\sigma_n T_n^{H},H}(0)\right).
		\end{align*}
	We suppose first that condition (i) holds.	Using the continuity of the sample paths of $B^{H}$, it follows from the dominated convergence theorem  that $I_1^{\mu_n T_n, \sigma_n T_n^{H},H},e^{\frac{\mu_n T_n}{2}}e^{J_1^{0,\sigma_n T_n^{H},H}(0)}\rightarrow 1$ $P$-a.s.\  as $n$ tends to infinity. Then, $I_{T_n}^{\mu_n,\sigma_n,H},e^{J_{T_n}^{\mu_n,\sigma_n,H}(0)}\rightarrow L$ $P$-a.s.\  Consequently, we obtain
	\begin{align*}
		P\left[e^{J_{T_n}^{\mu_n,\sigma_n,H}(0)}\leq x\right]-P\left[I_{T_n}^{\mu_n,\sigma_n,H}\leq x\right]\rightarrow 0
	\end{align*}
	for any $x>0$, $x\neq L$. 
	Notice that we can only ensure that this limit holds for $x\neq L$ because $I_{T_n}^{\mu_n,\sigma_n,H},e^{J_{T_n}^{\mu_n,\sigma_n,H}(0)}$ converge in distribution to a constant limit.
	Since
	\begin{align*}
		P\left[e^{J_{T_n}^{\mu_n,\sigma_n,H}(0)}\leq x\right]=\Phi\left(\frac{\sqrt{2H+2}}{\sigma_n T_n^{H}}\left(\log(x)-\log(T_n)-\frac{\mu_n T_n}{2}\right)\right),
	\end{align*}
	 we deduce the desired limit. If condition (ii) holds, then $I_{T_n}^{\mu_n,\sigma_n,H},e^{J_{T_n}^{\mu_n,\sigma_n,H}(0)}\rightarrow \infty$ $P$-a.s.\  On other hand, if condition (iii) holds, then $I_{T_n}^{\mu_n,\sigma_n,H},e^{J_{T_n}^{\mu_n,\sigma_n,H}(0)}\rightarrow 0$ $P$-a.s.\  In both cases, the desired limit follows. 
	\end{proof}

	\begin{rem}
		\emph{
			It follows from Proposition \ref{prop:limit} (i) that the difference between $P\left[I_{T}^{\mu,\sigma,H}\leq x\right]$ and $ \Phi\left(\frac{\sqrt{2H+2}}{\sigma T^{H}}\left(\log(x)-\log(T)-\frac{\mu T}{2}\right)\right)$ is small for sufficiently small values of $\mu$ and $\sigma$.
			We deduce from Proposition \ref{prop:limit} (ii) and (iii) that this difference is small as well for sufficiently large values of $|\mu|$.
		}
	\end{rem}

	We now proceed to estimate from below the c.d.f.\ of $I_T^{\mu,\sigma,H}$ following the approach of Lemma \ref{prop-general-lower-bound}.
	We will denote by erfc the complementary error function, i.e., 
	$$
	\textrm{erfc}(z)\coloneqq \frac{2}{\sqrt{\pi}}\int_{z}^{\infty}e^{-y^{2}}\,\mathrm{d}y,\quad z\in\mathbb{R}.
	$$ 

	\begin{thm}\label{lower-bound-finite-ii}
		Let $f\colon[0,T)\rightarrow\mathbb{R}$ be a bounded continuous function. If $H\in\big(0,\frac{1}{2}\big)$, then
		\begin{align*}
			P\left[ I_T^{\mu,\sigma,H}\leq x\right]\geq 1-\exp\left(-\frac{\left(\log\left( \frac{x}{\int_0^{T}e^{\mu t+f(t)}\,\mathrm{d}t}\right)+f_{\inf}-3.75 \sigma T^{H}\sqrt{\frac{2\pi}{H(\log 2)^{3}}}\emph{\textrm{erfc}}\left(\sqrt{\frac{H\log 2}{2}}\right)\right)^{2}}{2\sigma^{2}T^{2H}}\right)
		\end{align*}
	for all $$x>\exp\left(3.75 \sigma T^{H}\sqrt{\frac{2\pi}{H(\log 2)^{3}}}\emph{\textrm{erfc}}\left(\sqrt{\frac{H\log 2}{2}}\right)-f_{\inf}\right)\int_0^{T}e^{\mu t+f(t)}\,\mathrm{d}t.$$
	 If $H\in \big[\frac{1}{2},1\big]$, then
		\begin{align*}
			&P\left[ I_T^{\mu,\sigma,H}\leq x\right]\\
			&\geq \Phi\left(\frac{\log\left(\frac{x}{\int_0^{T}e^{\mu t+\lambda t^{2H}}\,\mathrm{d}t}\right)+\lambda T^{2H}}{\sigma T^{H}}\right)-\left(\frac{x}{\int_0^{T}e^{\mu t+\lambda t^{2H}}\,\mathrm{d}t}\right)^{-\frac{2\lambda}{\sigma^{2}}}\Phi\left(\frac{\log\left(\frac{\int_0^{T}e^{\mu t+\lambda t^{2H}}\,\mathrm{d}t}{x}\right)+\lambda T^{2H}}{\sigma T^{H}}\right)
		\end{align*}
	for any  $\lambda\in\mathbb{R}$ and $x>\int_0^{T}e^{\mu t+\lambda t^{2H}}\,\mathrm{d}t$.
	\end{thm}

	\begin{proof}
		Let $H\in\big(0,\frac{1}{2}\big)$. By applying Proposition \ref{cor-lower-bound-finite-t}, we get
		\begin{align*}
			P\left[ I_T^{\mu,\sigma,H}\leq x\right]&\geq 1-\exp\left(-\frac{1}{2\left(\sigma_T^{B^{\mu,\sigma,H}}\right)^{2}}\left(\log\left( \frac{x}{\int_0^{T}e^{\mu t+f(t)}\,\mathrm{d}t}\right)+f_{\inf}-\mu_T^{B^{\mu,\sigma,H}}\right)^{2}\right)
		\end{align*}
		for all $x>\exp\left(\mu_T^{B^{\mu,\sigma,H}}-f_{\inf}\right)\int_0^{T}e^{\mu t+f(t)}\,\mathrm{d}t$.
		The desired estimate follows by noticing that 
		\begin{align*}
			\left(\sigma_T^{B^{\mu,\sigma,H}}\right)^{2}=\sup_{t\in[0,T]}\textrm{Var}(B_t^{\mu,\sigma,H})=\sigma^{2}T^{2H},
		\end{align*}
	and
		\begin{align*}
			\mu_T^{B^{\mu,\sigma,H}}&=\mathbb{E}\left[\sup_{t\in[0,T]}\left(B_t^{\mu,\sigma,H}- \mathbb{E}\left[B_t^{\mu,\sigma,H}\right]\right)\right]\\
			&=\sigma T^{H}	\mathbb{E}\left[\sup_{t\in[0,1]}B_t^{H}\right]\\
			&<3.75 \sigma T^{H}\sqrt{\frac{2\pi}{H(\log 2)^{3}}}\textrm{erfc}\left(\sqrt{\frac{H\log 2}{2}}\right),
		\end{align*} 
		where we have used the estimate for $\mathbb{E}\big[\sup_{t\in[0,1]}B_t^{H}\big]$ proved in Theorem 2.2 (ii) in \cite{boro}. 
		
		Assume now that $H\in\big[\frac{1}{2},1\big]$. Let $\mathcal{R}_T$ be the set of all functions $r_\lambda\colon[0,T)\rightarrow\mathbb{R}$ defined by $r_\lambda(t)\coloneqq \lambda t^{2H}$, $t\in [0,T)$, for each $\lambda\in\mathbb{R}$. It follows from Lemma \ref{prop-general-lower-bound}, substituting $\mathcal{S}_T^{B^{\mu,\sigma,H}}$ by $\mathcal{R}_T$,  that
		\begin{align*}
			P\left[I_T^{\mu,\sigma,H}\leq x\right]\geq \sup_{\lambda\in\mathbb{R}} P\left[\sup_{t\in[0,T]}\left(\sigma B_t^{H}-\lambda t^{2H}\right)\leq \log\left(\frac{x}{\int_0^{T}e^{\mu t+\lambda t^{2H}}\,\mathrm{d}t}\right)\right],\quad x>0.
		\end{align*}
		Fix $\lambda\in\mathbb{R}$. Let $B\coloneqq B^{1/2}$ be a standard Brownian motion. Since $\{\sigma B_t^{H}-\lambda t^{2H}: 0\leq t\leq T\}$ and $\{\sigma B_{t^{2H}}-\lambda t^{2H}: 0\leq t\leq T\}$ are  continuous Gaussian processes such that 
		\begin{align*}
			\mathbb{E}[\sigma B_t^{H}-\lambda t^{2H}]&=\mathbb{E}[\sigma B_{t^{2H}}-\lambda t^{2H}],\\
			\textrm{Var}(\sigma B_t^{H}-\lambda t^{2H})&=\textrm{Var}(\sigma B_{t^{2H}}-\lambda t^{2H}),\\
			\textrm{Cov}(\sigma B_s^{H}-\lambda s^{2H},\sigma B_t^{H}-\lambda t^{2H})&\geq \textrm{Cov}(\sigma B_{s^{2H}}-\lambda s^{2H},\sigma B_{t^{2H}}-\lambda t^{2H}),
		\end{align*}
	for all $0\leq s,t\leq T$, it follows from Slepian's lemma (see Theorem 4.1 in \cite{debicki}) that
		\begin{align*}
			P\left[\sup_{t\in[0,T]}\left(\sigma B_t^{H}-\lambda t^{2H}\right)\geq  x\right]\leq P\left[\sup_{t\in[0,T]}\left(\sigma B_{t^{2H}}-\lambda t^{2H}\right)\geq  x\right], \quad x>0.
		\end{align*}
	Therefore, 
		\begin{align*}
			&P\left[\sup_{t\in[0,T]}\left(\sigma B_t^{H}-\lambda t^{2H}\right)\leq \log\left(\frac{x}{\int_0^{T}e^{\mu t+\lambda t^{2H}}\,\mathrm{d}t}\right)\right]\\
			&\geq P\left[\sup_{t\in[0,T]}\left(\sigma B_{t^{2H}}-\lambda t^{2H}\right)\leq \log\left(\frac{x}{\int_0^{T}e^{\mu t+\lambda t^{2H}}\,\mathrm{d}t}\right)\right]\\
			&= P\left[\sup_{t\in\left[ 0,T^{2H}\right]}\left(\sigma B_{t}-\lambda t\right)\leq \log\left(\frac{x}{\int_0^{T}e^{\mu t+\lambda t^{2H}}\,\mathrm{d}t}\right)\right]\\
			&=\Phi\left(\frac{\log\left(\frac{x}{\int_0^{T}e^{\mu t+\lambda t^{2H}}\,\mathrm{d}t}\right)+\lambda T^{2H}}{\sigma T^{H}}\right)-\left(\frac{x}{\int_0^{T}e^{\mu t+\lambda t^{2H}}\,\mathrm{d}t}\right)^{-\frac{2\lambda}{\sigma^{2}}}\Phi\left(\frac{\log\left(\frac{\int_0^{T}e^{\mu t+\lambda t^{2H}}\,\mathrm{d}t}{x}\right)+\lambda T^{2H}}{\sigma T^{H}}\right)
		\end{align*}
	for any $x>\int_0^{T}e^{\mu t+\lambda t^{2H}}\,\mathrm{d}t$, where we have used the identity (see the last formula in \cite[pg.\ 250]{borodin})
	\begin{align*}
		P\left[\sup_{t\in\left[ 0,T\right]}\left(\sigma B_{t}+\lambda t\right)\leq y\right]=\Phi\left(\frac{y-\lambda T}{\sigma\sqrt{T}}\right)-e^{\frac{2\lambda y}{\sigma^{2}}}\Phi\left(-\frac{y+\lambda T}{\sigma\sqrt{T}}\right), \quad y>0, \quad\lambda\in\mathbb{R}.
	\end{align*}
	\end{proof}

	\begin{rem}\label{improved-bound}\emph{
		In Corollary 3.1 in \cite{dung}, the tail of $I_T^{\mu,\sigma,H}$ was estimated, obtaining
		\begin{align*}
		 P\left[ I_T^{\mu,\sigma,H}>x\right]\leq 2\exp\left(-\frac{1}{2\sigma^{2}T^{2H}}\left(\log\left(\frac{x}{\int_0^{T}e^{\mu t+\frac{1}{2}\sigma^{2}t^{2H}}\,\mathrm{d}t}\right)\right)^{2}\right), \quad x>\int_0^{T}e^{\mu t+\frac{1}{2}\sigma^{2}t^{2H}}\,\mathrm{d}t.
		\end{align*}
		This estimate is similar, as $x\rightarrow\infty$, to the bound obtained in Theorem \ref{lower-bound-finite-ii} for the case $H<1/2$.
		%Nevertheless, the bound proved in \cite{dung} is non-trivial for a larger interval of values of $x$.
		On the other hand, for the case $H\geq 1/2$, our bounds are sharper.
		Indeed, for $H\geq 1/2$,  from Theorem \ref{cor-upper-bound-finite} and Theorem \ref{lower-bound-finite-ii} we get 
		\begin{align*}
			\Phi\left(\frac{\sqrt{2H+2}}{\sigma T^{H}}\left(\log\left(\frac{T}{x}\right)+\frac{\mu T}{2}\right)\right)\leq P\left[ I_T^{\mu,\sigma,H}>x\right]\leq 2\Phi\left(\frac{1}{\sigma T^{H}}\log\left(\frac{e^{\mu T}-1}{\mu x}\right)\right)
		\end{align*}
	for all $x>\frac{e^{\mu T}-1}{\mu}$, where $\mu\neq 0$. This estimate is sharper than that obtained in \cite{dung} since
	\begin{align*}
		2\Phi\left(\frac{1}{\sigma T^{H}}\log\left(\frac{e^{\mu T}-1}{\mu x}\right)\right)
		&\leq  \frac{\sqrt{2}\sigma T^{H}}{\sqrt{\pi}\log\left(\frac{\mu x}{e^{\mu T}-1}\right)}\exp\left(-\frac{1}{2\sigma^{2}T^{2H}}\left(\log\left(\frac{\mu x}{e^{\mu T}-1}\right)\right)^{2}\right)
	\end{align*}
	for all $x>\frac{e^{\mu T}-1}{\mu}$, where $\mu\neq 0$, which decays faster as $x\rightarrow\infty$. 
	In the case $\mu=0$ we have the estimate
	\begin{align*}
		\Phi\left(\frac{\sqrt{2H+2}}{\sigma T^{H}}\log\left(\frac{T}{x}\right)\right)\leq P\left[ I_T^{0,\sigma,H}>x\right]\leq 2\Phi\left(\frac{1}{\sigma T^{H}}\log\left(\frac{T}{ x}\right)\right),\quad x>T.
	\end{align*}
	These estimates tell us that the tail of $I_T^{\mu,\sigma,H}$, with $H\geq 1/2$, decays as the tail of  the log-normal distribution.
	Equivalently, this suggest that the tail of $\log\left(I_T^{\mu,\sigma,H}\right)$ may be approximated by the tail of a normal distribution. 
	}\end{rem}

	We now estimate the moments of $I_T^{\mu,\sigma,H}$.
	
	\begin{cor}\label{corolario-momento-caso-finito}
		It holds that $I_T^{\mu,\sigma,H}\in L^{p}(\Omega)$ for all $p\geq 1$. Moreover, for any $p\geq 1$, 
		\begin{align*}
			T^{p}\exp\left(\frac{\mu p T}{2}  +\frac{p^{2}\sigma^{2}T^{2H}}{4H+4}\right)\leq
			 \mathbb{E}\left[\left( I_T^{\mu,\sigma,H}\right)^{p}\right]\leq 
				\begin{cases}
					 T^{p} e^{\frac{1}{2}p^{2}\sigma^{2}T^{2H}} &\textrm{if }\mu=0,\\
					\left(\frac{e^{\mu p T}-1}{\mu p}\right)T^{p-1} e^{\frac{1}{2}p^{2}\sigma^{2}T^{2H} }&\textrm{if }\mu> 0,\\
					\left(\frac{e^{\mu  T}-1}{\mu }\right)^{p} e^{\frac{1}{2}p^{2}\sigma^{2}T^{2H} }&\textrm{if }\mu< 0.\\
				\end{cases}
		\end{align*}
	\end{cor}
	
	\begin{proof}
		By Proposition \ref{moments_functional} (ii), we get
		\begin{align*}
				\mathbb{E}\left[\left( I_T^{\mu,\sigma,H}\right)^{p}\right]\geq  \exp\left(pm_T^{\mu,\sigma,H}(0)+\frac{1}{2}p^{2}\left(s_T^{\mu,\sigma,H}(0)\right)^{2} \right)= T^{p}\exp\left(\frac{\mu p T}{2}  +\frac{p^{2}\sigma^{2}T^{2H}}{4H+4}\right).
		\end{align*}
		If $\mu>0$, we have
		\begin{align*}
			\mathbb{E}\left[\left( I_T^{\mu,\sigma,H}\right)^{p}\right]\leq \int_0^{T}e^{\mu p t+\frac{1}{2}p^{2}\sigma^{2}t^{2H}}\left(f_{0,T}(t)\right)^{1-p}\,\mathrm{d}t\leq\left(\frac{e^{\mu p T}-1}{\mu p}\right)T^{p-1}e^{\frac{1}{2}p^{2}\sigma^{2}T^{2H} },
		\end{align*}
		and if $\mu< 0$ we obtain
			\begin{align*}
			\mathbb{E}\left[\left( I_T^{\mu,\sigma,H}\right)^{p}\right]\leq \int_0^{T}e^{\mu p t+\frac{1}{2}p^{2}\sigma^{2}t^{2H}}\left(f_{\mu,T}(t)\right)^{1-p}\,\mathrm{d}t\leq\left(\frac{e^{\mu T}-1}{\mu }\right)^{p}e^{\frac{1}{2}p^{2}\sigma^{2}T^{2H} }.
		\end{align*}
		  The case $\mu=0$ can be proved in a similar way. 
	\end{proof}

	\begin{rem}\emph{
		Let $x>0$. The error in (\ref{finite_log}) is upper bounded by
		\begin{align*}
		 &\Phi\left(\frac{\sqrt{2H+2}}{\sigma T^{H}}\left(\log(x)-\log(T)-\frac{\mu T}{2}\right)\right)-	P\left[I_T^{\mu,\sigma,H}\leq x\right]\\
		 &=P\left[e^{J_T^{\mu,\sigma,H}(0)}\leq x,I_T^{\mu,\sigma,H}>x\right]\\
		 &\leq P\left[e^{J_T^{\mu,\sigma,H}(0)}\leq x\right]\wedge P\left[I_T^{\mu,\sigma,H}>x\right]\\
		 &\leq \Phi\left(\frac{\sqrt{2H+2}}{\sigma T^{H}}\left(\log(x)-\log(T)-\frac{\mu T}{2}\right)\right)\wedge \frac{ \left(\frac{e^{\mu T}-1}{\mu}\right) e^{\frac{1}{2}\sigma^{2}T^{2H}}}{x}
		\end{align*}
	for any $\mu\neq 0$. If $\mu=0$, then the above bound remains true if we replace $\left(\frac{e^{\mu T}-1}{\mu}\right)$ by $T$.
	}\end{rem}

 	We have the following bounds for the  moment-generating function of $I_T^{\mu,\sigma,H}$ for $\mu\neq0$. The case $\mu=0$ can be derived  replacing $\left(\frac{e^{\mu T}-1}{\mu}\right)$ by $T$ in the corresponding expressions.
	
	\begin{cor}\label{cor-gmf-finite}
		Let $\lambda>0$. Then
		\begin{align*}
			\mathbb{E}\left[e^{-\lambda I_T^{\mu,\sigma,H}}\right]\leq  \inf_{\varepsilon\in(0,1)}\left\lbrace \varepsilon+(1-\varepsilon)\Phi\left(\frac{\sqrt{ 2H+2}}{\sigma T^{H}}\left(\log\left(\log\left(\frac{1}{\varepsilon}\right)\right)-\log(\lambda)-\log(T)-\frac{\mu T}{2}\right)\right)\right\rbrace.
		\end{align*}
		If $H\in\big(0,\frac{1}{2}\big)$ and $\mu\neq 0$, then
		\begin{align*}
		&	\mathbb{E}\left[e^{-\lambda I_T^{\mu,\sigma,H}}\right]\\
		&\geq \int_0^{\exp\left(-\lambda e^{\frac{16.3 \sigma T^{H}}{\sqrt{H}}}\left(\frac{e^{\mu T}-1}{\mu}\right)\right)}\left(1-\exp\left(-\frac{\left(\log\left(\log\left(\frac{1}{x}\right)\right)-\log\left(\lambda\left(\frac{e^{\mu T}-1}{\mu}\right)\right)-\frac{16.3\sigma T^{H}}{\sqrt{H}}\right)^{2}}{2\sigma^{2}T^{2H}}\right)\right)\,\mathrm{d}x.
		\end{align*}
		If $H\in\big[\frac{1}{2},1\big]$ and $\mu\neq 0$, then
		\begin{align*}
				\mathbb{E}\left[e^{-\lambda I_T^{\mu,\sigma,H}}\right]\geq \int_0^{\exp\left(-\lambda\left(\frac{e^{\mu T}-1}{\mu}\right)\right)}\left(2\Phi\left(\frac{\log\left(\log\left(\frac{1}{x}\right)\right)-\log(\lambda)-\log\left(\frac{e^{\mu T}-1}{\mu}\right)}{\sigma T^{H}}\right)-1\right)\,\mathrm{d}x.
		\end{align*}
	\end{cor}

	\subsection{The case $T=\infty$}
	Our first objective will be to estimate from above the c.d.f.\ of $I_\infty^{\mu,\sigma,H}$.
	Let
	$\mathcal{N}_\infty\coloneqq \left\lbrace f_{\lambda}\colon \lambda>0\right\rbrace$ be the set of all exponential densities $f_\lambda\colon [0,\infty)\rightarrow (0,\infty)$ defined by 
	\begin{align}\label{exponential-density-2}
		f_{\lambda}(t)\coloneqq\lambda e^{-\lambda t },\quad t\geq 0,
	\end{align}
 	for each $\lambda>0$. 
 	Notice that $\mathcal{N}_\infty\subset \mathcal{M}_\infty^{B^{\mu,\sigma,H}}$. 
 	We can obtain explicit upper bounds for the c.d.f.\ of $I_\infty^{\mu,\sigma,H}$ for this family of densities.
 	This is due to the fact that $\log(f_\lambda)$ and $\mathbb{E}[B_\cdot^{\mu,\sigma,H}]$ are both linear functions, which simplify many computations.
 	Other families of densities, such as Gamma densities or folded normal densities, lead to intractable expressions.
 	We denote 
	\begin{align*}
	m_\infty^{\mu,\sigma,H}(\lambda)\coloneqq m_\infty^{B^{\mu,\sigma,H}}(f_\lambda)\quad\textrm{and}\quad  s_{\infty}^{\mu,\sigma,H}(\lambda)\coloneqq s_{\infty}^{B^{\mu,\sigma,H}}(f_\lambda)\quad\textrm{for each}\quad \lambda> 0. 
	\end{align*}
	Fix $\lambda>0$. Let us compute $m_\infty^{\mu,\sigma,H}(\lambda)$ and $s_{\infty}^{\mu,\sigma,H}(\lambda)$. We have
	\begin{align*}
		m_\infty^{\mu,\sigma,H}(\lambda)
		&=\int_0^{\infty} \left( -\log(\lambda)+\lambda t+\mu t\right)(\lambda e^{-\lambda t})\,\mathrm{d}t\\
		&=-\log(\lambda)+1+\frac{\mu}{\lambda},\\
		\left(s_{\infty}^{\mu,\sigma,H}(\lambda)\right)^{2}
		&=\int_0^{\infty}\int_0^{\infty} \frac{\sigma^{2}}{2}\left(s^{2H}+t^{2H}-|s-t|^{2H}\right) (\lambda e^{-\lambda s})(\lambda e^{-\lambda t})\,\mathrm{d}s\,\mathrm{d}t\\
&=\lambda\sigma^{2} \int_0^{\infty}t^{2H} e^{-\lambda t}\,\mathrm{d}t-\frac{\sigma^{2}}{2}\int_{-\infty}^{\infty} |t|^{2H} \left(\frac{\lambda}{2} e^{-\lambda |t|}\right)\,\mathrm{d}t\\
&=\frac{\lambda\sigma^{2}}{2} \int_0^{\infty}t^{2H} e^{-\lambda t}\,\mathrm{d}t\\
&=\frac{\sigma^{2}}{2\lambda^{2H}}\Gamma(2H+1).
	\end{align*}
Thus, from  Proposition \ref{general_bound} it follows that 
	\begin{align}\label{infimum_infinity}
		P\left[I_{\infty}^{\mu,\sigma,H}\leq x\right]\leq \Phi\left(\inf_{\lambda\in (0,\infty)}\frac{1}{\sigma}\sqrt{\frac{2}{\Gamma(2H+1)}}\left(\lambda^{H}\log(x)+\lambda^{H}\log(\lambda)-\lambda^{H}-\mu\lambda^{H-1}\right)\right)
	\end{align}
	for all  $x>0$. 
	
	Let $W\colon [-1/e,\infty)\rightarrow [-1,\infty)$ be the principal  branch of the Lambert $W$ function, i.e., $W$ is  the inverse function of the mapping $x\mapsto xe^{x}$ restricted to the interval $[-1,\infty)$. By using elementary calculus, we can obtain an explicit expression for the infimum  in the right-hand side of estimate (\ref{infimum_infinity}). This is proved in the next theorem.
	
	\begin{thm}\label{bounds_infinity}
		Let $x>0$. The following statements hold:
		\begin{enumerate}
			\item[(i)] If $\mu>0$ and $H\in (0,1)$, then
			\begin{align*}
				P\left[I_\infty^{\mu,\sigma,H}=\infty\right]=1.
			\end{align*}
		\item [(ii)] If $\mu=0$ and $H\in (0,1)$, then
		\begin{align}\label{mu0}
			P\left[I_\infty^{\mu,\sigma,H}\leq x\right]\leq\Phi\left(-\sqrt{\frac{2}{\Gamma(2H+1)}}\cdot\frac{1}{\sigma H e^{1-H}x^{H}}\right).
		\end{align}
			\item [(iii)] If $\mu<0$ and $H\in (0,1)$, then
			\begin{align}
				P\left[I_\infty^{\mu,\sigma,H}\leq x\right]&\leq  \Phi\left(-\sqrt{\frac{2}{\Gamma(2H+1)}}\cdot \frac{\mu+\lambda_{\mu,x,H}}{\sigma H \lambda_{\mu,x,H}^{1-H}} \right)\label{mu_negative}\\
				&\leq \Phi\left(\frac{(-\mu)^{H}}{\sigma}\sqrt{\frac{2}{\Gamma(2H+1)}}\log(-\mu x)\right),\label{mu_negative2}
			\end{align}
		where $\lambda_{\mu,x,H}\coloneqq\frac{\left(1-\frac{1}{H}\right)\mu}{ W\left(\left(1-\frac{1}{H}\right)\mu x e^{-1+\frac{1}{H}}\right)}$.
			\item [(iv)] If $\mu\in\mathbb{R}$ and $H=1$, then
		\begin{align*}
			P\left[I_\infty^{\mu,\sigma,H}\leq x\right]\leq \Phi\left(-\frac{1}{\sigma}\left(\mu+\frac{1}{x}\right)\right).
		\end{align*}
		\end{enumerate}
	\end{thm}

	\begin{proof}
		Let $F\colon (0,\infty)\rightarrow\mathbb{R}$ be the function given by
		\begin{align*}
		F(\lambda)&\coloneqq\lambda^{H}\log(x)+\lambda^{H}\log(\lambda)-\lambda^{H}-\mu \lambda^{H-1}, \quad \lambda> 0.
		\end{align*}
	We have
		\begin{align}
		\lim_{\lambda\rightarrow\infty}F(\lambda)=\infty,\label{limit_infty}
		\end{align}
	and
	\begin{numcases}{\lim_{\lambda \downarrow 0}F(\lambda)=}
		-\infty & $\textrm{if } \mu>0,\, H\in (0,1),$ \label{i}
		\\
		0 & $\textrm{if } \mu=0,\, H\in (0,1),$ \label{ii}\\
		\infty & $\textrm{if } \mu<0,\, H\in (0,1),$\label{iii}\\
		-\mu & $\textrm{if } \mu\in\mathbb{R}, \,H=1.\label{iv}$
	\end{numcases}
	 Observe that  $F$ is differentiable and its derivative is given by
		$$
		F'(\lambda)=H\lambda^{H-1}\log(\lambda)+\lambda^{H-1}(H\log(x)+1-H)+(1-H)\mu\lambda^{H-2}, \quad \lambda>0.
		$$
		We set $A\coloneqq H$, $B\coloneqq H\log(x)+1-H$ and $C\coloneqq (1-H)\mu$. Notice that a number $\lambda^{\ast}>0$ satisfies  $	F'(\lambda^{\ast})=0$ if and only if
		\begin{align}\label{eq_log0}
			A\lambda^{\ast}\log(\lambda^{\ast})+B\lambda^\ast +C=0,
		\end{align}
	or, equivalently,
	\begin{align} \label{eq_log}
		-\frac{C}{A}e^{\frac{B}{A}}=ze^{z},
	\end{align}
	where $z\coloneqq
	 \log(\lambda^{\ast})+\frac{B}{A}$.   Let us first assume that $\mu<0$ and $H\in(0,1)$. Since $-\frac{C}{A}e^{\frac{B}{A}}> 0 $, it follows from (\ref{eq_log}) that $z=W\left(-\frac{C}{A}e^{\frac{B}{A}}\right)$. This implies that equation (\ref{eq_log0}) has a unique solution $\lambda^{\ast}$, and it is  given   by  
	 \begin{align*}
	 	\lambda^{\ast}=\exp\left(W\left(-\frac{C}{A}e^{\frac{B}{A}}\right)-\frac{B}{A}\right)=\frac{\exp\left( W\left(\left(1-\frac{1}{H}\right)\mu x e^{-1+\frac{1}{H}}\right)\right)}{xe^{-1+\frac{1}{H}}}=\frac{\left(1-\frac{1}{H}\right)\mu}{ W\left(\left(1-\frac{1}{H}\right)\mu x e^{-1+\frac{1}{H}}\right)},
	 \end{align*}
 		where we have used the identity $e^{W(y)}=\frac{y}{W(y)}$, which holds for any $y>0$. We deduce from (\ref{limit_infty}) and (\ref{iii}) that $F$ has a unique global minimum at $\lambda^{\ast}$. Therefore, we can rewrite (\ref{infimum_infinity}) as
 	\begin{align}\label{simple_bound}
 		P\left[ I_\infty^{\mu,\sigma,H}\leq x\right]\leq \Phi\left(\frac{1}{\sigma } \sqrt{\frac{2}{\Gamma(2H+1)}} F(\lambda^\ast)\right).
 	\end{align}
 	We notice that
 	\begin{align*}
 		F(\lambda^{\ast})&=(\lambda^{\ast})^{H-1}\left(\lambda^{\ast}\left(\log(x)+\log(\lambda^\ast)-1\right)-\mu\right)\\&
 		=(\lambda^{\ast})^{H-1}\left(\lambda^{\ast}\left(W\left(\left(1-\frac{1}{H}\right)\mu x e^{-1+\frac{1}{H}}\right)-\frac{1}{H}\right)-\mu\right)\\&
 		=-\frac{\mu+\lambda^{\ast}}{H(\lambda^{\ast})^{1-H}}.
 	\end{align*}
 	 Hence,  the inequality (\ref{simple_bound}) becomes (\ref{mu_negative}). The estimate (\ref{mu_negative2}) follows from 
	(\ref{simple_bound}), replacing $\lambda^{\ast}$ by $-\mu$. This proves statement (iii).
 	
 	Suppose now that $\mu=0$ and $H\in (0,1)$. Then we have $C=0$. This implies, as before, that  equation (\ref{eq_log0}) has a unique solution, and it is given by $\lambda^{\ast}=1/\big(xe^{-1+1/H}\big)$. Since 
 	\begin{align*}
 		F(\lambda^{\ast})=-\frac{e^{H-1}}{Hx^{H}}<0,
 	\end{align*} 
 	we deduce from (\ref{limit_infty}) and (\ref{ii}) that $F$ has a unique global minimum at $\lambda^{\ast}$. Part (ii) follows due to (\ref{infimum_infinity}). 
 	
 	Assume that  $\mu\in\mathbb{R}$ and $H=1$. Then the unique solution of (\ref{eq_log0}) is  $\lambda^{\ast}=\frac{1}{x}$. We observe that
 	\begin{align*}
 		F(\lambda^{\ast})=-\mu-\frac{1}{x}<-\mu.
 	\end{align*}
 	Thus, from (\ref{limit_infty}) and (\ref{iv}) we obtain that $F$ attains its minimum value at $\lambda^{\ast}$. This yields part (iv). 
 	
 	Finally, if $\mu>0$ and $H\in (0,1)$, due to  (\ref{infimum_infinity}) and (\ref{i}) we get that
 	$P\big[ I_\infty^{\mu,\sigma,H}
 	\leq x\big]= 0$ for any $x>0$. Therefore, the statement (i) holds.  
	\end{proof}

	\begin{rem}\label{finitud_functional}\emph{Suppose that $H\in (0,1)$. Theorem \ref{bounds_infinity} (i) implies that $P\big[I_\infty^{\mu,\sigma,H}<\infty\big]=0$ for any $\mu>0$. By using the law of the iterated logarithm for fBM, this result was proved in Lemma 1 in Dozzi, Kolkovska and López-Mimbela \cite{dozzi3}. Moreover, the authors proved that $P\big[I_\infty^{\mu,\sigma,H}<\infty\big]=1$ for any $\mu<0$. For the case $\mu=0$, by  letting $x\rightarrow\infty$ in the estimate (\ref{mu0}), we obtain that 
	\begin{align*}
		P\left[I_\infty^{0,\sigma,H}<\infty\right]\leq \frac{1}{2}.
	\end{align*}  
	}
	\end{rem}

	We have the following convergence result for estimate (\ref{mu_negative2}).
	
	\begin{prop}\label{convergence-result-infinite-case}
		Let $x>0$, and suppose that $H\in (0,1)$. Then
		\begin{align}\label{asimptotic_conv}
			\lim_{\mu\rightarrow-\infty} \Phi\left(\frac{(-\mu)^{H}}{\sigma}\sqrt{\frac{2}{\Gamma(2H+1)}}\log(-\mu x)\right)-P\left[I_\infty^{\mu,\sigma,H}\leq x\right]=0.
		\end{align}
	\end{prop}

	\begin{proof} Taking into account Remark \ref{finitud_functional}, we know that $I_\infty^{\mu,\sigma,H}<\infty$ $P$-a.s.\  for any $\mu<0$. Hence, we can apply the dominated convergence theorem to obtain that $I_\infty^{\mu,\sigma,H}\rightarrow 0$ $P$-a.s.\  as $\mu\rightarrow-\infty$. On the other hand, we recall that $J_\infty^{B^{\mu,\sigma,H}}(f_{-\mu})\stackrel{d}{=} m_\infty^{\mu,\sigma,H}(-\mu)+s_\infty^{\mu,\sigma,H}(-\mu) N$, where $N$ is a standard normal random variable, and 
	\begin{align*}
		P\left[ e^{J_\infty^{B^{\mu,\sigma,H}}(f_{-\mu})}\leq x\right]=\Phi\left(\frac{(-\mu)^{H}}{\sigma}\sqrt{\frac{2}{\Gamma(2H+1)}}\log(-\mu x)\right)
	\end{align*}
	for any $\mu<0$. Since $m_\infty^{\mu,\sigma,H}(-\mu)=-\log(-\mu)\rightarrow -\infty $ and $s_\infty^{\mu,\sigma,H}(-\mu)=\frac{\sigma}{(-\mu^{H})}\sqrt{\frac{\Gamma(2H+1)}{2}}\rightarrow 0$ as $\mu\rightarrow-\infty$, we obtain that $e^{J_\infty^{B^{\mu,\sigma,H}}(f_{-\mu})}$ converges weakly to zero as $\mu\rightarrow-\infty$. Therefore, $	P\left[ e^{J_\infty^{B^{\mu,\sigma,H}}(f_{-\mu})}\leq x\right]-P\left[I_\infty^{\mu,\sigma,H}\leq x\right]\rightarrow 0$ as $\mu\rightarrow -\infty$.
	\end{proof}

	We now proceed to estimate from below the c.d.f.\ of $I_\infty^{\mu,\sigma,H}$ in the case where  $\mu<0$.
	
	\begin{thm}\label{pro-lowe-bound-func-infinite}
		Assume that $\mu<0$.  Let
			\begin{align*}
			\Lambda_{\mu,\sigma,x,H}\coloneqq \left(\frac{-\mu}{W\left(\left(1-\frac{1}{H}\right)\mu x e^{-1+\frac{1}{H}}\right)}\right)^{\frac{H}{1-H}}\left(\frac{W\left(\left(1-\frac{1}{H}\right)\mu x e^{-1+\frac{1}{H}}\right)+1-\frac{1}{H}}{\sigma}\right)^{\frac{1}{1-H}},\quad x>-\frac{1}{\mu}.
		\end{align*}
	 If $H\in\big(0,\frac{1}{2}\big)$, then
		\begin{align*}
			P\left[ I_\infty^{\mu,\sigma,H}\leq x\right]\geq 1- \frac{\mathscr{M}(H)}{\Lambda_{\mu,\sigma,x,H}},\quad x>-\frac{1}{\mu},
		\end{align*}
	where $\mathscr{M}(H)\coloneqq \mathbb{E}\big[\sup_{t\geq 0}(-t+B_t^{H})\big]$. On the other hand, if $H\in \big(\frac{1}{2},1\big)$, then
		\begin{align*}
			P\left[ I_\infty^{\mu,\sigma,H}\leq x\right]\geq 1-l_{H}\left(\Lambda_{\mu,\sigma,x,H}\right),\quad x>-\frac{1}{\mu},
		\end{align*}
	where 	$l_H\colon (0,\infty)\rightarrow(0,\infty)$ is the function given by
	\begin{align*}
		l_H(x)\coloneqq \frac{1}{\sqrt{2\pi}}\int_0^{\infty}\frac{1}{t^{H}}\exp\left(-\frac{(t+x)^{2}}{2 t^{2H}}\right)\,\mathrm{d}t,\quad x>0.
	\end{align*}
	\end{thm}
	
	\begin{proof}
		Let $\mathcal{G}$ be the set of all functions $g_\lambda\colon[0,\infty)\rightarrow\mathbb{R}$ given by $g_\lambda(t)=-\lambda t$, $t\geq 0$, for each $\lambda\in(\mu,0)$. It follows from Lemma \ref{prop-general-lower-bound}, substituting $\mathcal{S}_\infty^{B^{\mu,\sigma,H}}$ by $\mathcal{G}$, that
		\begin{align*}
				P\left[ I_\infty^{\mu,\sigma,H}\leq x\right]\geq	\sup_{\lambda\in(\mu,0)}P\left[\sup_{t\geq 0} \left(\lambda t+\sigma B_t^{H}\right) \leq \log\left(x(\lambda-\mu)\right)\right].
		\end{align*}
		Fix  $\lambda\in(\mu,0)$. By the self-similarity of $B^{H}$, we have (see \cite[pg. 2]{bise})
	\begin{align*}
		\sup_{t\geq 0}\left( \lambda t+\sigma B_t^{H}\right)\stackrel{d}{=}\frac{\sigma^{\frac{1}{1-H}}}{(-\lambda)^{\frac{H}{1-H}}} \sup_{t\geq 0} \left(-t+B_t^{H}\right).
	\end{align*}
	Hence,
		\begin{align*}
			P\left[\sup_{t\geq 0}\left( \lambda t+\sigma B_t^{H}\right) \leq \log\left(x(\lambda-\mu)\right)\right]=	P\left[\sup_{t\geq 0} \left(-t+B_t^{H}\right)\leq \frac{(-\lambda)^{\frac{H}{1-H}}\log\left(x(\lambda-\mu)\right)}{\sigma^{\frac{1}{1-H}}} \right].
		\end{align*}
		Let $F\colon(\mu,0)\rightarrow\mathbb{R}$ be the function given by
		\begin{align*}
			F(\lambda)\coloneqq  \frac{(-\lambda)^{\frac{H}{1-H}}\log\left(x(\lambda-\mu)\right)}{\sigma^{\frac{1}{1-H}}}, \quad \mu<\lambda<0.
		\end{align*}
		 We notice that $F'(\lambda^{\ast})=0$ if and only if
		\begin{align*}
			\left(\frac{H}{1-H}\right) z \log(z)+z+\mu x=0,
		\end{align*}
		where $z\coloneqq x(\lambda^{\ast}-\mu)$.  Proceeding as in the proof of Theorem \ref{bounds_infinity},  we obtain that the above equation has a unique solution, and it is given by
		$	z=\exp\left(W\left(\left(1-\frac{1}{H}\right)\mu x e^{-1+\frac{1}{H}}\right)+1-\frac{1}{H}\right)\geq 1$. Since $\lim_{\lambda\downarrow \mu}F(\lambda)=-\infty$ and $\lim_{\lambda\uparrow 0}F(\lambda)=0$,  we deduce that $F$ attains its maximum at
		\begin{align*}
			\lambda^{\ast}=\mu+\frac{z}{x}=\mu+\frac{\left(1-\frac{1}{H}\right)\mu}{W\left(\left(1-\frac{1}{H}\right)\mu x e^{-1+\frac{1}{H}}\right)}\in(\mu,0).
		\end{align*}
		Thus, if $H\in\big(0,\frac{1}{2}\big)$, we apply Markov's inequality and get
		\begin{align*}
			P\left[ I_\infty^{\mu,\sigma,H}\leq x\right]\geq	P\left[\sup_{t\geq 0} \left(-t+B_t^{H}\right)\leq F(\lambda^{\ast}) \right]\geq 1- \frac{\mathscr{M}(H)}{F(\lambda^{\ast})}.
		\end{align*}
		On the other hand, if $H\in \big(\frac{1}{2},1\big)$, we deduce from Slepian's lemma that
		\begin{align*}
			P\left[\sup_{t\geq 0} \left(-t+B_t^{H}\right)\leq F(\lambda^{\ast}) \right]\geq P\left[\sup_{t\geq 0} \left(-t+B_{t^{2H}}\right)\leq F(\lambda^{\ast}) \right],
		\end{align*}
		where $B$ is a standard Brownian motion. Using Theorem 3.1 in D\k{e}bicki, Michna and Rolski \cite{debicki}, we conclude that
	\begin{align*}
		P\left[ I_\infty^{\mu,\sigma,H}\leq x\right]\geq P\left[\sup_{t\geq 0} \left(-t+B_{t^{2H}}\right)\leq F(\lambda^{\ast}) \right]\geq 1-l_H( F(\lambda^{\ast})).
	\end{align*}	
	\end{proof}

	\begin{rem}\emph{Suppose that $\mu<0$ and $H\in \big(0,\frac{1}{2}\big)$. An estimate for the expected value $\mathscr{M}(H)$ was obtained in Proposition 5 in \cite{bise}. This result implies that $\mathscr{M}(H)$ is upper bounded by 
		\begin{align*}
		\left(\left(2(2H)^{\frac{H}{1-H}}(1-2H)^{\frac{1-2H}{2-2H}}\right)\wedge\left(H ^{H}(1-H)^{1-H}\right)\right)\left(\lambda_H^{\frac{1}{1-H}}+\frac{\sqrt{\pi/2}}{1-H}\left(\lambda_H^{\frac{H}{1-H}}+\frac{2^{\frac{H}{2-2H}}}{\sqrt{\pi}}\Gamma\left(\frac{1}{2-2H}\right)\right)\right),
		\end{align*}
	where $\lambda_H\coloneqq \frac{1.695}{\sqrt{2/\log_2\lceil  2^{2/H}\rceil}}$. Using this estimate, we can  obtain an explicit lower bound for the c.d.f.\ of $I_\infty^{\mu,\sigma,H}$ in Theorem \ref{pro-lowe-bound-func-infinite}. Also, since $\mathbb{E}[I_\infty^{\mu,\sigma,H}]<\infty$ (see Corollary \ref{cor-moment-func-infinte} below), applying Markov's inequality we get the estimate
	\begin{align*}
		P\left[I_\infty^{\mu,\sigma,H}\leq x\right]\geq 1-\frac{1}{x}\mathbb{E}\left[I_\infty^{\mu,\sigma,H}\right]\geq 1+\frac{2}{\mu x}\exp\left(\left(\frac{1}{2}-H\right)\left(\left(-\frac{2H}{\mu}\right)^{H}\sigma\right)^{\frac{2}{1-2H}}\right), \quad x>0.
	\end{align*}
	}\end{rem}
	
	\begin{rem} \emph{
	Our method to obtain lower bounds for the c.d.f.\ of exponential functionals does not apply for $I_\infty^{\mu,\sigma,H}$ in the case $\mu=0$.
	The law of $I_\infty^{0,\sigma,H}$ is only known for $H=\frac{1}{2}$ and $H=1$.
	When $H=\frac{1}{2}$, we have $I_\infty^{0,\sigma,H}=\infty$ $P$-a.s. On the other hand, it follows from  Proposition \ref{case_infty_1} that $P\left[I_\infty^{0,\sigma,1}\leq x\right]=\Phi\left(-\frac{1}{\sigma x}\right)\textbf{1}_{\{x>0\}}$.  In particular, $P[I_\infty^{0,\sigma,1}=\infty]= \frac{1}{2}$.
	}\end{rem}

	We estimate the moments of $I_\infty^{\mu,\sigma,H}$ in the following corollary.
	
	\begin{cor}\label{cor-moment-func-infinte}
		 The following statements hold:
		\begin{enumerate}
			\item[(i)] If $\mu\geq 0$ and $H\in (0,1]$, then $\mathbb{E}\left[\left(I_\infty^{\mu,\sigma,H}\right)^{p}\right]=\infty$ for all $p>0$.
			\item[(ii)] If $\mu<0$ and $H>\frac{1}{2}$, then $\mathbb{E}\left[\left(I_\infty^{\mu,\sigma,H}\right)^{p}\right]=\infty$ for all $p>0$.
			\item[(iii)] If $\mu<0$ and $H=\frac{1}{2}$, then $I_\infty^{\mu,\sigma,H}\in L^{p}(\Omega)$  if and only if $p<-2\mu/\sigma^{2}$, and 
			\begin{align*}
				\mathbb{E}\left[\left(I_\infty^{\mu,\sigma,H}\right)^{p}\right]=\frac{\Gamma\left(-\frac{2\mu}{\sigma^{2}}-p\right)}{\Gamma\left(-\frac{2\mu}{\sigma^{2}}\right)}\left(\frac{2}{\sigma^{2}}\right)^{p}, \quad 0< p<-\frac{2\mu}{\sigma^{2}}.
			\end{align*}
			\item[(iv)] If $\mu<0$ and $H<\frac{1}{2}$, then $I_\infty^{\mu,\sigma,H}\in L^{p}(\Omega)$ for all $p\geq 1$, and
			\begin{align*}
				\mathbb{E}\left[\left(I_\infty^{\mu,\sigma,H}\right)^{p}\right]\leq  \frac{2}{(-\mu)^{p}}\exp\left(\left(\frac{1}{2}-H\right)\left(\left(-\frac{2H}{\mu}\right)^{H}p\sigma\right)^{\frac{2}{1-2H}}\right),\quad p\geq 1.
			\end{align*}
			Also,
			\begin{align*}
				\mathbb{E}\left[\left(I_\infty^{\mu,\sigma,H}\right)^{p}\right]\geq \sup_{\lambda>0}\frac{1}{\lambda^{p}}\exp\left(p \left(1+\frac{\mu}{\lambda}\right)+\frac{ p^{2}\sigma^2\Gamma(2H+1)}{4\lambda^{2H}} \right), \quad p>0,
			\end{align*}
		where the supremum is   attained at the unique solution of the equation 
		\begin{align*}
			\lambda+\frac{p\sigma^{2}H\Gamma(2H+1)}{2}\lambda^{1-2H}+\mu=0.
		\end{align*}
		\end{enumerate}
	\end{cor}
	\begin{proof} 
		Assertion (i) follows from Theorem \ref{bounds_infinity} (i) and (ii). If $\mu<0$ and $H>\frac{1}{2}$, we apply Proposition \ref{moments_functional} (i) and obtain
		\begin{align*}
				\mathbb{E}\left[\left(I_\infty^{\mu,\sigma,H}\right)^{p}\right]
				\geq \int_0^{\infty} e^{\mu p t+\frac{1}{2}p^{2}\sigma^{2}t^{2H}}\left(f_{-\mu}(t)\right)^{1-p}\,\mathrm{d}t=(-\mu)^{1-p}\int_0^{\infty}e^{\mu t+\frac{1}{2}p^{2}\sigma^{2}t^{2H}}\,\mathrm{d}t=\infty
		\end{align*}
		for any $p\in (0,1)$. This proves  part (ii). The statement in (iii) follows from (\ref{func_exp_brow}). Now suppose that $\mu<0$ and $H<\frac{1}{2}$. Using Proposition \ref{moments_functional} (ii) we get
		\begin{align*}
			\mathbb{E}\left[\left(I_\infty^{\mu,\sigma,H}\right)^{p}\right]&\leq  \int_0^{\infty} e^{\mu p t+\frac{1}{2}p^{2}\sigma^{2}t^{2H}}\left(f_{-\mu}(t)\right)^{1-p}\,\mathrm{d}t\\
			&=(-\mu)^{1-p}\int_0^{\infty}e^{\mu t+\frac{1}{2}p^{2}\sigma^{2}t^{2H}}\,\mathrm{d}t\\
			&\leq \frac{1}{(-\mu)^{p-1}}\left(\int_0^{\infty}e^{\frac{1}{2}\mu t}\,\mathrm{d}t\right)\left(\sup_{t\geq 0}e^{\frac{1}{2}\mu t+\frac{1}{2}p^{2}\sigma^{2}t^{2H}}\right)\\
			&=  \frac{2}{(-\mu)^{p}}\exp\left(\left(\frac{1}{2}-H\right)\left(\left(-\frac{2H}{\mu}\right)^{H}p\sigma\right)^{\frac{2}{1-2H}}\right)
		\end{align*}
		for any $p\geq 1$. Hence, $I_\infty^{\mu,\sigma,H}\in L^{p}(\Omega)$ for all $p\geq 1$. We deduce from Proposition \ref{moments_functional} (ii) that 
		\begin{align*}
			\mathbb{E}\left[\left(I_\infty^{\mu,\sigma,H}\right)^{p}\right]\geq \sup_{\lambda>0}\exp\left(p \left(-\log(\lambda)+1+\frac{\mu}{\lambda}\right)+\frac{ p^{2}\sigma^2\Gamma(2H+1)}{4\lambda^{2H}} \right)=\sup_{\lambda>0}e^{p F(\lambda)}
		\end{align*}
		for any $p>0$, where $F(\lambda)\coloneqq-\log(\lambda)+1+\frac{\mu}{\lambda}+\frac{ p\sigma^2\Gamma(2H+1)}{4\lambda^{2H}}$ for each $\lambda>0$. Notice that $F(\lambda)\rightarrow-\infty$ if  $\lambda\rightarrow 0$ or $\lambda\rightarrow \infty$. Since $F$ is differentiable on $(0,\infty)$, $F$ has a unique global maximum at $\lambda^\ast$, the unique number that satisfies $F'(\lambda^{\ast})=0$, i.e., $\lambda^{\ast}+\frac{p\sigma^{2}H\Gamma(2H+1)}{2}(\lambda^{\ast})^{1-2H}+\mu=0$.
	\end{proof}

	\begin{rem}\emph{
		Let  $x>0$. Suppose that $\mu<0$ and $H<\frac{1}{2}$. The error in the estimate (\ref{mu_negative2})  is upper bounded by
		\begin{align*}
			&\Phi\left(\frac{(-\mu)^{H}}{\sigma}\sqrt{\frac{2}{\Gamma(2H+1)}}\log(-\mu x)\right)-P\left[I_\infty^{\mu,\sigma,H}\leq x\right]\\
			&=P\left[e^{J_\infty^{B^{\mu,\sigma,H}}(f_{-\mu})}\leq x, I_\infty^{\mu,\sigma,H}> x\right]\\
			&\leq P\left[e^{J_\infty^{B^{\mu,\sigma,H}}(f_{-\mu})}\leq x\right]\wedge P\left[I_\infty^{\mu,\sigma,H}> x\right]\\
			&\leq \Phi\left(\frac{(-\mu)^{H}}{\sigma}\sqrt{\frac{2}{\Gamma(2H+1)}}\log(-\mu x)\right)\wedge  \frac{2\exp\left(\left(\frac{1}{2}-H\right)\left(\left(-\frac{2H}{\mu}\right)^{H}\sigma\right)^{\frac{2}{1-2H}}\right)}{-\mu x}.
		\end{align*}
	Notice that this error is small for a sufficiently large negative value of $\mu$.
	}\end{rem}

	Using the lower and upper bounds for the c.d.f.\ of $I_\infty^{\mu,\sigma,H}$ proved above, we can estimate its moment-generating function. The upper bound has the following form.
	
	\begin{cor} \label{cor-gmf-infinite}
	Let $\lambda>0$. Then
	\begin{align*}
		&\mathbb{E}\left[e^{-\lambda  I_\infty^{\mu,\sigma,H}}\right]\\
		&\leq \inf_{\substack{\varepsilon\in(0,1),\\ \delta>0}} \left\lbrace \varepsilon+(1-\varepsilon)\Phi\left(\frac{\delta^{H}}{\sigma}\sqrt{\frac{2}{\Gamma(2H+1)}}\left(\log\left(\log\left(\frac{1}{\varepsilon}\right)\right)-\log(\lambda)+\log(\delta)-1
		-\frac{\mu}{\delta}\right)\right)\right\rbrace.
	\end{align*} 
	\end{cor}

	\subsection{The case $H=1$}
	
	In this subsection, we compute the exact c.d.f.\ of $I_T^{\mu,\sigma,H}$ in the case where  $H=1$.
	
	Let $W_{-1}\colon [-1/e,0)\rightarrow (-\infty,-1]$ be the lower branch of the Lambert $W$ function, i.e., $W_{-1}$ is  the inverse function of the mapping $x\mapsto xe^{x}$ restricted to the interval  $(-\infty,-1]$. We consider first the case $T<\infty$.
	
	\begin{prop}\label{pro-dist-cas-1}
		Suppose that $T<\infty$. For any $x>0$, we have
		\begin{align*}
			P\left[I_T^{\mu,\sigma,1}\leq x\right]
			=\Phi\left(-\frac{1}{\sigma}\left(\mu+\frac{1}{x}+\frac{1}{T}W_{-1}\left(-\frac{T}{x}e^{-\frac{T}{x}}\right)\right)\right).
		\end{align*}
	\end{prop}

	\begin{proof}
		Let $N$  be a standard normal random variable.  Since $B^{1}\stackrel{d}{=}\left\lbrace tN: t\geq 0\right\rbrace $,
		we have
		\begin{align*}
			P\left[I_T^{\mu,\sigma,H}\leq x\right]
			=	P\left[\int_0^{T} e^{t(\mu +\sigma  N)}\,\mathrm{d}t\leq x\right]
			=	P\left[F(\mu+\sigma N)\leq x\right],\quad x>0,
		\end{align*}
		where $F:\mathbb{R}\rightarrow \mathbb{R}$ is the function given by
		\begin{equation*}
			F(y)\coloneqq\begin{cases}
				\frac{e^{Ty}-1}{y} &\textrm{if }y\neq 0,\\
				T &\textrm{if }y=0.
			\end{cases}
		\end{equation*}
		It is easy to see that $F$ is continuous and strictly increasing. Therefore,
		\begin{align*}
			\left\lbrace F(\mu+\sigma N)\leq x\right\rbrace =\left\lbrace  \mu +\sigma N\leq F^{-1}(x)\right\rbrace.
		\end{align*} Let $u=F^{-1}(x)$. Then $	-\frac{T}{x}e^{-\frac{T}{x}}= z e^{z},$
		were $z\coloneqq- Tu-\frac{T}{x}$.
		We notice that $-\frac{T}{x}e^{-\frac{T}{x}}\in [-1/e,0)$ and $z\in (-\infty,-1]$, which implies that
		$z=W_{-1}\left(-\frac{T}{x}e^{-\frac{T}{x}}\right)$.  Thus,
		\begin{align*}
			P\left[\mu+\sigma N\leq F^{-1}(x)\right]
			&=	P\left[\mu+\sigma N\leq -\frac{1}{x}-\frac{1}{T}W_{-1}\left(-\frac{T}{x}e^{-\frac{T}{x}}\right)\right]\\
			&=  \Phi\left(-\frac{1}{\sigma}\left(\mu+\frac{1}{x}+\frac{1}{T}W_{-1}\left(-\frac{T}{x}e^{-\frac{T}{x}}\right)\right)\right).
		\end{align*}
	\end{proof}

	We now deal with the case $T=\infty$.
	
	\begin{prop}\label{case_infty_1}
		Let $x>0$. Then
		\begin{align*}
			P\left[I_\infty^{\mu,\sigma,1}\leq x\right]=\Phi\left(-\frac{1}{\sigma}\left(\mu+\frac1{x}\right)\right).
		\end{align*}
	\end{prop}
	
	\begin{proof}
		Let $N$ be a standard normal variable. The result follows by a straightforward calculation:
		\begin{align*}
			P\left[ I_\infty^{\mu,\sigma,1}\leq x\right]
			&=	P\left[\int_0^{\infty} e^{\mu t+\sigma t N}\,\mathrm{d}t\leq x\right]\\
			&=	P\left[\mu+\sigma N\geq  0,\int_0^{\infty} e^{t(\mu +\sigma N) }\,\mathrm{d}t\leq x,\right]+P\left[\mu +\sigma N< 0,\int_0^{\infty} e^{t(\mu +\sigma N) }\,\mathrm{d}t\leq x\right]\\
			&=P\left[\mu+\sigma N<0,-\frac{1}{\mu+\sigma N}\leq x\right]\\
			&= P\left[ N< -\frac{\mu}{\sigma},\, N\leq -\frac{1}{\sigma}\left(\mu+\frac{1}{x}\right)\right]\\
			&= \Phi\left(-\frac{1}{\sigma}\left(\mu+\frac{1}{x}\right)\right).
		\end{align*}
	\end{proof}

	\begin{rem} \emph{Notice that the estimate in Theorem \ref{bounds_infinity} (iv) coincides with the exact c.d.f.\ of $I_\infty^{\mu,\sigma,1}$.
	Let $N$ be a standard normal variable.			 
	Since
			\begin{align*}
				\Phi\left(-\frac{1}{\sigma}\left(\mu+\frac{1}{x}\right)\right)=P\left[-\mu-\sigma N\geq \frac{1}{x}\right]=P\left[\frac{1}{(-\mu-\sigma N)\textbf{1}_{\{-\mu-\sigma N>0\}}}\leq x\right], \quad x>0,
			\end{align*}
			it follows that $I_\infty^{\mu,\sigma,1}\stackrel{d}{=}\frac{1}{Y\textbf{1}_{\{Y>0\}}}$, where $Y$ has normal distribution with mean $-\mu$ and 		variance $\sigma^{2}$. In particular, we have
			\begin{align*}
				P\left[I_\infty^{\mu,\sigma,1}= \infty\right]
				=	P\left[Y\leq 0\right]=1-\Phi\left(-\frac{\mu}{\sigma}\right).
			\end{align*}
		}
	\end{rem}

	\subsection{The case $H=1/2$}
	
	The law of $I_T^{\mu,\sigma,H}$ is well known in the case where $H=1/2$, see \cite{duf,yor}.
	However, we do not recover these results with our approach. 
	This is so because our approach  relies solely  on the Gaussianity of fBM.
	We remark that the results known for $I_T^{\mu,\sigma,1/2}$ have been obtained by exploiting nice properties of standard Brownian motion, such as the Markov property and the independent increments property.
	Since fBM does not satisfy such properties, the techniques used to investigate $I_T^{\mu,\sigma,1/2}$ can not be extrapolated to exponential functionals of fBM.
	Nevertheless, our approach provides log-normal type approximations to the c.d.f.\ of $I_T^{\mu,\sigma,H}$ for any Hurst parameter $H$, including $H=1/2$.

	In the case $T=\infty$, the functional $I_T^{\mu,\sigma,1/2}$ has inverse Gamma distribution and its c.d.f.\ is given by (\ref{func_exp_brow}) for $\mu<0$ and $\sigma>0$.
	It follows from Theorem \ref{bounds_infinity} (iii) that
	\begin{align*}
		P\left[I_\infty^{\mu,\sigma,1/2}\leq x\right]\leq  \Phi\left( \frac{2\sqrt{-2\mu }}{\sigma}\left(\sqrt{W(-\mu x e)}-\frac{1}{\sqrt{W(-\mu x e)}}\right)\right)
		\leq \Phi\left(\frac{\sqrt{-2\mu}}{\sigma}\log(-\mu x)\right)
	\end{align*}
	for $\mu<0$ and $\sigma,x>0$.
	The error in these estimates can be written as
	\begin{align*}
	&\Phi\left(\frac{\sqrt{-2\mu}}{\sigma}\log(-\mu x)\right)-	P\left[I_\infty^{\mu,\sigma,1/2}\leq x\right]\\
	&=\int_0^{x}\left(\frac{1}{\sigma y}\sqrt{\frac{-\mu}{\pi}}\exp\left(\frac{\mu}{\sigma^{2}}\left(\log(-\mu y)\right)^{2}\right)- \frac{1}{\Gamma\left(-\frac{2\mu}{\sigma^{2}}\right) y}\left(\frac{2}{\sigma^{2}y}\right)^{-\frac{2\mu}{\sigma^{2}}} \exp\left(-\frac{2}{\sigma^{2}y}\right) \right) \,\mathrm{d}y.
	\end{align*}
	This expression corresponds to the difference between a log-normal and an inverse Gamma c.d.f. 
	To the best of our knowledge,  there are no results dealing with this expression.
	Nevertheless, it follows from Proposition \ref{convergence-result-infinite-case} that it goes to zero as $\mu\rightarrow-\infty$.
	This is corroborated by numerical experiments in Section \ref{numerical-experiments}.		

	 In the case $T<\infty$, several integral representations for the density of $I_T^{\mu,\sigma,1/2}$ have been obtained in the literature (see for instance Proposition 2.3 in \cite{priv} and  Paper 2 in \cite{yor}).
  	 However, all these representations are complicated and difficult to analyze because they involve several special functions.
  	 From Theorem \ref{cor-upper-bound-finite} and Theorem \ref{lower-bound-finite-ii} we get 
  	 \begin{align*}
  	 	\Phi\left(\frac{1}{\sigma }\sqrt{\frac{2}{T}}\left(\log\left(\frac{T}{x}\right)+\frac{\mu T}{2}\right)\right)\leq P\left[ I_T^{\mu,\sigma,1/2}>x\right]\leq 2\Phi\left(\frac{1}{\sigma \sqrt{T}}\log\left(\frac{e^{\mu T}-1}{\mu x}\right)\right)
  	 \end{align*}
  	 for all $x>\frac{e^{\mu T}-1}{\mu}$, where $\mu\neq 0$.
  	 These estimates are more tractable,  and reveal that the tail of $I_T^{\mu,\sigma,1/2}$ decays similarly as the tail of a log-normal distribution.
  	\bigskip

	\subsection{Continuity in law of exponential functionals of fBMs}
	
	In this subsection, we prove the continuity in law of $I_T^{\mu,\sigma,H}$ with respect to parameters $\mu,\sigma,H$.  

	Let $\mathcal{C}\coloneqq C([0,\infty); \mathbb{R})$ be the set of all real-valued continuous functions on $[0,\infty)$ endowed with the metric
	\begin{align*}
		\rho(f,g)\coloneqq\sum_{k=1}^{\infty}\left(\frac{1}{2^{k}}\right)\left(\frac{\rho_k(f,g)}{1+\rho_k(f,g)}\right), \quad f,g\in \mathcal{C},
	\end{align*}
	where $\rho_k(f,g)\coloneqq\sup_{t\in[0,k]}|f(t)-g(t)|$, $k\geq 1$.
	%and let $\mathscr{C}$ be the Borel $\sigma$-algebra on $(\mathcal{C},\rho)$. 
	For each $k\geq 1$, let $\mathcal{C}_k\coloneqq C([0,k];\mathbb{R})$ be the Banach space of all real-valued  continuous functions on $[0,k]$ equipped with the norm 
	\begin{eqnarray}\label{eq:norm}
		\left\lVert f\right\rVert_k\coloneqq\sup_{t\in[0,k]}|f(t)|,\quad f\in\mathcal{C}_k.
	\end{eqnarray}
	%and let $\mathscr{C}_k$ be the Borel $\sigma$-algebra on $(\mathcal{C}_k,\left\lVert\cdot\right\rVert_k)$. 
	
	First, we prove that $B^{H}$ is continuous in law with respect to its Hurst parameter $H$.
	
	\begin{lem}
	If $\{H_n\}_{n\geq 1}\subset (0,1]$ is a sequence that converges to $H\in (0,1]$, then $\{ B^{H_n}\}_{n\geq 1}$ converges weakly to $B^{H}$ in $(\mathcal{C},\rho)$.
	\end{lem}
	
	\begin{proof}
		We notice that $\mathbb{E}[B_t^{H_n}]=0$ for all $n\geq 1$ and $t\geq 0$, and
		\begin{align*}
			\mathbb{E}[B_s^{H_n}B_t^{H_n}]=\frac{1}{2}\left(s^{2H_n}+t^{2H_n}-|s-t|^{2H_n}\right)\rightarrow \frac{1}{2}\left(s^{2H}+t^{2H}-|s-t|^{2H}\right)=\mathbb{E}[B_s^{H}B_t^{H}]
		\end{align*}
		as $n\rightarrow\infty$ for all $s,t\geq 0$.  Since $B^{H}$ and  $B^{H_n}$, $n\geq 1$, are Gaussian processes, it follows that the finite-dimensional distributions of $B^{H_n}$ converge  to those of $B^{H}$ as $n\rightarrow\infty$. Fix $k\geq 1$. Define $h\coloneqq \inf_{n\geq 1} H_n>0$, and let $p\geq 1$ be an even integer such that $ph>1$.  
		We have
		\begin{align*}
			\mathbb{E}\left[\left|B_t^{H_n}-B_s^{H_n}\right|^{p}\right]= \mathbb{E}\left[\left|B_{t-s}^{H_n}\right|^{p}\right]=(p-1)!!|t-s|^{p H_n}\leq k^{p}(p-1) !! |t-s|^{ph}
		\end{align*}
		for all $0\leq s\leq t\leq k$ and $n\geq 1$, where $(p-1)!!$ denotes the double factorial of $p-1$. Since $B_0^{H_n}\equiv 0$ for all $n\geq 1$, it follows from the Kolmogorov-Chentsov tightness criterion (see Theorem 23.7 in\cite{kallenberg}) that $\{\{B_t^{H_n}:0\leq t\leq k\}\}_{n\geq 1}$ is tight. By using Prohorov's theorem,  we deduce that $\{\{B_t^{H_n}:0\leq t\leq k\}\}_{n\geq 1}$ converges weakly to $\{B_t^{H}:0\leq t\leq k\}$ in $(\mathcal{C}_k,\rho_k)$. Finally, applying Theorem 5 in Whitt \cite{whitt}, we conclude that $\{B^{H_n}\}_{n\geq 1}$ converges weakly to $B^{H}$ in $(\mathcal{C},\rho)$.
	\end{proof}

	We define the process $I^{\mu,\sigma,H}\coloneqq \{ I_t^{\mu,\sigma,H}: t\geq 0\}$ for any $\mu\in\mathbb{R}$, $\sigma\geq 0$ and $H\in(0,1]$. In the next proposition we show that $I^{\mu,\sigma,H}$ is continuous in law with respect to parameters $\mu,\sigma,H$.
	
	\begin{prop}\label{prop-weak-conv-process}
		If $\{(\mu_n,\sigma_n,H_n)\}_{n\geq 1}\subset\mathbb{R}\times [0,\infty)\times (0,1]$ converges to $(\mu,\sigma,H)\in \mathbb{R}\times [0,\infty)\times (0,1]$, then $\lbrace I^{\mu_n,\sigma_n,H_n}\rbrace_{n\geq 1}$ converges weakly to $I^{\mu,\sigma,H}$ in $(\mathcal{C},\rho)$.
	\end{prop}
	
	\begin{proof}
		By Skorohod's representation theorem, there exist $\mathcal{C}$-valued random elements $\tilde{B}^H$ and $\tilde{B}_n^{H_n}$, $n\geq 1$, defined on a common probability space $(\tilde{\Omega},\tilde{\mathcal{F}},\tilde{P})$, such that $\tilde{B}_n^{H_n}(\omega)\rightarrow \tilde{B}^{H}(\omega)$ in $(\mathcal{C},\rho)$ for every $\omega\in\tilde{\Omega}$,  $\tilde{B}^{H}$ has the same distribution as $B^{H}$, and  $\tilde{B}_n^{H_n}$ has the same distribution as  $B^{H_n}$, $n\geq 1$.
		
		Fix $k\geq 1$. Let $\tilde{I}^{\mu_n,\sigma_n,H_n,k}$ be the $\mathcal{C}_k$-valued random element defined by
		\begin{align*}
			\tilde{I}^{\mu_n,\sigma_n,H_n,k}(\omega)(t)\coloneqq \int_0^{t}e^{\mu_n s+\sigma_n \tilde{B}_n^{H_n}(\omega)(s)}\,\mathrm{d}s,\quad \omega\in\tilde\Omega,\quad t\in[0,k],
		\end{align*} 
		for each $n\geq 1$. We define  similarly $\tilde{I}^{\mu,\sigma,H,k}$. Let $\omega\in\tilde{\Omega}$, and let
		\begin{align*}
			M(\omega)\coloneqq k|\mu|+k\sup_{n\geq 1}|\mu_n|+\left(\sup_{n\geq 1}\sigma_n\right)\left(\sup_{n\geq 1}\sup_{t\in[0,k]}\tilde{B}_n^{H_n}(\omega)(t)\right)+\sigma \sup_{t\in[0,k]}\tilde{B}^{H}(\omega)(t).
		\end{align*}
		 Notice that $M(\omega)<\infty$. From the mean  value theorem we get
		\begin{align*}
			&\left\lVert\tilde{I}^{\mu_n,\sigma_n,H_n,k}(\omega)-\tilde{I}^{\mu,\sigma,H,k}(\omega)\right\rVert_k\\
			&\leq \int_0^{k} \left|e^{\mu_n t+\sigma_n \tilde{B}_n^{H_n}(\omega)(t)}-e^{\mu t+\sigma \tilde{B}^{H}(\omega)(t)}\right|\,\mathrm{d}t\\
			&\leq e^{M(\omega)}\int_0^{k}\left|\left(\mu_n t+\sigma_n \tilde{B}_n^{H_n}(\omega)(t)\right)-\left(\mu t+\sigma \tilde{B}^{H}(\omega)(t)\right)\right|\,\mathrm{d}t\\
			&\leq e^{M(\omega)}\left(k^{2}|\mu_n-\mu|+k|\sigma_n-\sigma|\left(\sup_{n\geq 1}\sup_{t\in[0,k]}\tilde{B}_n^{H_n}(\omega)(t)\right)+k\sigma\rho_k\left( \tilde{B}_n^{H_n}(\omega),\tilde{B}(\omega)\right)\right)\\
			&\rightarrow 0
		\end{align*}
		as $n\rightarrow\infty$, where the norm $\left\lVert \cdot\right\rVert_k$ is defined in \eqref{eq:norm}.
		Hence, using Theorem 3.1 in Billingsley \cite{billi}, we deduce that $\{\tilde{I}^{\mu_n,\sigma_n,H_n,k}\}_{n\geq 1}$ converges weakly to $\tilde{I}^{\mu,\sigma,H,k}$ in $(\mathcal{C}_k,\rho_k)$. Applying Theorem 5 in \cite{whitt},  we obtain the desired conclusion.
	\end{proof}

	We conclude this section by proving the following convergence results.
	
	\begin{cor}\label{corollary:kolmogorov_distance}
	Let $T\in(0,\infty)$.	If $\{(\mu_n,\sigma_n,H_n)\}_{n\geq 1}\subset\mathbb{R}\times [0,\infty)\times (0,1]$ converges to $(\mu,\sigma,H)\in \mathbb{R}\times (0,\infty)\times (0,1]$, then
		\begin{align*}
			\lim_{n\rightarrow\infty}\sup_{x\in\mathbb{R}} \left| P\left[I_{T}^{\mu_n,\sigma_n,H_n}\leq x\right]-P\left[I_{T}^{\mu,\sigma,H}\leq x\right]\right|=0.
		\end{align*}
	\end{cor}

	\begin{proof}
		It follows from Proposition \ref{prop-weak-conv-process} that $\{I_T^{\mu_n,\sigma_n,H_n}\}_{n\geq 1}$  converges in distribution to $I_T^{\mu,\sigma,H}$. We deduce from Theorem 2.1.3 in Nualart \cite{nualart} and Lemma 2.1 in Dung et al. \cite{dung-dist} that the c.d.f.\ of $I_T^{\mu,\sigma,H}$ is absolutely continuous %with respect to the Lebesgue measure on $\mathbb{R}$ 
		for any $H\in(0,1)$. On the other hand, we obtain from  Proposition \ref{pro-dist-cas-1} that the c.d.f.\ of  $I_T^{\mu,\sigma,1}$ is continuous. Hence, the conclusion follows from Polya's theorem (see Lemma 3 in \cite[pg. 283]{chow}).
	\end{proof}
	
	\begin{rem}\emph{
		The preceding corollary tell us that $I_T^{\mu_n,\sigma_n,H_n}$  converges in Kolmogorov distance to $I_T^{\mu,\sigma,H}$ as $(\mu_n,\sigma_n,H_n)\rightarrow(\mu,\sigma,H)$.
		For the case $\mu_n\equiv \mu$ and $\sigma_n\equiv \sigma$, an estimate for the Kolmogorov distance between $I_T^{\mu_n,\sigma_n,H_n}$ and $I_T^{\mu,\sigma,H}$  is obtained in \cite[Thm. 1.1]{dung-dist}.}
	\end{rem}
	
	\begin{cor}
		Let  $M>0$ and $x\in\mathbb{R}$. If $\{T_n\}_{n\geq 1}\subset [0,\infty)$ converges to $T\in (0,\infty)$, then
		\begin{align*}
			\lim_{n\rightarrow\infty}\sup_{(\mu,\sigma,H)\in\left[-M,M\right]\times \left[\frac{1}{M},M\right]\times\left[\frac{1}{M},1\right]} \left| P\left[I_{T_n}^{\mu,\sigma,H}\leq x\right]-P\left[I_{T}^{\mu,\sigma,H}\leq x\right]\right|=0.
		\end{align*}
	\end{cor}

	\begin{proof}
	Let $\{T_{n_k}\}_{k\geq 1}$ be a monotone subsequence of $\{T_n\}_{n\geq 1}$, and assume it is increasing. The another case can be proved similarly. Let  $F_k\colon [-M,M]\times[\frac{1}{M},M]\times[\frac{1}{M},1]\rightarrow[0,1]$ be the function given by $F_k(\mu,\sigma,H)\coloneqq P[I_{T_{n_k}}^{\mu,\sigma,H}\leq x]$, $(\mu,\sigma,H)\in[-M,M]\times[\frac{1}{M},M]\times[\frac{1}{M},1]$, for each $k\geq 1$. We observe that $F_k$ is continuous and $F_{k}\geq F_{k+1}$ for all $k\geq 1$. Since $\{F_k\}_{k\geq 1}$ converges pointwise to $(\mu,\sigma,H)\mapsto P[I_{T}^{\mu,\sigma,H}\leq x]$, 
	the desired limit follows from Dini's theorem.
	\end{proof}
			
	\section{Exponential functional of a series of independent fBMs} \label{suma_infinta}
	
	In this section we investigate estimates for the exponential functional of a series of independent fBMs.

	Let $\mu\in\mathbb{R}$ and $\left\lbrace \sigma_n\right\rbrace_{n\geq 1}\subset\mathbb{R}$ be a sequence of numbers such that
	\begin{align*}
		\sigma^{2}\coloneqq \sum_{n=1}^{\infty}\sigma_n^{2}\in (0,\infty).
	\end{align*}
	Let $\{ B_n^{H_n}\}_{n\geq 1}$ be a sequence of independent fBMs defined on a common probability space $(\Omega,\mathcal{F},P)$, where $B_n^{H_n}\coloneqq \{ B_n^{H_n}(t):t\geq 0\}$ denotes a fBM with  Hurst parameter $H_n\in (0,1]$.  Set
	\begin{align*}
		H_0\coloneqq\inf_{\substack{n\geq 1: \\\sigma_n\neq 0}}H_n,\quad H_\infty\coloneqq \sup_{\substack{n\geq 1:\\ \sigma_n\neq 0}}H_n.
	\end{align*}
	We will assume that $H_0>0$. Let us consider  the stochastic process $Z\coloneqq\left\lbrace Z_t: t \geq 0\right\rbrace $ given by
	\begin{align}\label{process_x}
		Z_t\coloneqq \sum_{n=1}^{\infty} \sigma_n B_n^{H_n}(t)
	\end{align}
	for each $t\geq 0$. We notice that $\mathbb{E}\big[\sigma_n B_n^{H_n}(t)\big]=0$ for all $n\geq 1$, and
	\begin{align*}
		\sum_{n=1}^{\infty}\textrm{Var}(\sigma_n B_n^{H_n}(t))=\sum_{n=1}^{\infty}\sigma_n^{2}t^{2H_n}\leq (1+t)^{2}\sigma^{2}<\infty 
	\end{align*}
	for all $t\geq 0$. Hence, by applying Kolmogorov's two-series theorem (Lemma 5.16 in \cite{kallenberg}) we get that the series in (\ref{process_x}) converges $P$-a.s.\  for each $t\geq 0$. Moreover, this series converges in $L^{2}(\Omega)$ uniformly on any compact interval of $[0,\infty)$. Therefore, $Z$ is a centered  Gaussian process whose covariance function is given by
	\begin{align*}
		\mathbb{E}\left[Z_s Z_t\right]=\sum_{n=1}^{\infty}\sigma_n^{2}\mathbb{E}\left[B_n^{H_n}(s)B_n^{H_n}(t)\right]=\sum_{n=1}^{\infty}\frac{\sigma_n^{2}}{2}\left(s^{2H_n}+t^{2H_n}-|s-t|^{2H_n}\right),\quad s,t\geq 0.
	\end{align*}
 	We notice that
	\begin{align*}
		\mathbb{E}\big[(Z_t-Z_s)^{2}\big]=\sum_{n=1}^{\infty}\sigma_{n}^{2}|t-s|^{2H_n}=\mathbb{E}\big[ Z_{t-s}^{2}\big],\quad 0\leq s\leq t,
	\end{align*}
	which implies that $Z$ has stationary increments. Furthermore, since
	\begin{align*}
		\mathbb{E}\big[(Z_t-Z_s)^{2}\big]=|t-s|^{2H_0}\sum_{n=1}^{\infty}\sigma_{n}^{2}|t-s|^{2(H_n-H_0)}\leq  (1+s\vee t)^{2}\sigma^{2}|t-s|^{2H_0}, \quad s,t\geq 0.
	\end{align*}
	 it follows from the Kolmogorov continuity theorem that $Z$  has a modification which has locally H\"older continuous trajectories with  an arbitrary exponent  smaller than $H_0$ $P$-a.s., i.e.,
	\begin{align*}
		\sup_{\substack{r,s\in [0,t]:\\r\neq s}}\frac{|Z_r-Z_s|}{|r-s|^{\alpha}}<\infty \quad \textrm{for all}\quad t\geq 0\quad \textrm{and}\quad \alpha \in (0,H_0)\quad P\textrm{-a.s.\ }
	\end{align*} 
	Let $Z^{\mu}\coloneqq \left\lbrace Z^{\mu}_t: t\geq 0\right\rbrace$ be the continuous Gaussian process given by $Z^{\mu}_t\coloneqq \mu t+Z_t$ for each $t\geq 0$. The mean and covariance functions of $Z$ are given by  
	\begin{align*}
		\mathbb{E}\left[Z_t^{\mu}\right]=\mu t,\quad \textrm{Cov}\left(Z_s^{\mu},Z_t^{\mu}\right)=\mathbb{E}\left[ Z_s Z_t\right],\quad s,t\geq 0.
	\end{align*}
	 Fix $T\in (0,\infty]$. We have $f\in\mathcal{M}_{T}^{\mu,\sigma_n,H_n}$  for any  $f\in\mathcal{M}_T^{Z^{\mu}}$ and $n\geq 1$. Moreover,
	\begin{align*}
		m_{T}^{Z^{\mu}}(f)=m_{T}^{\mu,0,1}(f),\quad \left(s_T^{Z^{\mu}}(f)\right)^{2}=\sum_{n=1}^{\infty}\left(s_T^{\mu,\sigma_n,H_n}(f)\right)^2.
	\end{align*}
	From Proposition \ref{general_bound} we obtain the estimate
	\begin{align*}
		P\left[ I^{Z^{\mu}}_T \leq x\right]\leq \Phi\left(\inf_{f\in\mathcal{M}_T^{Z^{\mu}}} \frac{\log(x)-m_{T}^{\mu,0,1}(f)}{\left(\sum_{n=1}^{\infty}\left(s_T^{\mu,\sigma_n,H_n}(f)\right)^2\right)^{\frac{1}{2}}}\right)
		\end{align*}
	for each $x>0$. It seems impossible to us to give an explicit expression for the above infimum, even if we replace $\mathcal{M}_T^{Z^{\mu}}$ by $\mathcal{N}_T$. Instead, we  can establish explicit upper bounds for the c.d.f.\ of  $I_T^{Z^\mu}$  that depends only on the parameters $\mu,\sigma,T,H_0,H_\infty$.
	
	\begin{cor}\label{cor_bound_funct_serie_finite}
		Let $x>0$. Suppose that $T<\infty$. Then
		\begin{equation*}
			P\left[ I_T^{Z^{\mu}}\leq x\right]\leq \begin{cases}
				\Phi\left(\frac{\sqrt{2H_0+2}}{\sigma \left(T^{H_0}\vee T^{H_\infty}\right)}\left(\log(x)-\log(T)-\frac{\mu T}{2}\right)\right) &\textrm{if } x<Te^{\mu T/2},\\
				\Phi\left(\frac{\sqrt{2H_\infty+2}}{\sigma\left(T^{H_0}\wedge T^{H_\infty}\right)}\left(\log(x)-\log(T)-\frac{\mu T}{2}\right)\right)&\textrm{if }x\geq Te^{\mu T/2}.
			\end{cases}
		\end{equation*}
	\end{cor}
	
	\begin{proof}
		First notice that $f_{0,T}\in \mathcal{M}_T^{Z^\mu}$, where $f_{0,T}$ is defined in (\ref{exponential-density}). We have
		\begin{align*}
			m_T^{Z^{\mu}}(f_{0,T})&=m_{T}^{\mu,0,1}(0)=\log(T)+\frac{\mu T}{2},\\
		\left(s_T^{Z^\mu}(f_{0,T})\right)^{2}&=\sum_{n=1}^{\infty}\left(s_T^{\mu,\sigma_n,H_n}(0)\right)^{2}=\sum_{n=1}^{\infty}\frac{\sigma_n^{2}T^{2H_n}}{2H_n+2}.
	\end{align*}
	Then
	\begin{align*}
		P\left[I_T^{Z^{\mu}}\leq x\right]\leq \Phi\left(\frac{\log(x)-m_T^{Z^{\mu}}(f_{0,T})}{s_T^{Z^{\mu}}(f_{0,T})}\right)
		= \Phi\left(\frac{\log(x)-\log(T)-\frac{\mu T}{2}}{\sqrt{\sum_{n=1}^{\infty}\frac{\sigma_n^{2}T^{2H_n}}{2H_n+2}}}\right).
	\end{align*}
	Using the fact that $\Phi$ is increasing and the inequality
	\begin{align*}
		 \frac{\left(T^{2H_0}\wedge T^{2H_\infty}\right)\sigma^{2}}{2H_\infty+2}\leq \sum_{n=1}^{\infty}\frac{\sigma_n^{2}T^{2H_n}}{2H_n+2}\leq \frac{\left(T^{2H_0}\vee T^{2H_\infty}\right)\sigma^{2}}{2H_0+2},
	\end{align*}
	we obtain the desired estimate.
	\end{proof}

		\begin{rem}\emph{
		Notice that if $\sigma_n=0$ for all $n\geq 2$, then $I_T^{Z^{\mu}}$ coincides with the functional $I_T^{\mu,\sigma,H_0}$.
		In this case, the estimates for $P\left[ I_T^{Z^{\mu}}\leq x\right]$ proved in Corollary \ref{cor_bound_funct_serie_finite} coincide with the one obtained in \eqref{finite_log}.
	}\end{rem}

	\begin{cor}\label{cor_bound_series}
		Let $x>0$. The following statements hold:
		\begin{enumerate}
		\item[(i)] If $\mu>0$ and $H_\infty<1$, then $P\big[I_\infty^{Z^{\mu}}=\infty\big]=1$.
		\item[(ii)] If $\mu=0$ and $H_\infty<1$, then 
		\begin{align*}
			P\left[I_\infty^{Z^{\mu}}\leq x\right]\leq \Phi\left(-\sqrt{\frac{2}{\Gamma(2H_0+1)\vee \Gamma(2H_\infty+1)}}\cdot\frac{1}{\sigma\left(x^{H_0}\vee x^{H_\infty}\right)}\right).
		\end{align*}
		\item[(iii)] If $\mu<0$ and $H_\infty<1$, then
		\begin{align*}
		P\left[I_\infty^{Z^{\mu}}\leq x\right]\leq \begin{cases}
			\Phi\left(\frac{(-\mu)^{H_0}\vee (-\mu)^{H_\infty}}{\sigma}\sqrt{\frac{2}{\Gamma_0}}\log(-\mu x)\right)&\textrm{if }\mu\leq -\frac{1}{x},\\
			\Phi\left(\frac{(-\mu)^{H_0}\wedge (-\mu)^{H_\infty}}{\sigma}\sqrt{\frac{2}{\Gamma(2H_0+1)\vee \Gamma(2H_\infty+1)}}\log(-\mu x)\right)&\textrm{if }\mu> -\frac{1}{x},
		\end{cases}
		\end{align*}
	where $\Gamma_0\coloneqq \min_{z>0}\Gamma(z)\approx 0.88$.
		\item[(iv)] If $\mu\in\mathbb{R}$ and $H_\infty=1$, then
		\begin{align*}
			P\left[I_\infty^{Z^{\mu}}\leq x\right]\leq 
			\begin{cases}
				\Phi\left(-\sqrt{\frac{2}{\Gamma_0}}\cdot \frac{1+\mu x}{\sigma\left(x\wedge x^{H_0}\right)}\right)&\textrm{if }\mu\leq -\frac{1}{x},\\
				\Phi\left(-\frac{1+\mu x}{\sigma\left(x\vee x^{H_0}\right)}\right)&\textrm{if }\mu> -\frac{1}{x}.
			\end{cases}
		\end{align*}
		\end{enumerate}
	\end{cor}

	\begin{proof}
		Let $\lambda>0$. Notice that $f_{\lambda}\in \mathcal{M}_\infty^{Z^\mu}$, where $f_\lambda$ is defined in (\ref{exponential-density-2}). We have
		\begin{align*}
			m_\infty^{Z^{\mu}}(f_\lambda)&=m_{\infty}^{\mu,0,1}(\lambda)=-\log(\lambda)+1+\frac{\mu}{\lambda},\\
			\left(s_\infty^{Z^\mu}(f_\lambda)\right)^{2}&=\sum_{n=1}^{\infty}\left(s_\infty^{\mu,\sigma_n,H_n}(\lambda)\right)^{2}=\sum_{n=1}^{\infty}\frac{\sigma_n^{2}}{2\lambda^{2H_n}}\Gamma(2H_n+1).
		\end{align*}
		Then
		\begin{align*}
			P\left[I_\infty^{Z^{\mu}}\leq x\right]\leq \Phi\left(\frac{\log(x)-m_\infty^{Z^{\mu}}(f_{\lambda})}{s_\infty^{Z^{\mu}}(f_\lambda)}\right)
			= \Phi\left(\frac{\log(x)+\log(\lambda)-1-\frac{\mu}{\lambda}}{\sqrt{\sum_{n=1}^{\infty}\frac{\sigma_n^{2}}{2\lambda^{2H_n}}\Gamma(2H_n+1)}}\right).
		\end{align*}
		Using the inequality
		\begin{align*}
		\frac{\sigma^{2} \Gamma_0}{2\left(\lambda^{2H_0}\vee\lambda^{2H_\infty}\right)}\leq \sum_{n=1}^{\infty}\frac{\sigma_n^{2}}{2\lambda^{2H_n}}\Gamma(2H_n+1)\leq \frac{\sigma^{2}\left(\Gamma(2H_0+1)\vee  \Gamma(2H_\infty+1)\right)}{2\left(\lambda^{2H_0}\wedge\lambda^{2H_\infty}\right)},
		\end{align*}
		we deduce that
		\begin{align}\label{cotas_cases}
			P\left[ I_\infty^{Z^{\mu}}\leq x\right]\leq \begin{cases}
					\Phi\left(\frac{\lambda^{H_0}\vee\lambda^{H_\infty}}{\sigma}\sqrt{\frac{2}{\Gamma_0}}\left(\log(x)+\log(\lambda)-1-\frac{\mu}{\lambda}\right)\right)&\textrm{if }\frac{e^{\mu/\lambda}}{\lambda}\leq \frac{x}{e},\\
					\Phi\left(\frac{\lambda^{H_0}\wedge \lambda^{H_\infty}}{\sigma}\sqrt{\frac{2}{\Gamma(2H_0+1)\vee \Gamma(2H_\infty+1)}}\left(\log(x)+\log(\lambda)-1-\frac{\mu}{\lambda}\right)\right)&\textrm{if } \frac{e^{\mu/\lambda}}{\lambda}>\frac{x}{e}.
		\end{cases}
		\end{align}
		Suppose first that $\mu>0$ and $H_\infty<1$. Since $\lim_{\lambda\downarrow 0}\frac{e^{\mu/\lambda}}{\lambda}=\infty$, we have
		\begin{align*}
			P\left[ I_\infty^{Z^{\mu}}\leq x\right]\leq \Phi\left(\frac{\lambda^{H_\infty}}{\sigma}\sqrt{\frac{2}{\Gamma(2H_0+1)\vee \Gamma(2H_\infty+1)}}\left(\log(x)+\log(\lambda)-1-\frac{\mu}{\lambda}\right)\right)
		\end{align*}
		for all sufficiently small $\lambda$. By letting $\lambda$ tend to zero in the above expression, we get $P\big[I_\infty^{Z^{\mu}}\leq x\big]=0$. Then, by letting $x$ tend to infinity, we conclude that $P\big[I_\infty^{Z^{\mu}}<\infty\big]=0$, which proves part (i). The statement (iii) follows from (\ref{cotas_cases}), replacing $\lambda$ by $-\mu$. On the other hand,  parts (ii) and (iv) follows from (\ref{cotas_cases}) but now replacing $\lambda$ by $\frac{1}{x}$.
	\end{proof}

	\begin{rem}\emph{
		Let $z_0\coloneqq \textrm{arg min}_{z>0} \Gamma(z)\approx 1.46$.
		We recall that the  Gamma function is increasing on the interval $[z_0,\infty)$, and decreasing on $[0,z_0]$. If $2H_0+1>z_0$, then we can replace $\Gamma_0$ by $\Gamma(2H_0+1)$ in the proof of Corollary \ref{cor_bound_series} and obtain a better bound. If $2H_\infty+1<\Gamma_0$, then  $\Gamma_0$ can be replaced by $\Gamma(2H_\infty+1)$. 
	}\end{rem}
	
	The moments of $I_T^{Z^{\mu}}$ are estimated in the next corollaries.
	
	\begin{cor}
		Let $T\in (0,\infty)$. Then   $I_T^{Z^{\mu}}\in L^{p}(\Omega)$ for all $p\geq 1$. Moreover,
		\begin{align*}
			T^{p}\exp\left(\frac{\mu p T}{2}+\frac{p^{2}\sigma^{2}\left(T^{2H_0}\wedge T^{2H_\infty}\right)}{4H_\infty+4}\right)
			&\leq \mathbb{E}\left[\left(I_T^{Z^{\mu}}\right)^{p}\right]\\
			&\leq
			\begin{cases}
				T^{p} \exp\left(\frac{1}{2}p^{2}\sigma^{2}\left(T^{2 H_0}\vee T^{2H_\infty}\right)\right)&\textrm{if } \mu=0,\\
				\left(\frac{e^{\mu p T}-1}{\mu p}\right)T^{p-1} e^{\frac{1}{2}p^{2}\sigma^{2}\left(T^{2 H_0}\vee T^{2H_\infty}\right) }&\textrm{if }\mu> 0,\\
				\left(\frac{e^{\mu  T}-1}{\mu }\right)^{p} e^{\frac{1}{2}p^{2}\sigma^{2}\left(T^{2 H_0}\vee T^{2H_\infty}\right) }&\textrm{if }\mu< 0,
			\end{cases}
		\end{align*}
	for any $p\geq 1$.
	\end{cor}
	
	\begin{proof}
		From  Proposition \ref{moments_functional} (ii) we obtain
		\begin{align*}
			\mathbb{E}\left[\left(I_T^{Z^{\mu}}\right)^{p}\right]&\geq \exp\left(p m_T^{Z^{\mu}}(f_{0,T})+\frac{1}{2}p^{2}\left(s_T^{Z^{\mu}}(f_{0,T})\right)^{2}\right)\\
			&\geq T^{p}\exp\left(\frac{\mu p T}{2}+\frac{\sigma^{2}p^{2}\left(T^{2H_0}\wedge T^{2H_\infty}\right)}{4H_\infty+4}\right).
		\end{align*}
		If $\mu>0$, we have
		\begin{align*}
			\mathbb{E}\left[\left(I_T^{Z^{\mu}}\right)^{p}\right]\leq \int_0^{T}e^{\mu p t+\frac{1}{2}p^{2}\sum_{n=1}^{\infty}\sigma_{n}^{2}t^{2 H_n}}(f_{0,T}(t))^{1-p}\,\mathrm{d}t
			\leq \left(\frac{e^{\mu p T}-1}{\mu p}\right)T^{p-1} e^{\frac{1}{2}p^{2}\sigma^{2}\left(T^{2 H_0}\vee T^{2H_\infty}\right) },
		\end{align*}
		and if $\mu<0$ we get
		\begin{align*}
			\mathbb{E}\left[\left(I_T^{Z^{\mu}}\right)^{p}\right]\leq \int_0^{T}e^{\mu p t+\frac{1}{2}p^{2}\sum_{n=1}^{\infty}\sigma_{n}^{2}t^{2 H_n}}(f_{\mu,T}(t))^{1-p}\,\mathrm{d}\leq \left(\frac{e^{\mu  T}-1}{\mu }\right)^{p} e^{\frac{1}{2}p^{2}\sigma^{2}\left(T^{2 H_0}\vee T^{2H_\infty}\right) }.
		\end{align*}
		The case $\mu=0$ can be handled in  a similar way.
	\end{proof}
	
	\begin{cor}\label{cor-moments-serie}
		The following statements hold:
		\begin{enumerate}
			\item[(i)] If $\mu\geq 0$, then $\mathbb{E}\big[\big(I_\infty^{Z^{\mu}}\big)^{p}\big]=\infty$ for all $p>0$.
			\item[(ii)] If $\mu< 0$ and $H_\infty>\frac{1}{2}$, then $\mathbb{E}\big[\big(I_\infty^{Z^{\mu}}\big)^{p}\big]=\infty$ for all $p>0$.
			\item[(iii)] If $\mu<0$ and $H_\infty<\frac{1}{2}$, then $I_\infty^{Z^{\mu}}\in L^{p}(\Omega)$ for all $p\geq 1$ , and
			\begin{align*}
					\mathbb{E}\left[\left(I_\infty^{Z^{\mu}}\right)^{p}\right]\leq
					\frac{1}{(-\mu)^{p}}\left((1-e^{\mu})e^{\frac{1}{2}p^{2}\sigma^{2}}+2e^{\frac{\mu}{2}}\exp\left(\left(\frac{1}{2}-H_\infty\right)\left(\left(-\frac{2H_\infty}{\mu}\right)^{H_\infty} p\sigma\right)^{\frac{2}{1-2H_\infty}}\right)\right)
			\end{align*}
		for any $p\geq 1$. Also,
			\begin{align*}
				\mathbb{E}\left[\left(I_\infty^{Z^{\mu}}\right)^{p}\right]\geq \sup_{\lambda>0}\frac{1}{\lambda^{p}}\exp\left(p\left(1+\frac{\mu}{\lambda}\right)+\frac{p^{2}\sigma^{2}\Gamma_0}{4\left(\lambda^{2H_0}\vee \lambda^{2H_\infty}\right)}\right),\quad p>0.
			\end{align*}
		\end{enumerate}
	\end{cor}

	\begin{proof} Part (i) follows from Corollary \ref{cor_bound_series}  (i), (ii) and (iv). If $\mu<0$ and $H_\infty>\frac{1}{2}$, then there exists $k\geq 1$ such that $\sigma_k\neq 0$ and $H_k>\frac{1}{2}$. Then it follows from  Proposition \ref{moments_functional} (i) that
		\begin{align*}
				\mathbb{E}\left[\left(I_\infty^{Z^{\mu}}\right)^{p}\right]&\geq\int_0^{\infty}e^{\mu pt +\frac{1}{2}p^{2}\sum_{n=1}^{\infty}\sigma_{n}^{2}t^{2 H_n}}(f_{-\mu}(t))^{1-p}\,\mathrm{d}t\\
				&=(-\mu)^{1-p}  \int_0^{\infty}e^{\mu t+\frac{1}{2}p^2\sum_{n=1}^{\infty}\sigma_n^{2} t^{2H_n} }\,\mathrm{d}t\\
				&\geq (-\mu)^{1-p}  \int_0^{\infty}e^{\mu t+\frac{1}{2}p^2\sigma_k^{2} t^{2H_k} }\,\mathrm{d}t\\
				&=\infty
		\end{align*}
		for any $p\in (0,1)$. Thus, part (ii) holds. Now suppose that $\mu<0$ and $H_\infty<\frac{1}{2}$.  From Proposition \ref{moments_functional} (ii) we get
		\begin{align*}
			\mathbb{E}\left[\left(I_\infty^{Z^{\mu}}\right)^{p}\right]
			&\leq \int_0^{\infty}e^{\mu pt +\frac{1}{2}p^{2}\sum_{n=1}^{\infty}\sigma_{n}^{2}t^{2 H_n}}(f_{-\mu}(t))^{1-p}\,\mathrm{d}t\\
			&=(-\mu)^{1-p}\int_0^{\infty}e^{\mu t+\frac{1}{2}p^2\sum_{n=1}^{\infty}\sigma_n^{2}t^{2H_n}}\,\mathrm{d}t\\
			&\leq (-\mu)^{1-p}\left(\int_0^{1}e^{\mu t+\frac{1}{2}p^{2}\sigma^{2} t^{2H_0}}\,\mathrm{d}t+\int_1^{\infty}e^{\mu t+\frac{1}{2}p^{2}\sigma^{2} t^{2H_\infty}}\,\mathrm{d}t\right)\\
			&\leq  \frac{1}{(-\mu)^{p-1}}\left(\left(\frac{e^{\mu}-1}{\mu}\right)e^{\frac{1}{2}p^{2}\sigma^{2}}-\frac{2e^{\frac{\mu}{2}}}{\mu}\exp\left(\left(\frac{1}{2}-H_\infty\right)\left(\left(-\frac{2 H_\infty}{\mu}\right)^{H_\infty} p\sigma\right)^{\frac{2}{1-2H_\infty}}\right)\right)
		\end{align*}
		for any $p\geq 1$. Hence, $I_\infty^{Z^{\mu}}\in L^{p}(\Omega)$ for all $p\geq 1$, and it follows from Proposition \ref{moments_functional}  (ii) that
		\begin{align*}
				\mathbb{E}\left[\left(I_\infty^{Z^{\mu}}\right)^{p}\right]&\geq \sup_{\lambda>0}\exp\left(p\left(-\log(\lambda)+1+\frac{\mu}{\lambda}\right)+\frac{1}{2}p^{2}\sum_{n=1}^{\infty}\frac{\sigma_n^{2}}{2\lambda^{2H_n}}\Gamma(2H_n+1)\right)\\
				&\geq \sup_{\lambda>0}\frac{1}{\lambda^{p}}\exp\left(p\left(1+\frac{\mu}{\lambda}\right)+\frac{\sigma^{2}p^{2}\Gamma_0}{4\left(\lambda^{2H_0}\vee \lambda^{2H_\infty}\right)}\right).
		\end{align*}
	\end{proof}

	We now give some lower bounds for the c.d.f.\ of $I_T^{Z^{\mu}}$.
	
	\begin{cor}
		The following statements hold:
		\begin{enumerate}
			\item[(i)]
			Let $T<\infty$ and $H_\infty<\frac{1}{2}$. Suppose that $\rho \coloneqq \sum_{n=1}^{\infty}\sigma_n<\infty$. Then for any bounded continuous function $f$ on $[0,T)$ we have
			\begin{align*}
				&P\left[I_T^{Z^{\mu}}\leq x\right]\\
				&\geq  1-\exp\left(-\frac{\left(\log\left( \frac{x}{\int_0^{T}e^{\mu t+f(t)}\,\mathrm{d}t}\right)+f_{\inf}-\frac{3.75 \sqrt{2\pi} \rho\left(T^{H_0}\vee T^{H_\infty}\right)}{\sqrt{H_0(\log 2)^{3}}}\emph{\textrm{erfc}}\left(\sqrt{\frac{H_0\log 2}{2}}\right)\right)^{2}}{2\sigma^{2}\left(T^{2H_0}\vee T^{2H_\infty}\right)}\right)
			\end{align*}
		for all $x>\exp\left(\frac{3.75 \sqrt{2\pi} \rho\left(T^{H_0}\vee T^{H_\infty}\right)}{\sqrt{H_0(\log 2)^{3}}}\emph{\textrm{erfc}}\left(\sqrt{\frac{H_0\log 2}{2}}\right)-f_{\inf}\right)\int_0^{T}e^{\mu t+f(t)}\,\mathrm{d}t$.
			\item[(ii)]Let $T<\infty$ and $H_0\geq \frac{1}{2}$. Suppose that $\rho \coloneqq \sum_{n=1}^{\infty}\sigma_n<\infty$. Then for any bounded continuous function $f$ on $[0,T)$ we have
			\begin{align*}
				P\left[I_T^{Z^{\mu}}\leq x\right]\geq  1-\exp\left(-\frac{\left(\log\left( \frac{x}{\int_0^{T}e^{\mu t+f(t)}\,\mathrm{d}t}\right)+f_{\inf}- \sqrt{\frac{2}{\pi}}\rho\left(T^{H_0}\vee T^{H_\infty}\right)\right)^{2}}{2\sigma^{2}\left(T^{2H_0}\vee T^{2H_\infty}\right)}\right)			
			\end{align*}
			for all $x>\exp\left( \sqrt{\frac{2}{\pi}}\rho\left(T^{H_0}\vee T^{H_\infty}\right)-f_{\inf}\right)\int_0^{T}e^{\mu t+f(t)}\,\mathrm{d}t$.
			\item[(iii)] If $\mu<0$ and $H_\infty<\frac{1}{2}$, then
			\begin{align*}
			P\left[I_\infty^{Z^{\mu}}\leq x\right]\geq
			1+\frac{1}{\mu x}\left((1-e^{\mu})e^{\frac{\sigma^{2}}{2}}+2e^{\frac{\mu}{2}}\exp\left(\left(\frac{1}{2}-H_\infty\right)\left(\left(-\frac{2H_\infty}{\mu}\right)^{H_\infty} \sigma\right)^{\frac{2}{1-2H_\infty}}\right)\right)
			\end{align*}			
		for any $x>0$.
		\end{enumerate}
	\end{cor}
	
	\begin{proof}
		Assertions (i)  and (ii) follow from Proposition \ref{cor-lower-bound-finite-t}, using  the estimates
		\begin{align*}
			\mathbb{E}\left[\sup_{t\in[0,T]}\left(Z^{\mu}_t-\mathbb{E}[Z^{\mu}_t]\right)\right]
			\leq \frac{3.75 \sqrt{2\pi} \rho\left(T^{H_0}\vee T^{H_\infty}\right)}{\sqrt{H_0(\log 2)^{3}}}\textrm{erfc}\left(\sqrt{\frac{H_0\log 2}{2}}\right) \quad\textrm{if}\quad H_\infty<\frac{1}{2},
		\end{align*}
	 and 
	\begin{align*}
		\mathbb{E}\left[\sup_{t\in[0,T]}\left(Z^{\mu}_t-\mathbb{E}[Z^{\mu}_t]\right)\right]&\leq \sqrt{\frac{2}{\pi}}\rho\left(T^{H_0}\vee T^{H_\infty}\right) \quad\textrm{if}\quad H_0\geq \frac{1}{2},
	\end{align*}
	where we have used the inequality $\mathbb{E}\big[\sup_{t\in[0,1]} B_t^{H}\big]\leq \sqrt{\frac{2}{\pi}}$ which is valid for any $H\in\big[\frac{1}{2},1\big]$ (see Lemma 5 in Shao \cite{shao}). The statement in (iii) follows from Markov's inequality and Corollary \ref{cor-moments-serie} (iii).
	\end{proof}

	We conclude this section by examining the finiteness of  $I_\infty^{Z^{\mu}}$.
	By Corollary \ref{cor_bound_series} (i), we know that $I_\infty^{Z^{\mu}}=\infty$ $P$-a.s.\ if $\mu>0$ and $H_\infty<1$. 
	Also, from Corollary \ref{cor_bound_series} (iv) we have $P\big[I_\infty^{Z^{\mu}}<\infty\big]\leq\Phi\left(-\frac{\mu}{\sigma}\right)$ if $\mu\geq 0$ and $H_\infty=1$,  
	and it follows from  Corollary \ref{cor_bound_series} (ii) that $P\big[I_\infty^{Z^{\mu}}<\infty\big]\leq \frac{1}{2}$ if $\mu=0$ and $H_\infty<1$.
	On the other hand, we deduce from Corollary \ref{cor-moments-serie} (iii) that $I_\infty^{Z^{\mu}}<\infty$ $P$-a.s. if $\mu<0$ and $H_\infty<\frac{1}{2}$.
	Suppose now that $\mu<0$, $H_\infty<1$ and there exists $N\geq 1$ such that $\sigma_n=0$ for all $n> N$, i.e., $Z$ is a mfBM. By H\"older's inequality we have
	\begin{align*}
		I_{\infty}^{Z^{\mu}}=\int_0^{\infty} e^{\mu t}\prod_{n=1}^{N} e^{\sigma_n B_n^{H_n}(t)}\,\mathrm{d}t
		\leq  \prod_{n=1}^{N} \left(\int_0^{\infty} e^{\mu t+\sigma_n N B_n^{H_n}(t)}\,\mathrm{d}t\right)^{1/N}
		=\prod_{n=1}^{N} \left(I_\infty^{\mu,\sigma_n N,H_n}\right)^{1/N}.
	\end{align*} 
	Hence, $I_{\infty}^{Z^{\mu}}<\infty$ $P$-a.s.\ We  also have the following criterion for the finiteness of $Z^{\mu}$.
	
	\begin{prop}
		If $H_\infty<\frac{1+H_0}{2}$ and $\mu<0$, then $Z^{\mu}<\infty$ $P$-a.s.
	\end{prop}
	
	\begin{proof}
		Let $v\colon[0,\infty)\rightarrow[0,\infty)$ be the function given by
		\begin{align*}
			v(t)\coloneqq \frac{1}{2}\mathbb{E}\big[Z_t^{2}\big]=\frac{1}{2}\sum_{n=1}^{\infty}\sigma_n^{2}t^{2H_n},\quad t\geq 0.
		\end{align*}
		We have
		\begin{align*}
			0< \left(\frac{t}{s}\right)^{2H_0}v(s)\leq 	v(t) 	\leq  \left(\frac{t}{s}\right)^{2H_\infty}v(s)
		\end{align*}
		for all $0<s\leq t$. Since $2H_\infty<\frac{1}{2}(2H_0)+1$, we can apply Theorem 1.1 in Orey \cite{orey} to obtain
		\begin{align*}
			\limsup_{t\rightarrow\infty}\frac{Z_t}{\sigma t^{H_\infty}\sqrt{2\log(\log(t))}}\leq 1\quad P\textrm{-a.s.}
		\end{align*}
		By a similar argument to the one given in Lemma 1 in \cite{dozzi3},  we obtain the desired conclusion. 
	\end{proof}

	\section{Numerical examples and discussion}\label{numerical-experiments}
	
	In this last section, we present some plots of our estimates for the c.d.f.\ of the exponential functional of fBM.
	We do not make an exhaustive analysis of all possible cases for the parameters $\mu$, $\sigma$, $T$ and $H$.
	Instead, our aim is to illustrate graphically the applicability of our estimates.
	Let us remark that one advantage of our bounds is that they are computable and not asymptotical.
	
	First, we estimate the c.d.f.\ of $I_T^{\mu,\sigma,H}$ in the case where $T<\infty$.
	We consider the values $\mu=-1,0,1$, $\sigma=1/10, 1, 2$, $T=1$ and $H=1/4,3/4$. 
	Notice that in this case, we can simulate the random variable $I_T^{\mu,\sigma,H}$ by simulating paths of fBM on $[0,T]$ and computing the integral of their exponentials.
	Furthermore, if our sample size for $I_T^{\mu,\sigma,H}$ is large enough, we may expect that the  empirical cumulative distribution function (e.c.d.f.) of $I_T^{\mu,\sigma,H}$ approximates well the c.d.f.\ of $I_T^{\mu,\sigma,H}$.
	For the case $H=1/4$, we observe in Figure \ref{figure:finite_1} that the lower bound is not so accurate.
	Recall that the lower bound proved in Theorem \ref{lower-bound-finite-ii} for $H<1/2$ was obtained by applying  maximal inequalities, and this bound is similar to the one obtained in \cite{dung} as $x\rightarrow\infty$.
	In contrast, the upper bound seems to approximate well the e.c.d.f. of $I_T^{\mu,\sigma,H}$, especially in the case $\mu=0$.
	Thus, it seems reasonable to conjecture that the c.d.f.\ of $I_T^{\mu,\sigma,H}$, with $H<1/2$, can be approximated by a log-normal c.d.f.
	
	For the case $H=3/4$, we  can appreciate in Figure \ref{figure:finite_2} that the lower bound is more accurate in comparison to the previous case.
	Recall from Remark \ref{improved-bound} that for $H\geq 1/2$, the c.d.f.\ of $I_T^{\mu,\sigma,H}$ is upper and lower bounded by some log-normal c.d.f.s for $x$ large enough.
	Again, the upper bound seems to be close to the e.c.d.f. of $I_T^{\mu,\sigma,H}$.
	As an example, we present in Figure \ref{figure:finite_3} a plot of our estimates for the moment-generating function of exponential functionals of fBM obtained in Corollary \ref{cor-gmf-finite}.
	These estimates have a better performance when $H\geq 1/2$.
	
	\begin{figure}[htbp]
		\begin{subfigure}[t]{0.26\textwidth}
			\includegraphics[width=\linewidth]{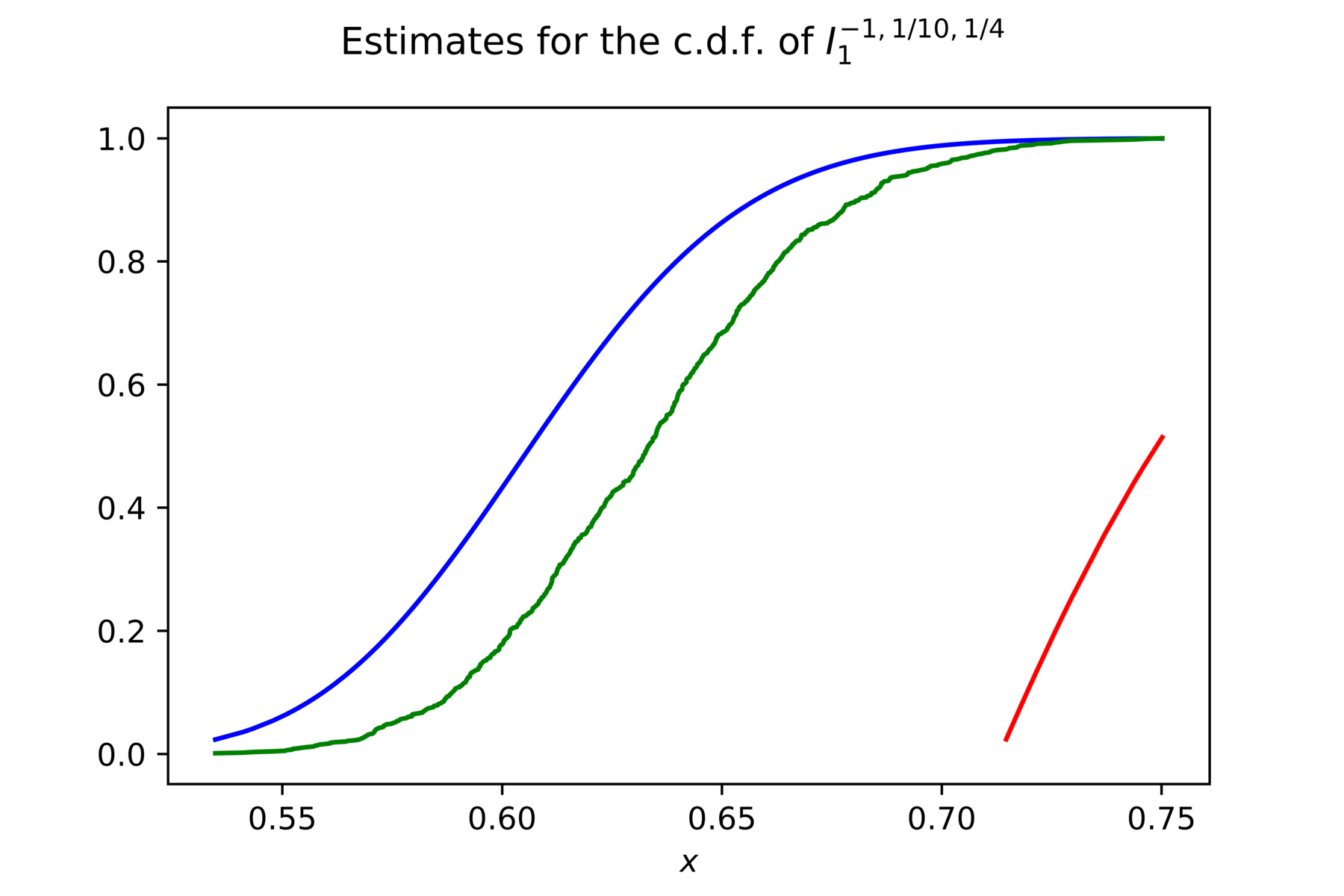}
		\end{subfigure}\hfill
		\begin{subfigure}[t]{0.26\textwidth}
			\includegraphics[width=\linewidth]{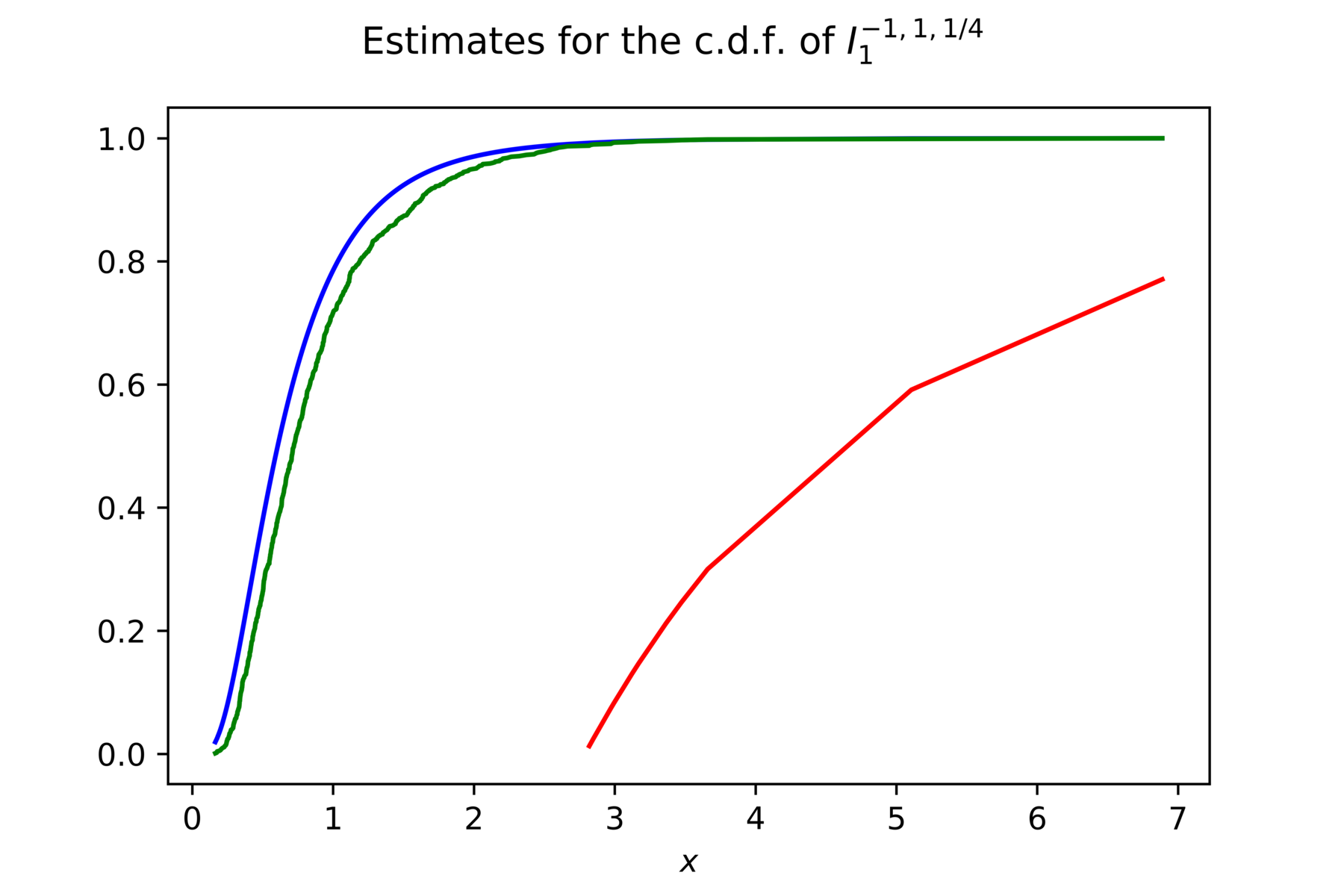}
		\end{subfigure}\hfill
		\begin{subfigure}[t]{0.26\textwidth}
			\includegraphics[width=\linewidth]{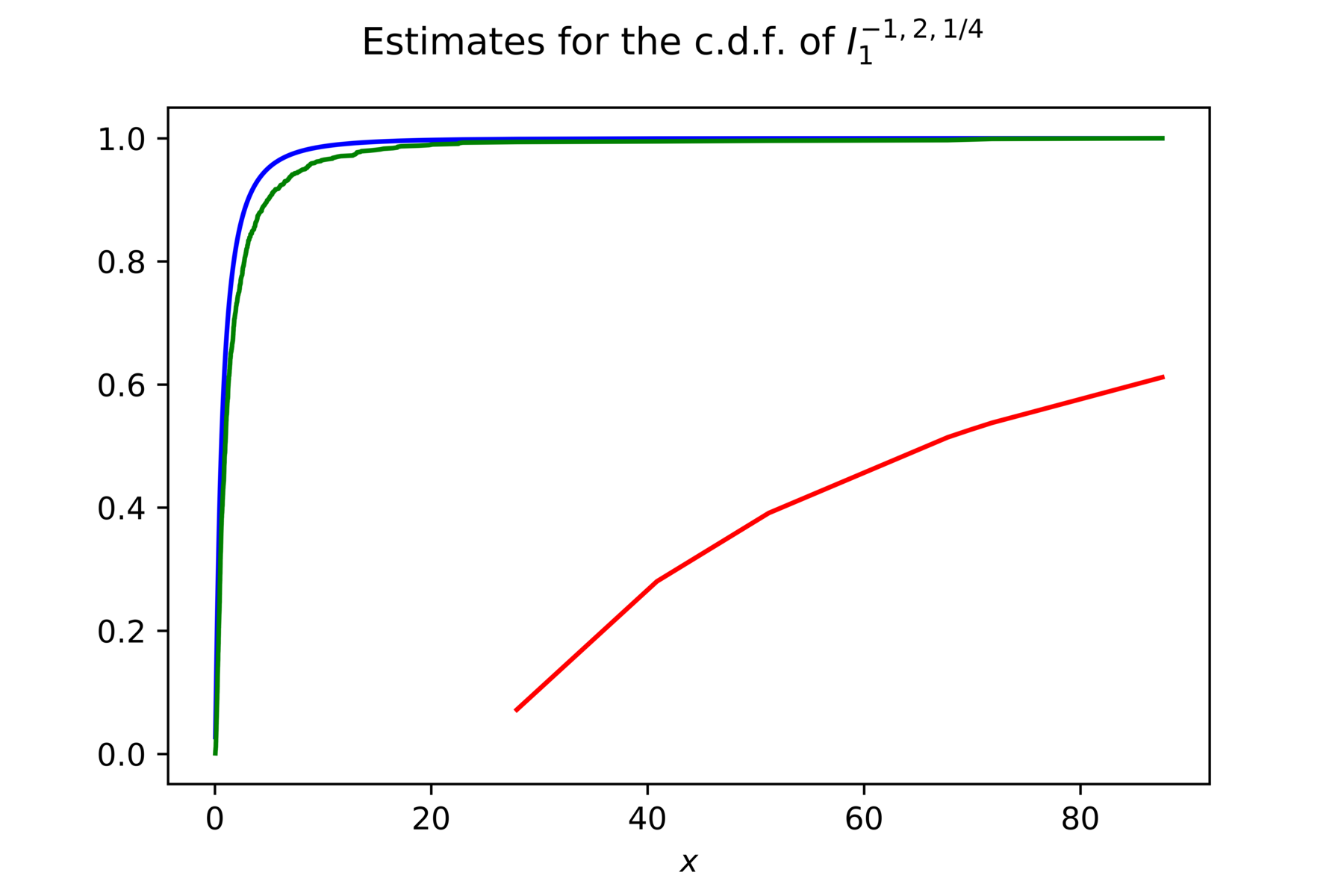}
		\end{subfigure}
		
		\begin{subfigure}[t]{0.26\textwidth}
			\includegraphics[width=\linewidth]{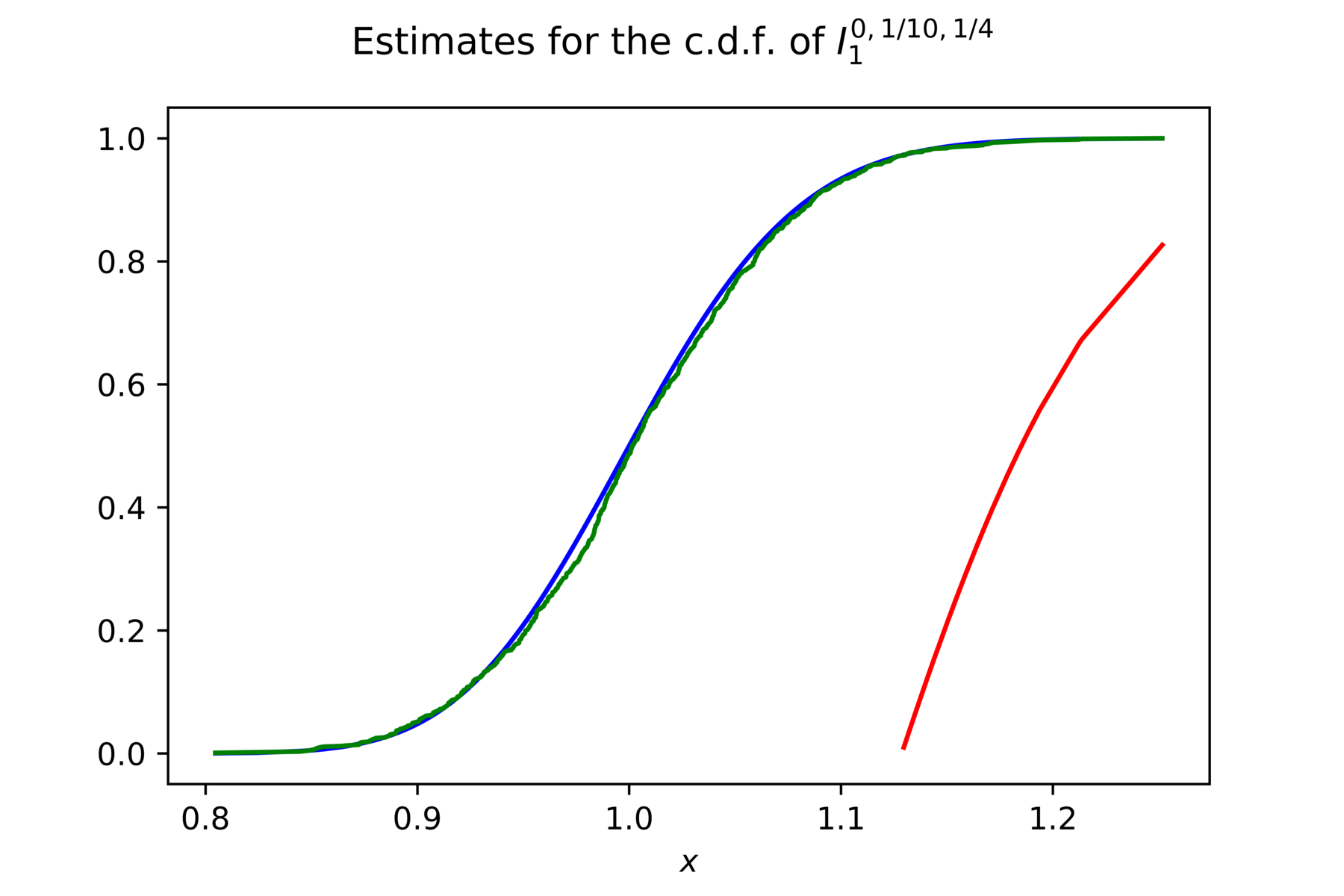}
		\end{subfigure}
		\hfill
		\begin{subfigure}[t]{0.26\textwidth}
			\includegraphics[width=\linewidth]{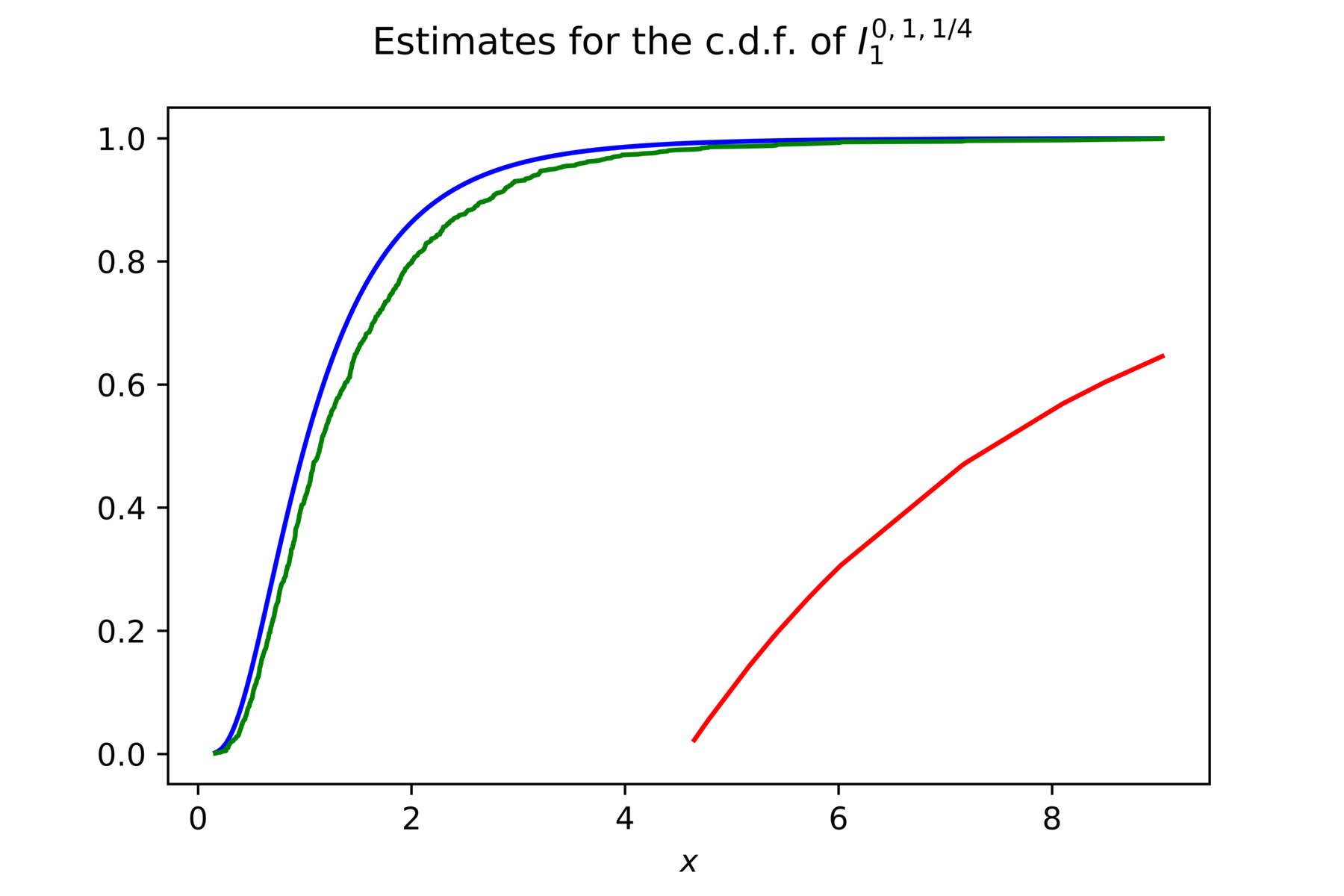}
		\end{subfigure}\hfill
		\begin{subfigure}[t]{0.26\textwidth}
			\includegraphics[width=\textwidth]{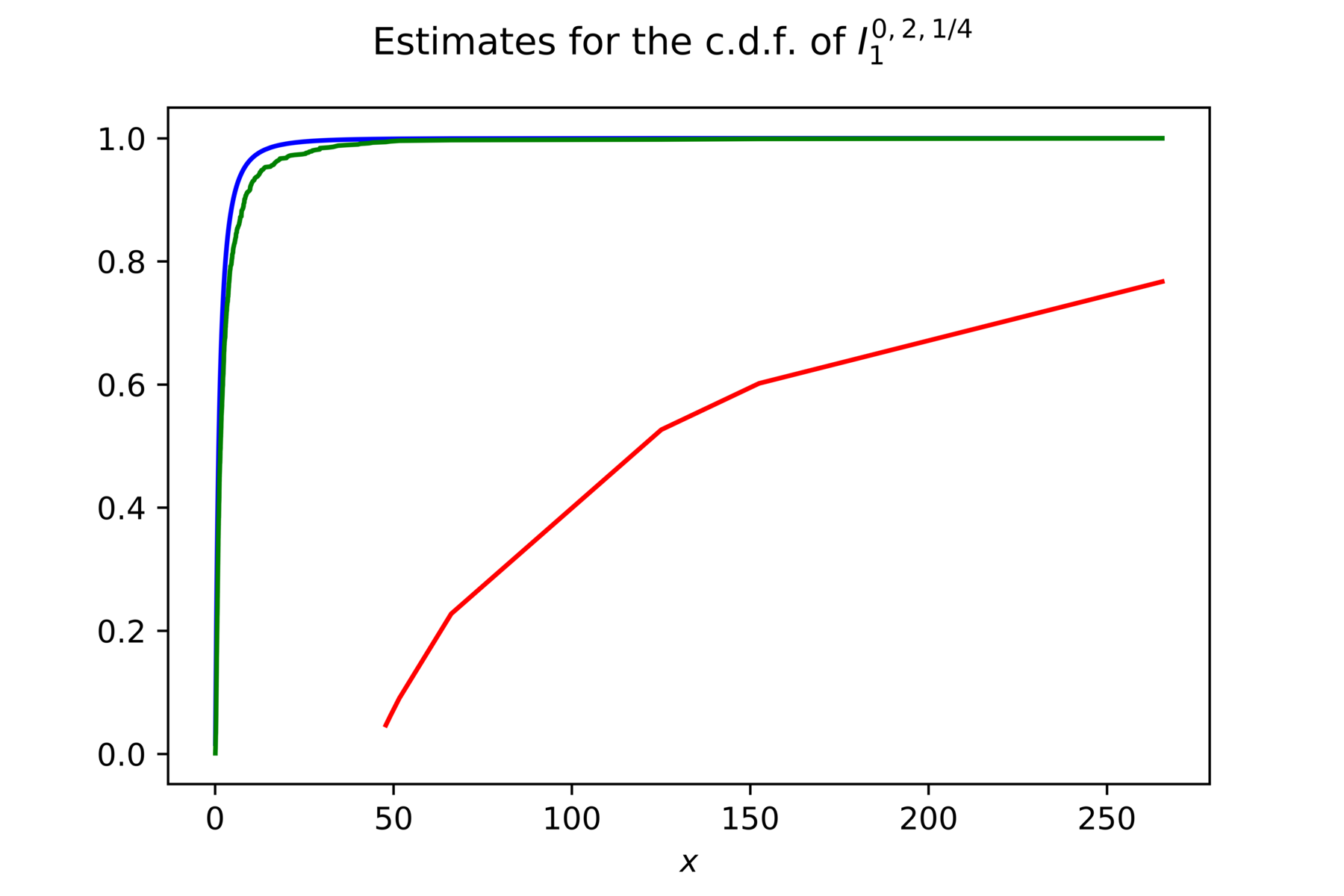}
		\end{subfigure}
		
		\begin{subfigure}[t]{0.26\textwidth}
			\includegraphics[width=\linewidth]{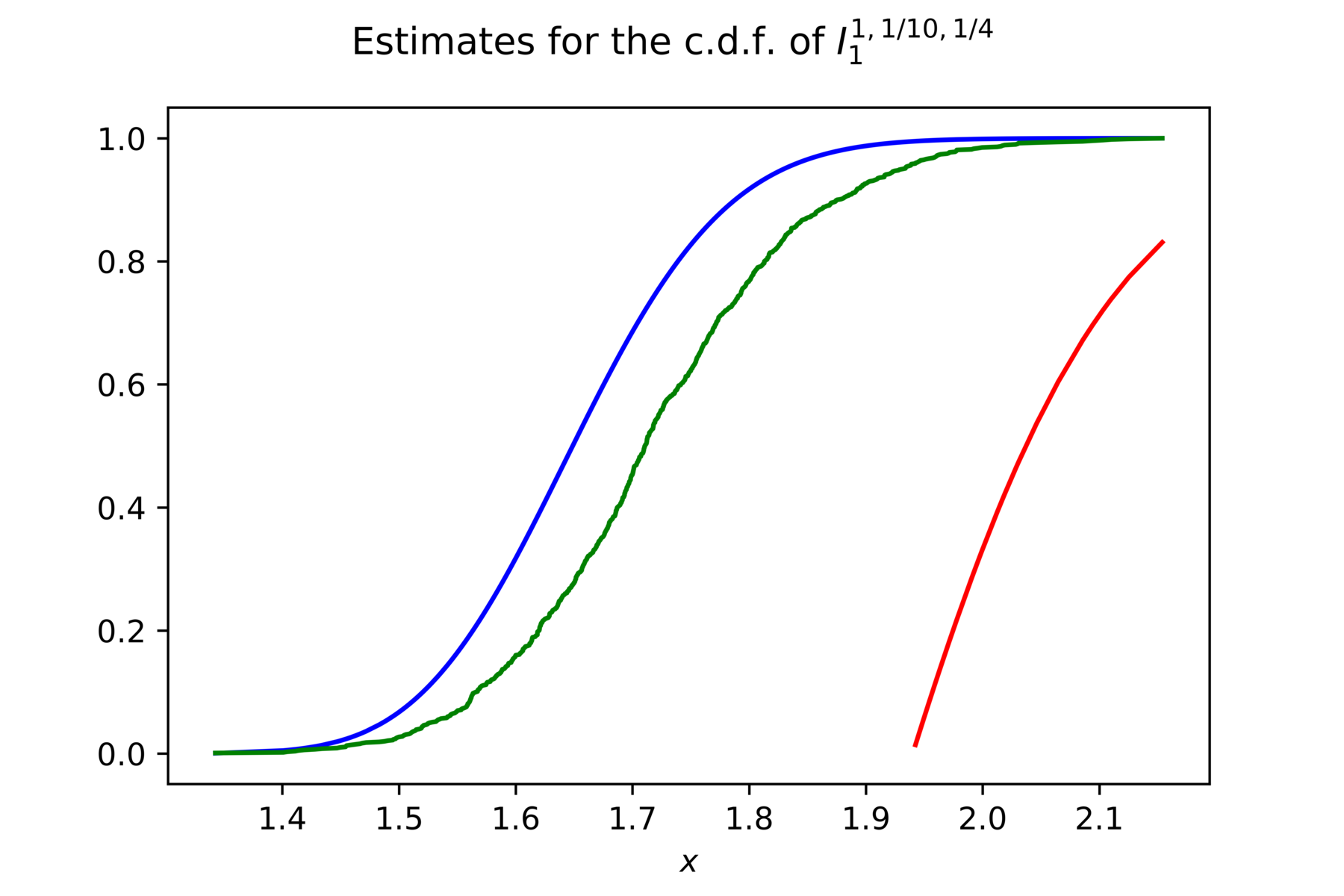}
		\end{subfigure}\hfill
		\begin{subfigure}[t]{0.26\textwidth}
			\includegraphics[width=\linewidth]{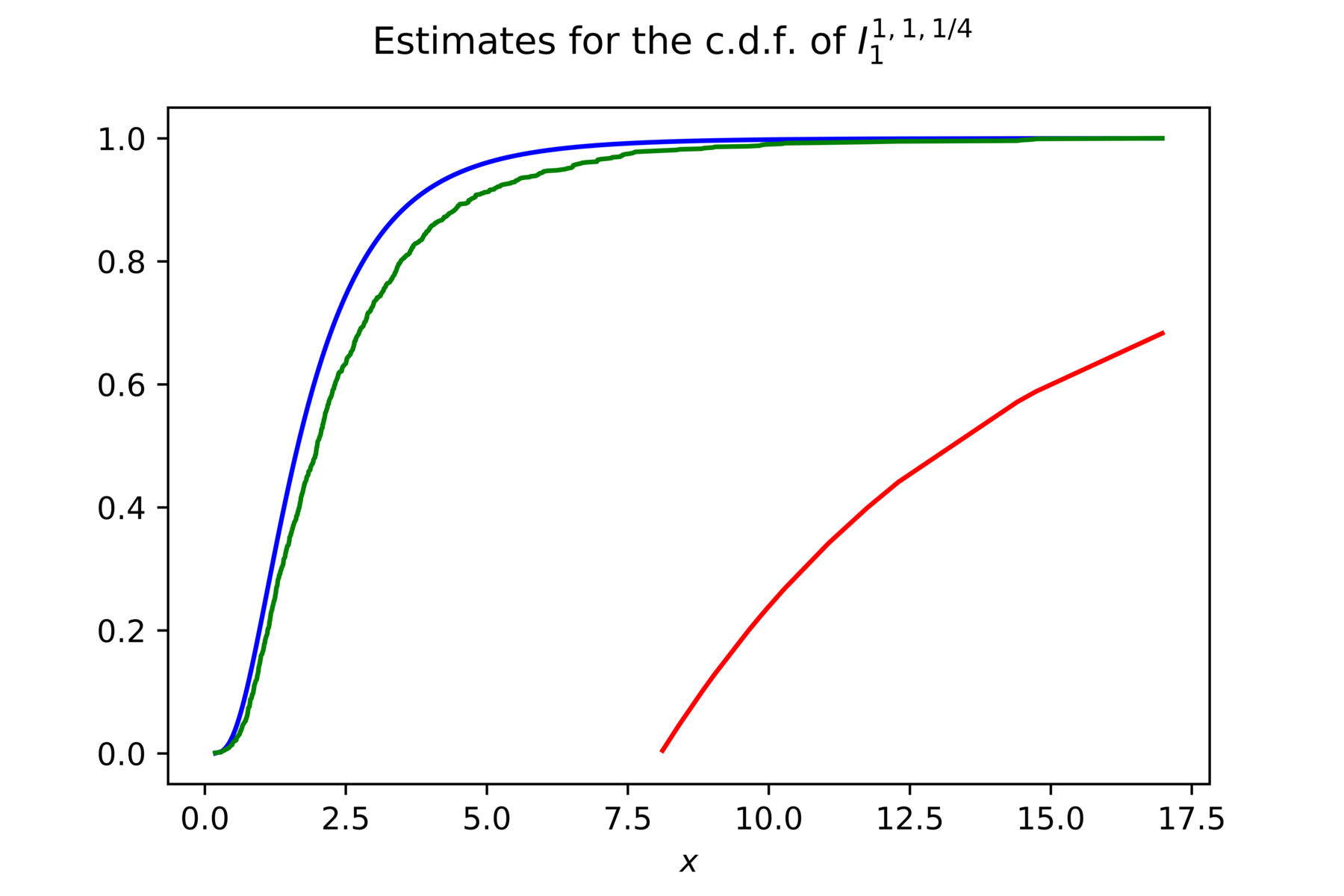}
		\end{subfigure}\hfill
		\begin{subfigure}[t]{0.26\textwidth}
			\includegraphics[width=\linewidth]{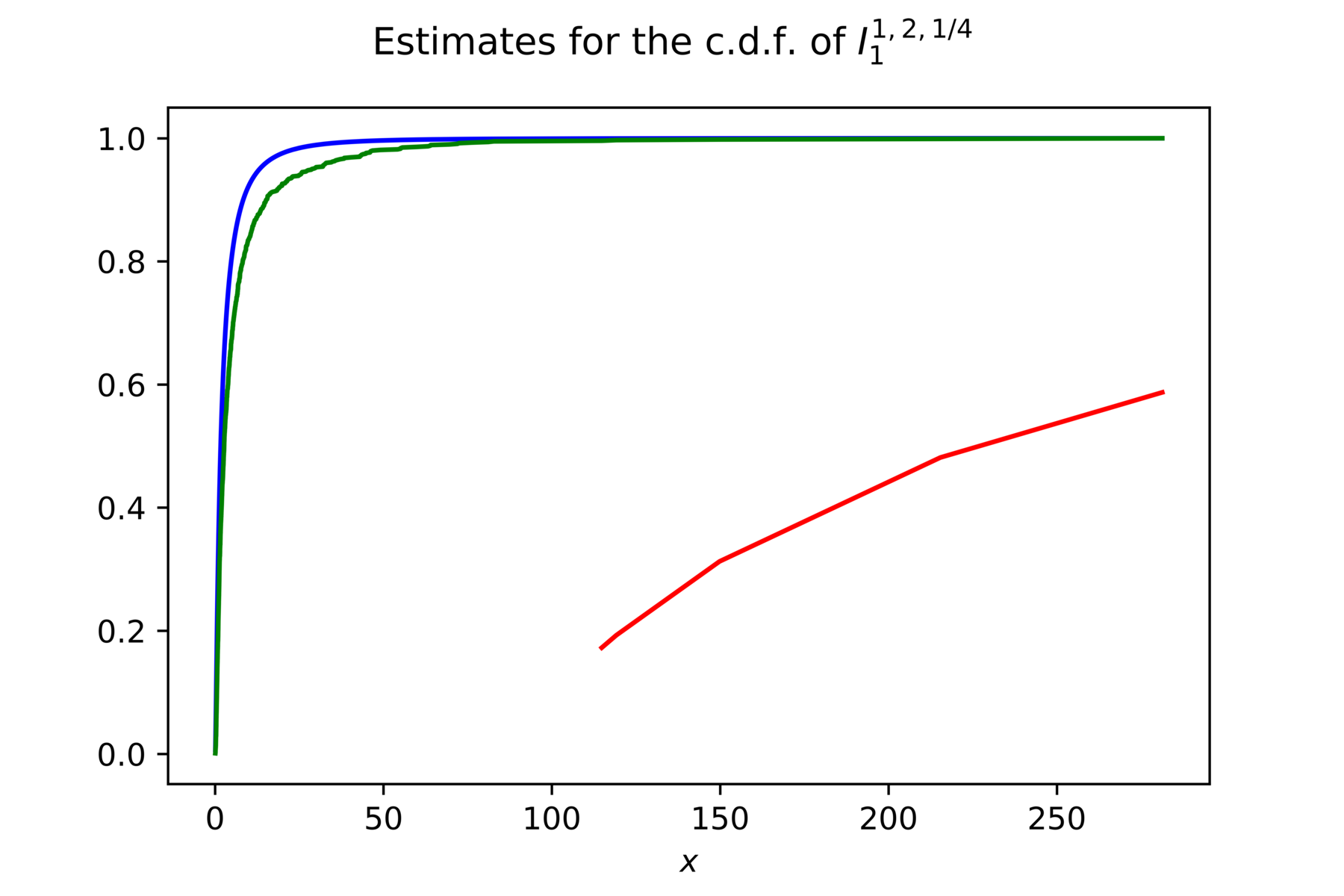}
		\end{subfigure}
		
		\caption{Upper bounds (blue lines) are derived from Theorem \ref{cor-upper-bound-finite} with $\lambda=0$. 
			Lower bounds (red lines) are derived from \cite[Cor.\ 3.1]{dung}.
			The e.c.d.f.\ of $I_T^{\mu,\sigma,H}$ (green lines) was plotted with $1000$ simulations of $I_T^{\mu,\sigma,H}$.}
		\label{figure:finite_1}
	\end{figure}
	
	\begin{figure}[htbp]
		\begin{subfigure}[t]{0.26\textwidth}
			\includegraphics[width=\linewidth]{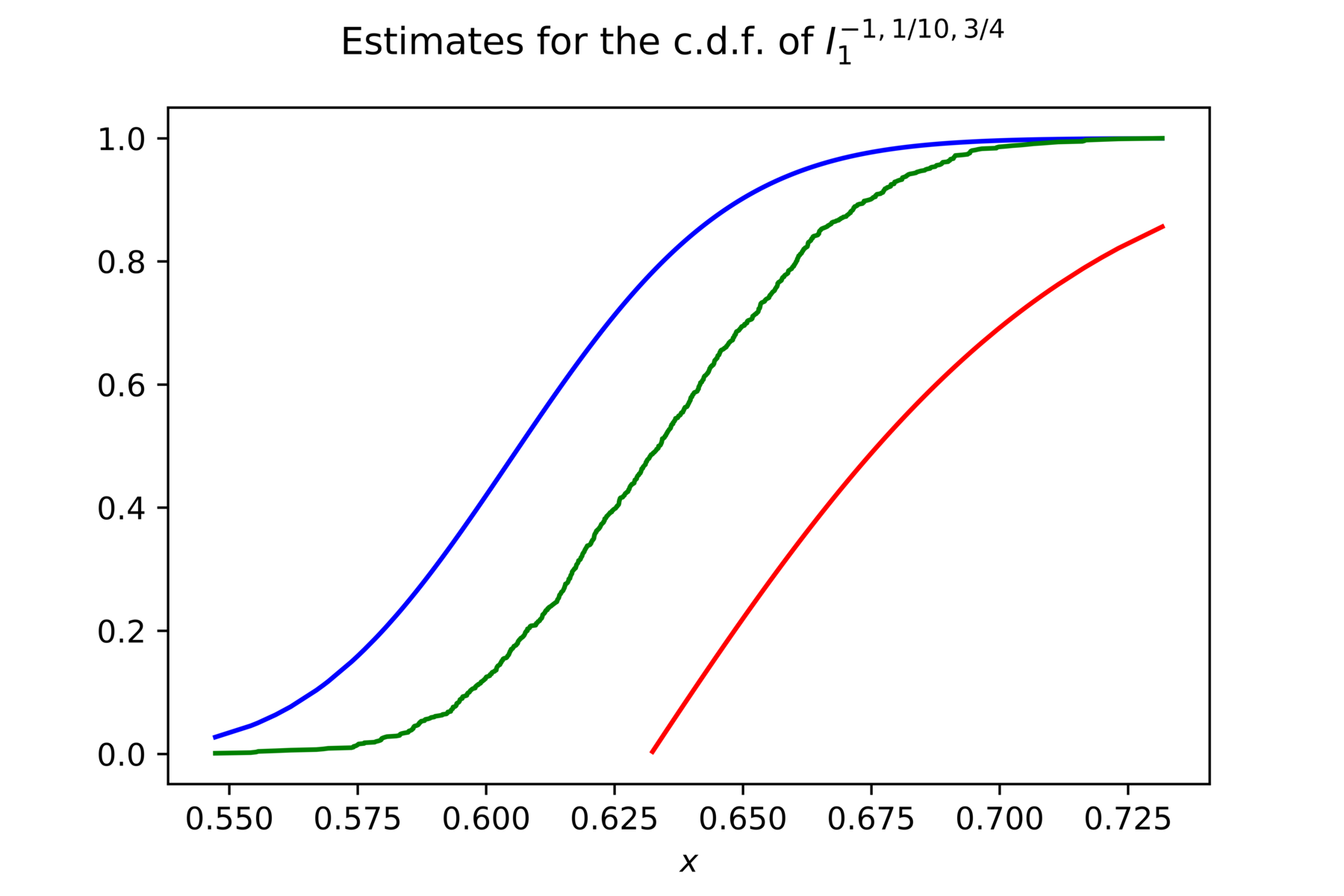}
		\end{subfigure}\hfill
		\begin{subfigure}[t]{0.26\textwidth}
			\includegraphics[width=\linewidth]{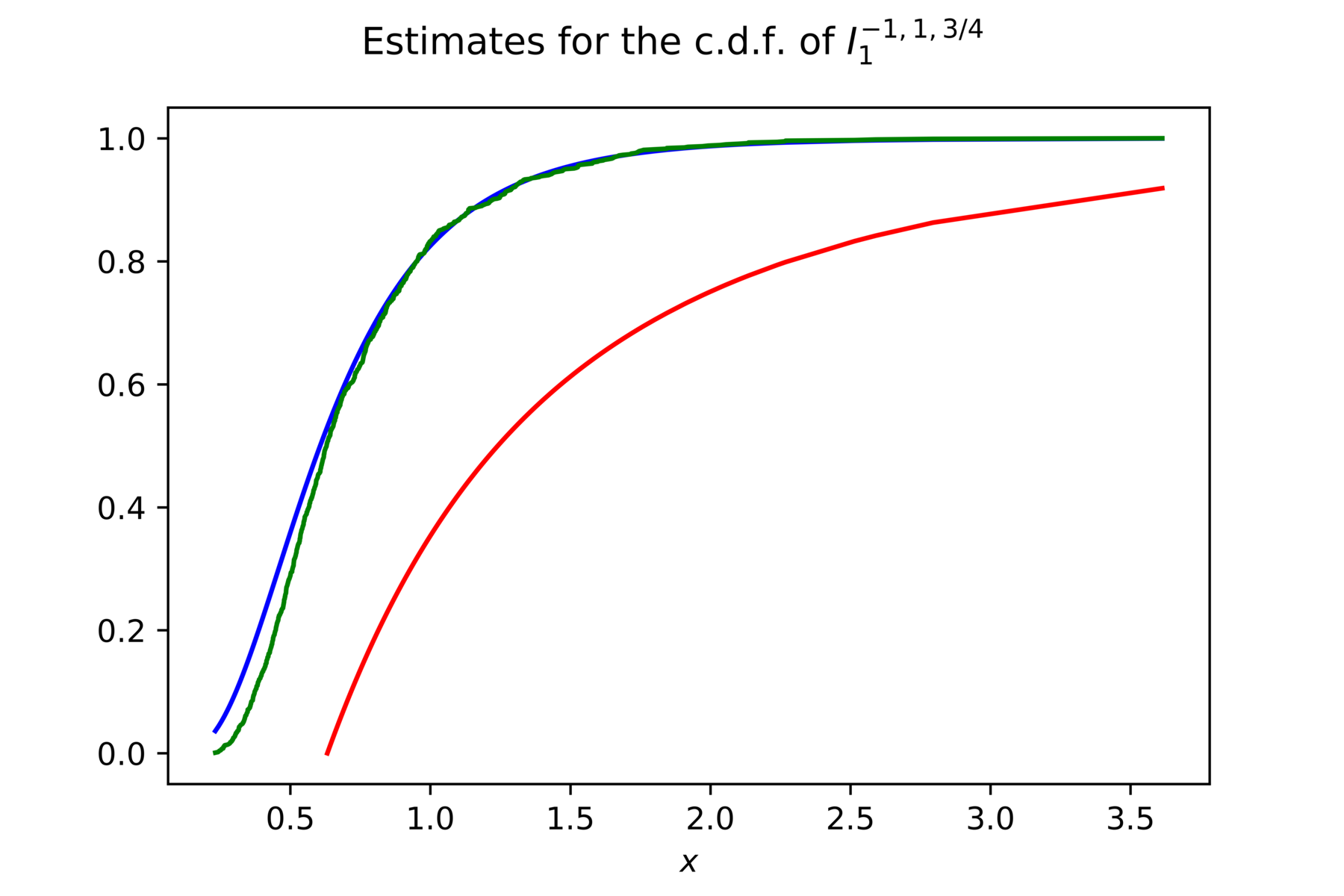}
		\end{subfigure}\hfill
		\begin{subfigure}[t]{0.26\textwidth}
			\includegraphics[width=\linewidth]{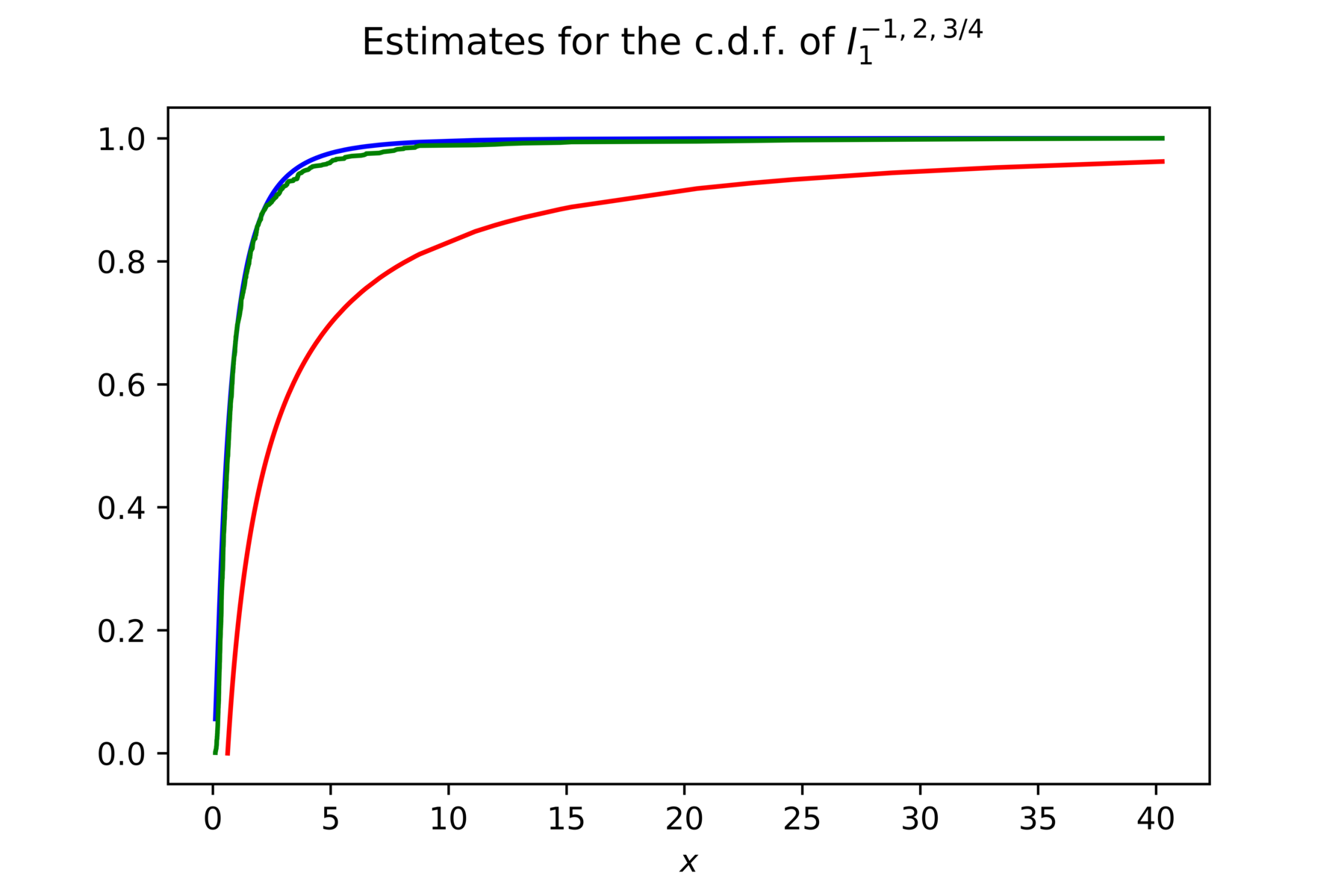}
		\end{subfigure}
		
		\begin{subfigure}[t]{0.26\textwidth}
			\includegraphics[width=\linewidth]{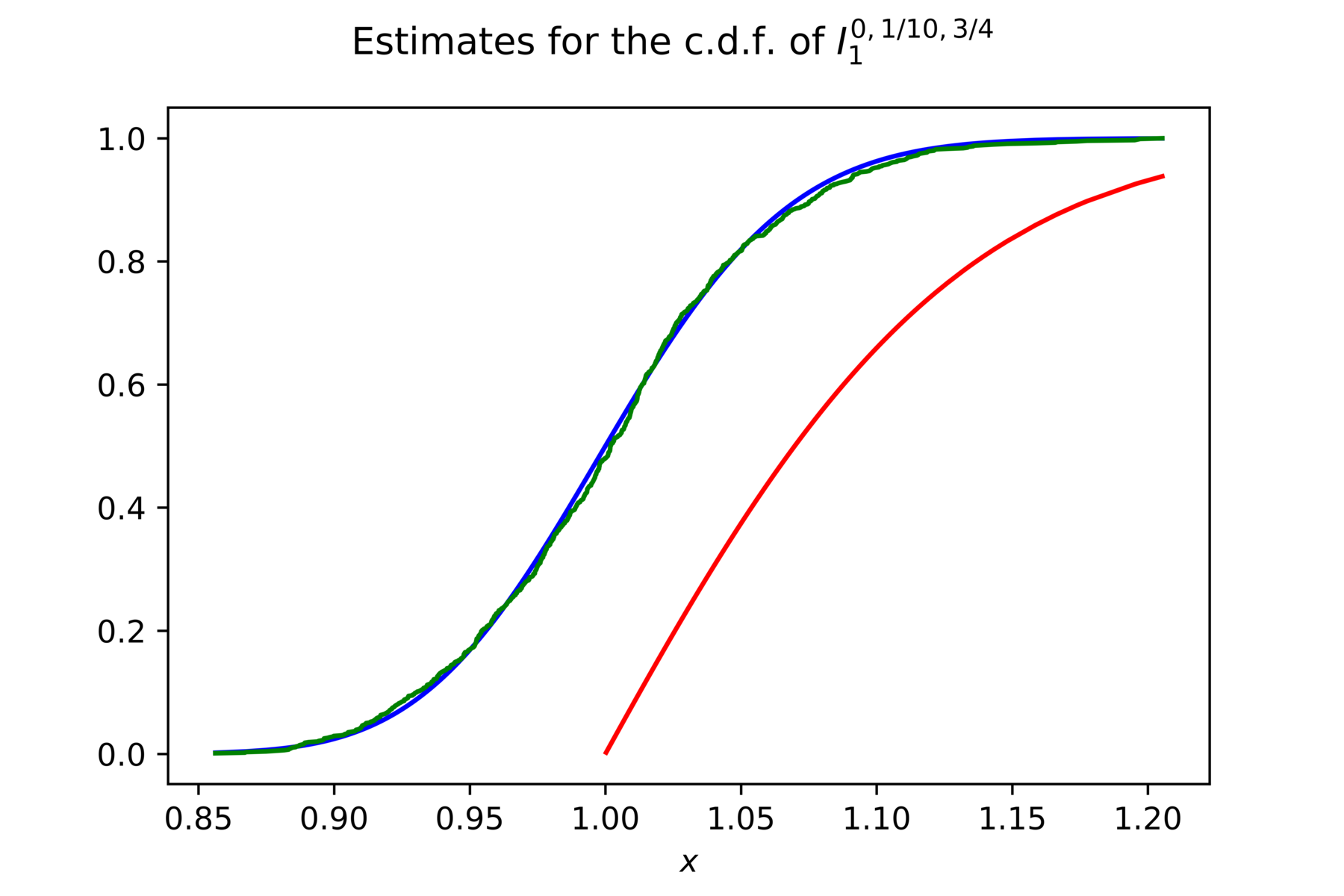}
		\end{subfigure}
		\hfill
		\begin{subfigure}[t]{0.26\textwidth}
			\includegraphics[width=\linewidth]{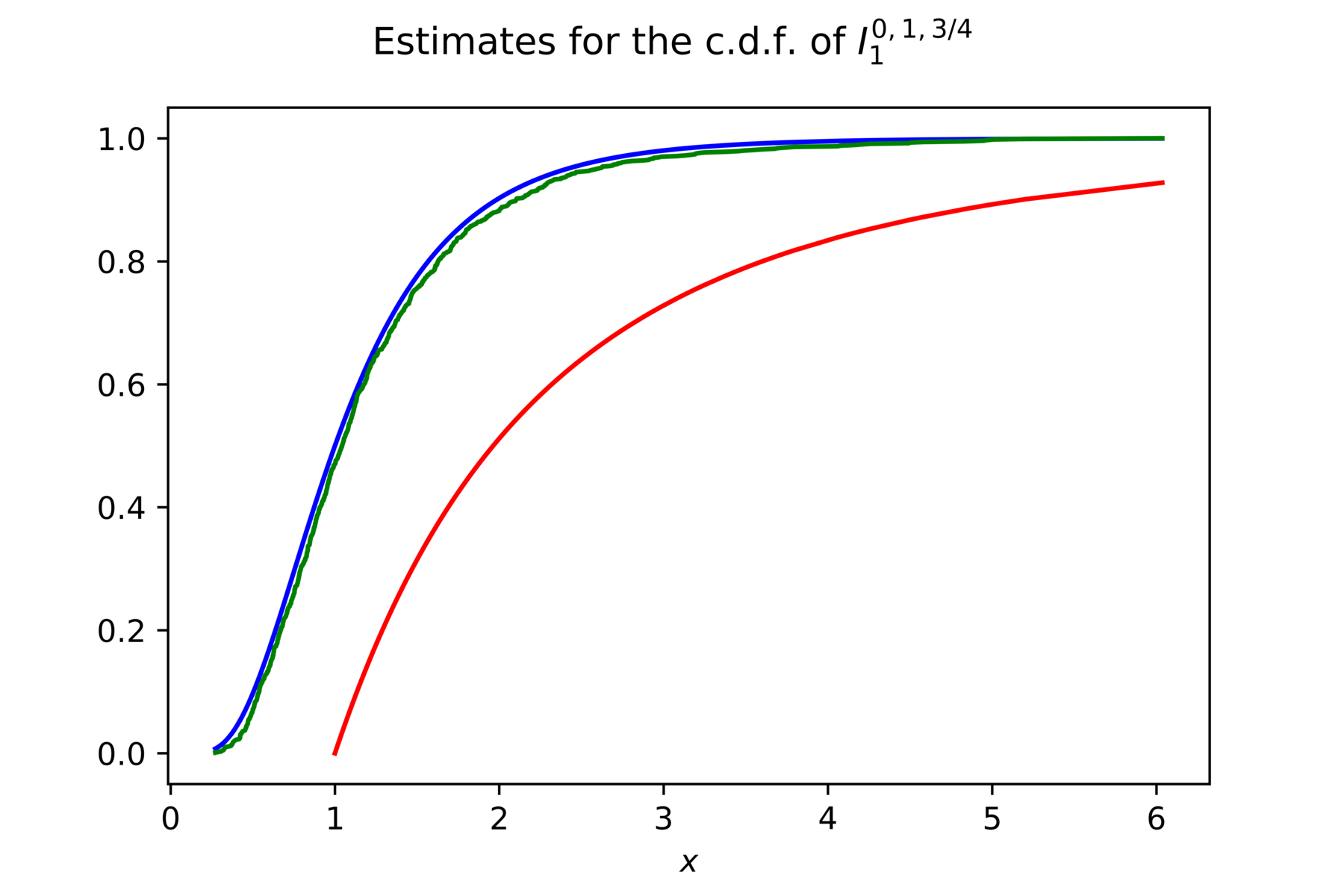}
		\end{subfigure}\hfill
		\begin{subfigure}[t]{0.26\textwidth}
			\includegraphics[width=\textwidth]{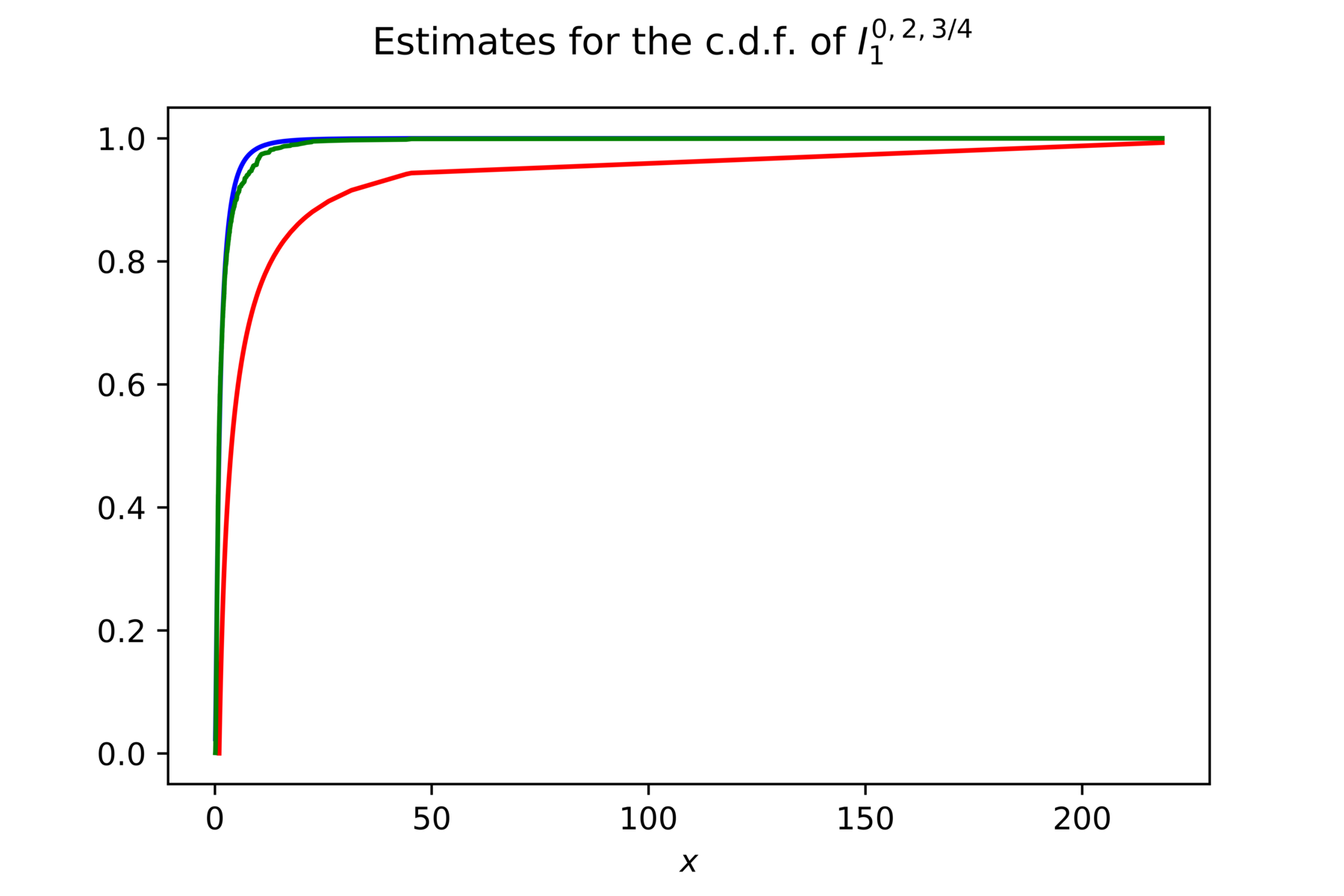}
		\end{subfigure}
		
		\begin{subfigure}[t]{0.26\textwidth}
			\includegraphics[width=\linewidth]{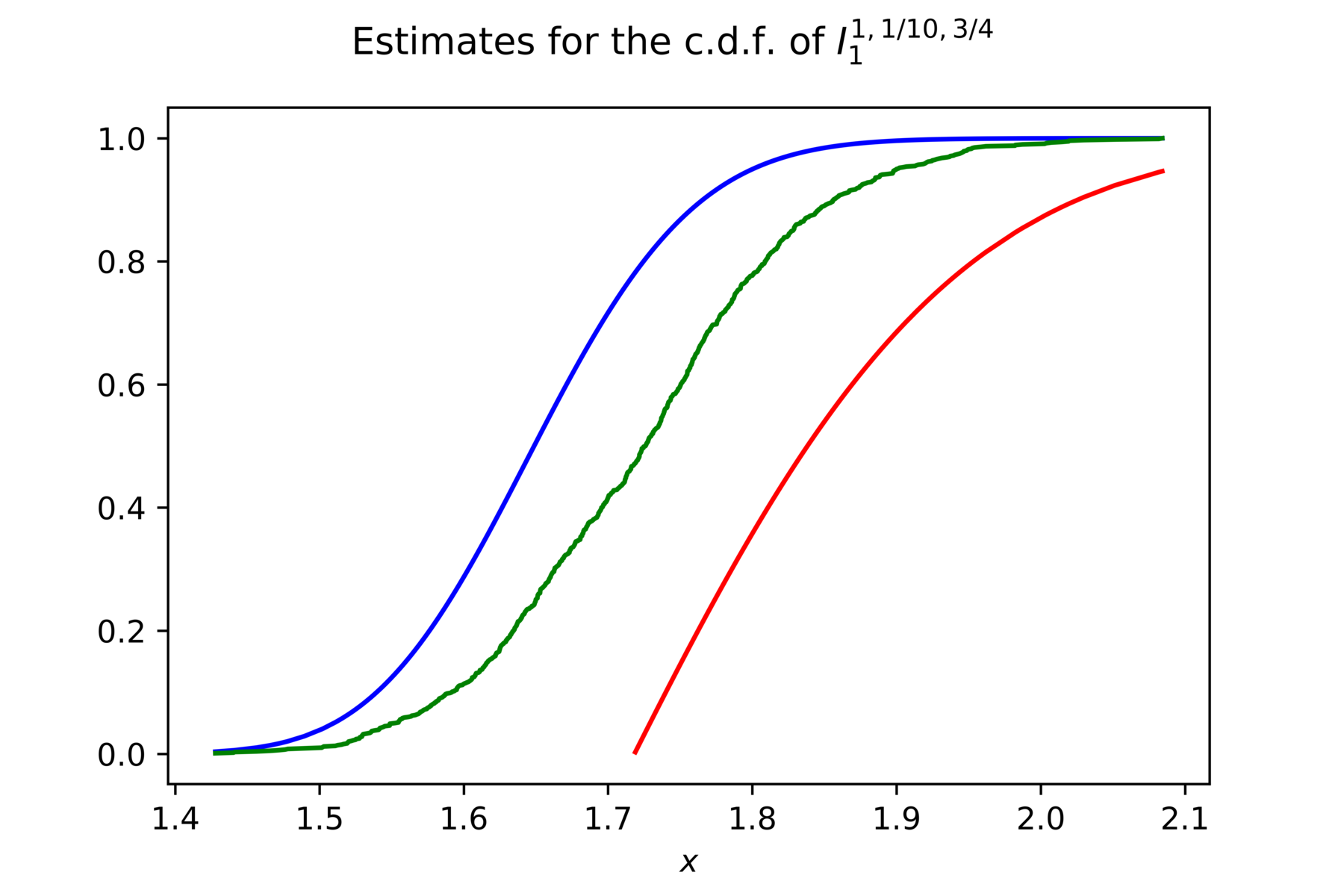}
		\end{subfigure}\hfill
		\begin{subfigure}[t]{0.26\textwidth}
			\includegraphics[width=\linewidth]{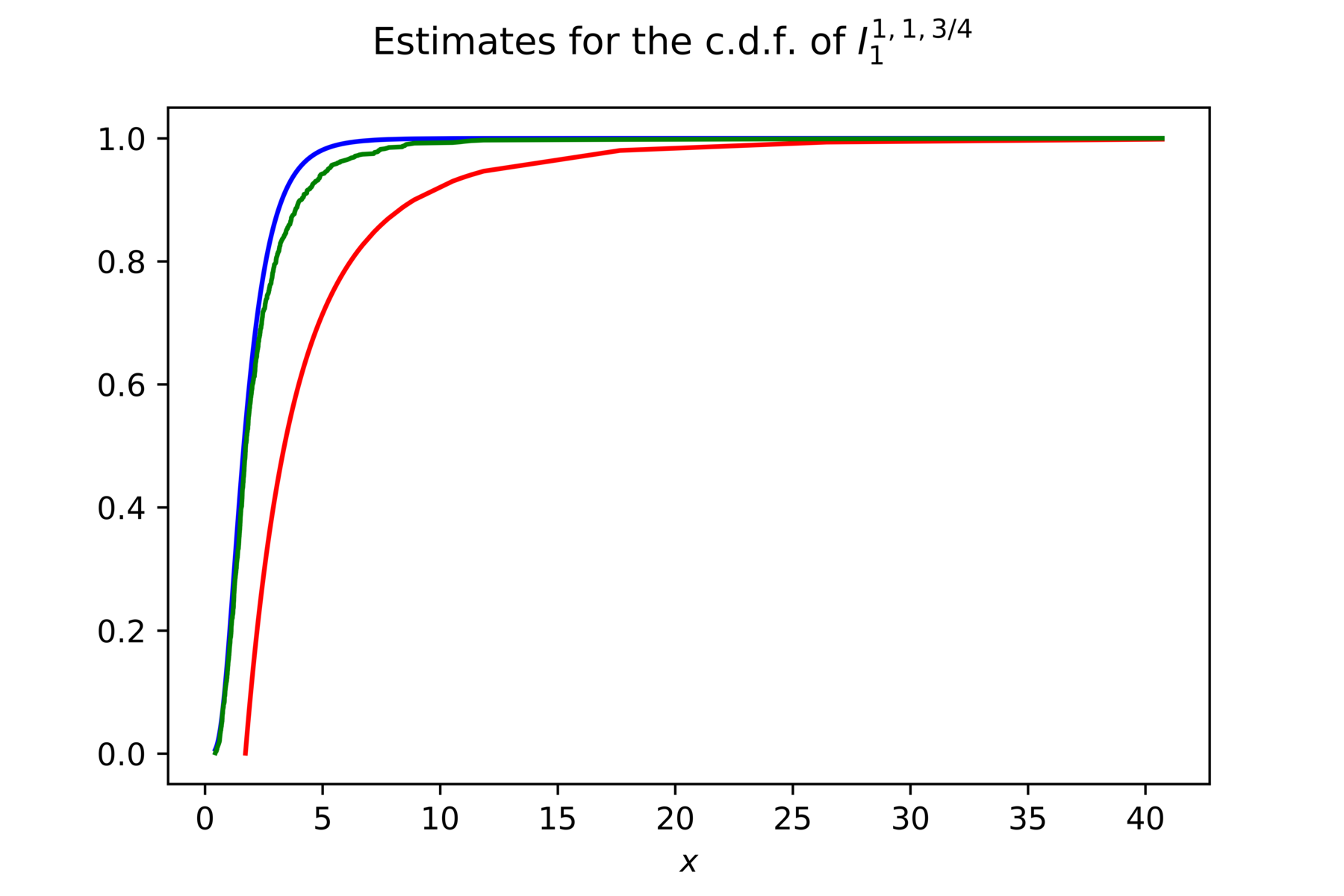}
		\end{subfigure}\hfill
		\begin{subfigure}[t]{0.26\textwidth}
			\includegraphics[width=\linewidth]{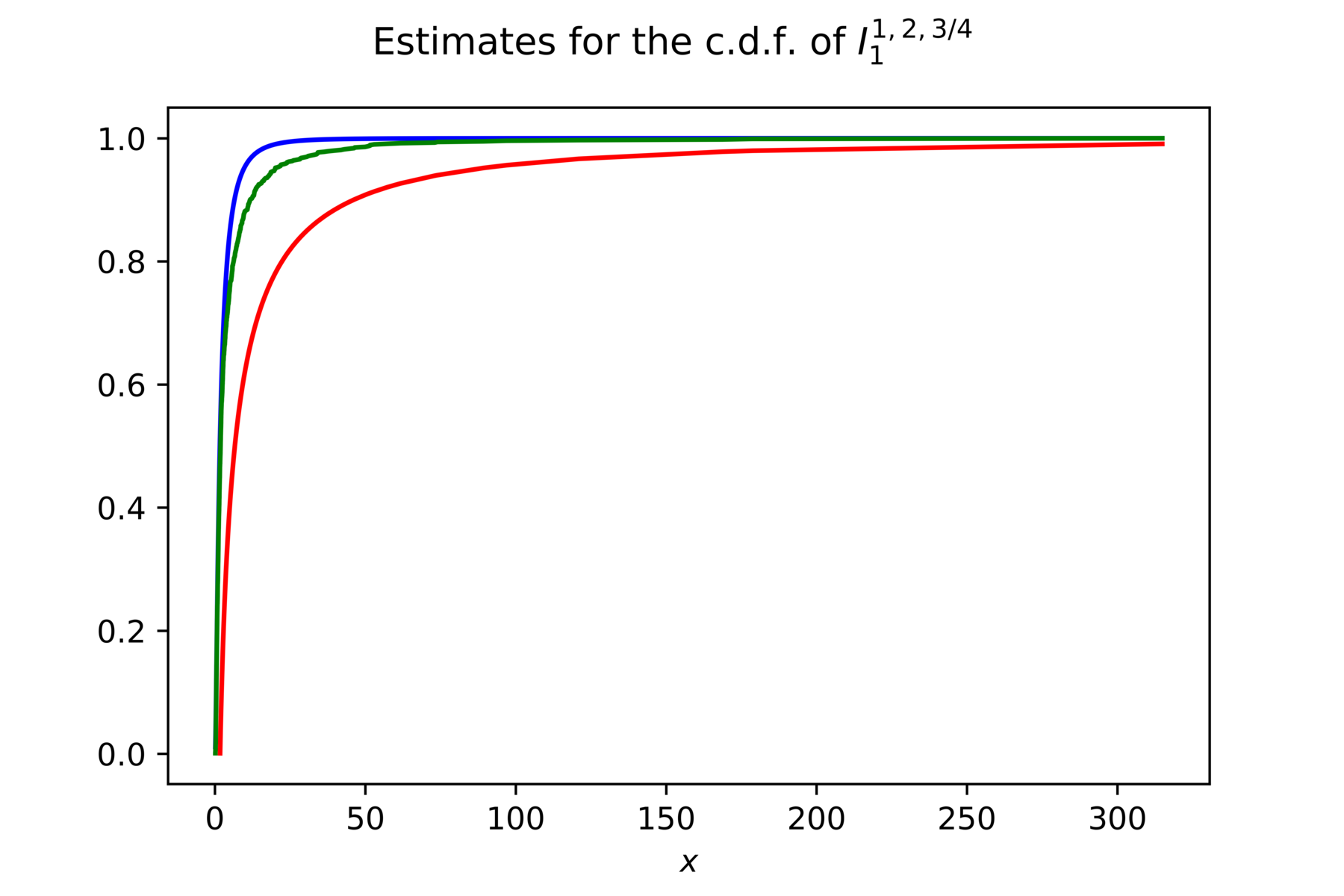}
		\end{subfigure}
		
		\caption{Upper bounds (blue lines) are derived from Theorem \ref{cor-upper-bound-finite} with $\lambda=0$. 
			Lower bounds (red lines) are derived from Theorem  \ref{lower-bound-finite-ii} with $\lambda=0$.
			The e.c.d.f.\ of $I_T^{\mu,\sigma,H}$ (green lines) was plotted with  $1000$ simulations of $I_T^{\mu,\sigma,H}$.}
		\label{figure:finite_2}
	\end{figure}

	\begin{figure}[!htbp]
		\begin{subfigure}[t]{0.26\textwidth}
			\includegraphics[width=\linewidth]{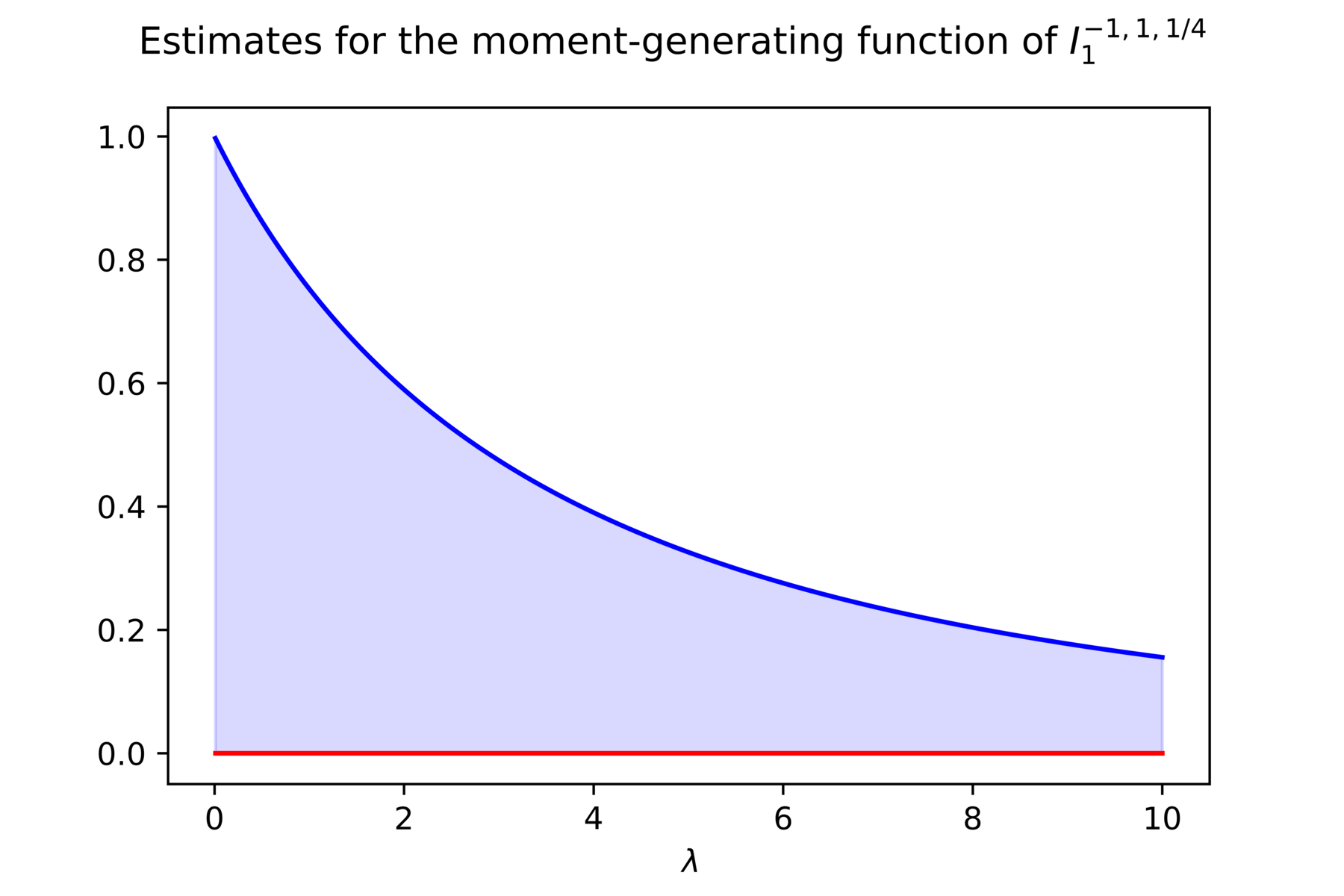}
		\end{subfigure}\hfill
		\begin{subfigure}[t]{0.26\textwidth}
			\includegraphics[width=\linewidth]{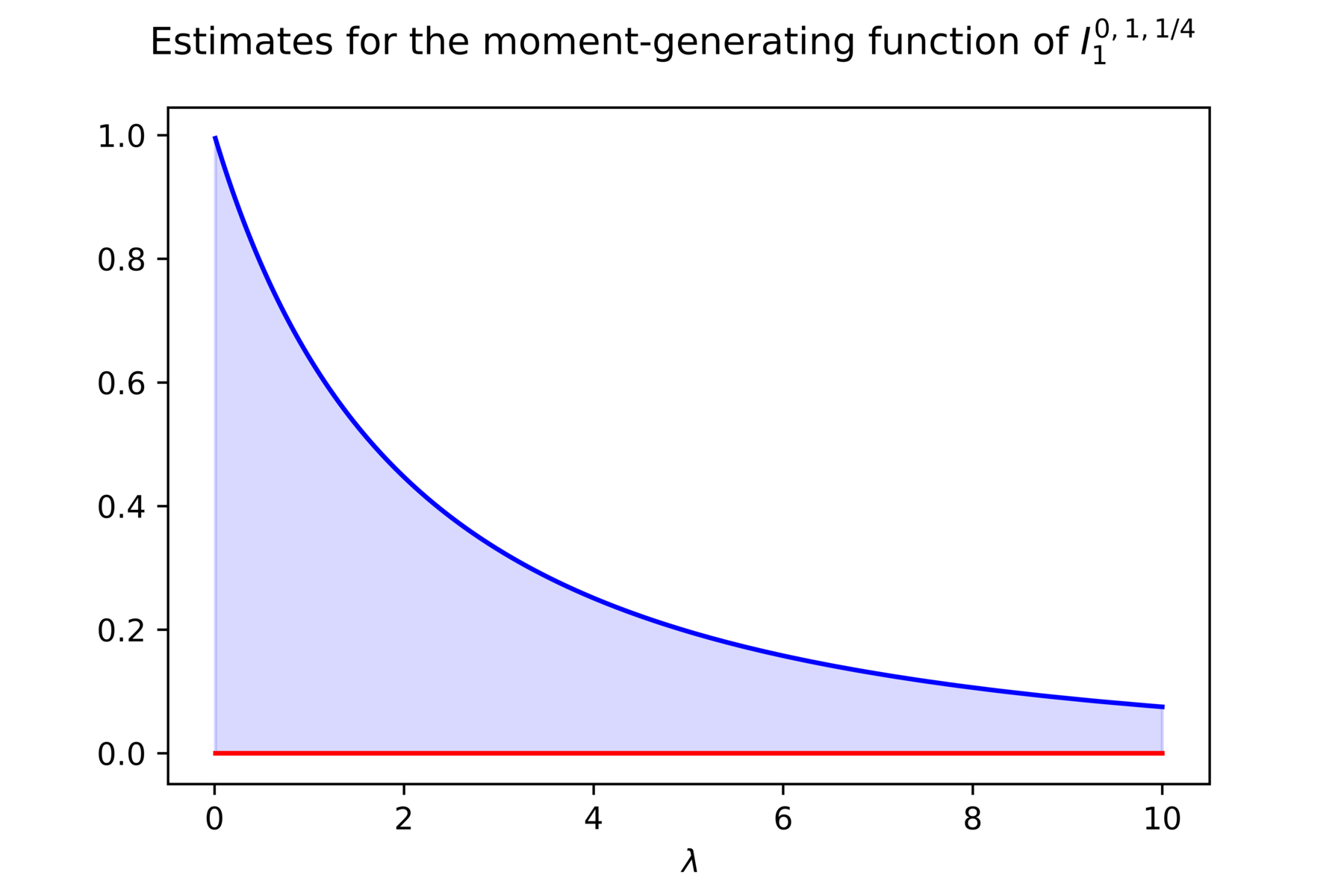}
		\end{subfigure}\hfill
		\begin{subfigure}[t]{0.26\textwidth}
			\includegraphics[width=\linewidth]{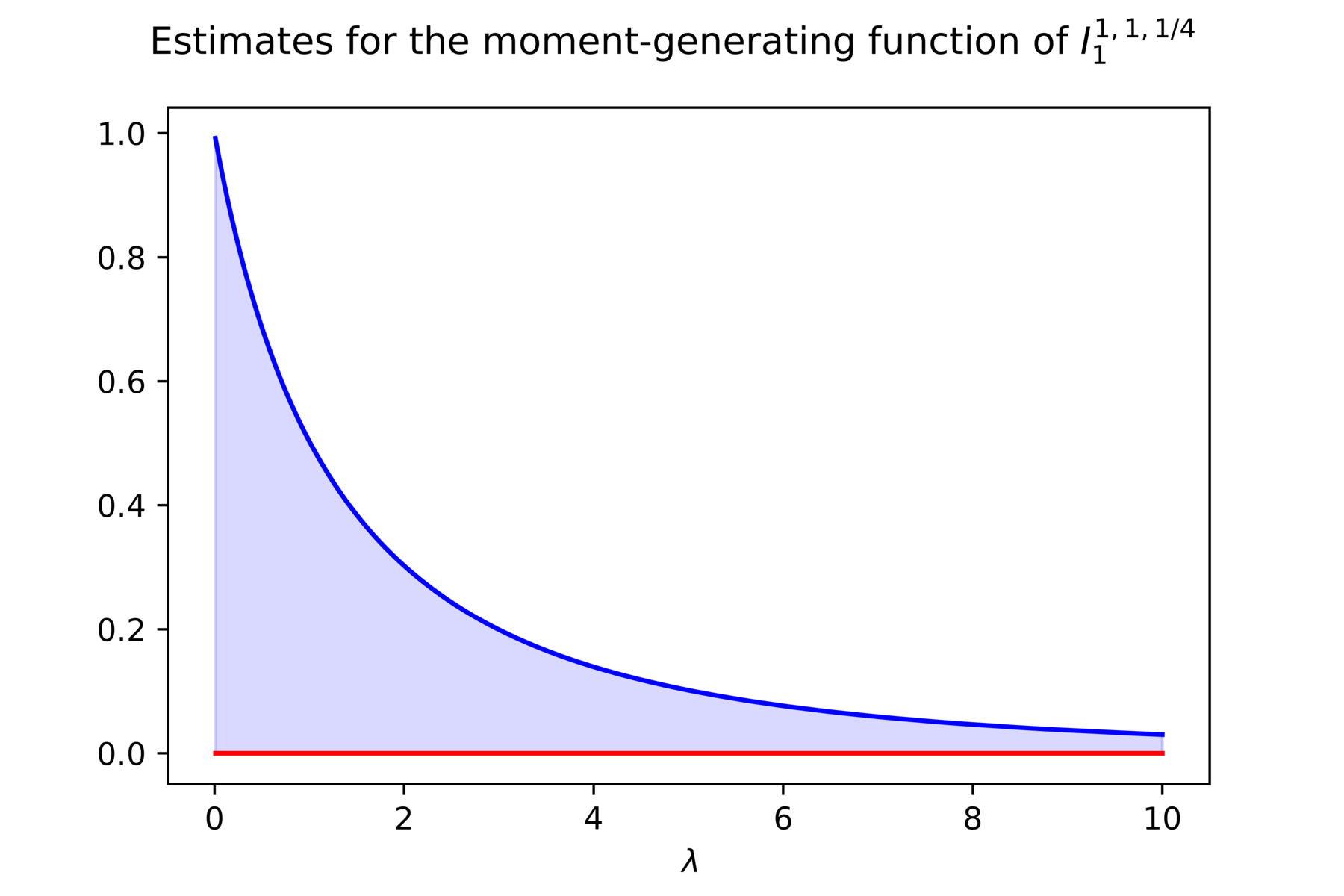}
		\end{subfigure}
	
		\begin{subfigure}[t]{0.26\textwidth}
			\includegraphics[width=\linewidth]{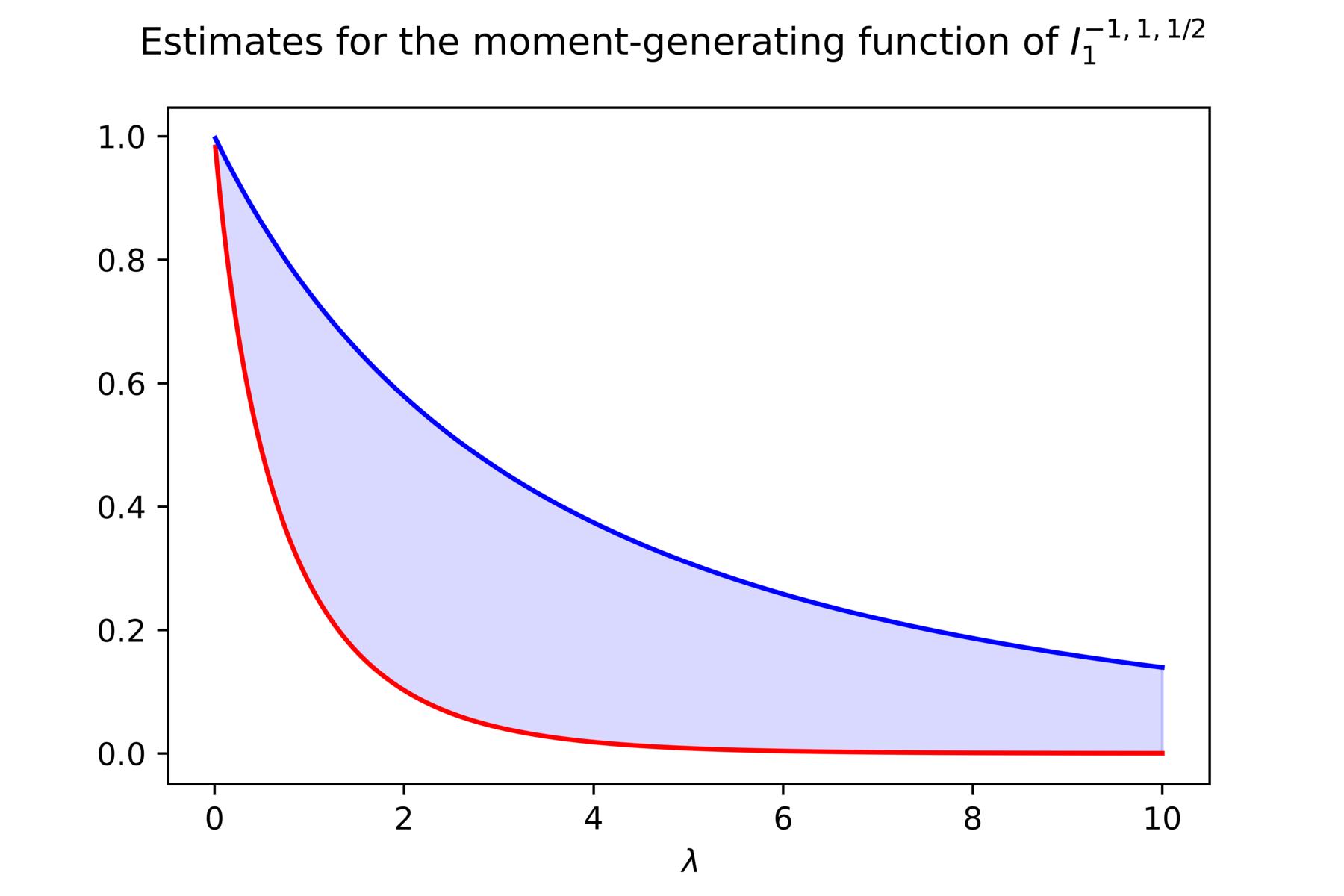}
		\end{subfigure}\hfill
		\begin{subfigure}[t]{0.26\textwidth}
			\includegraphics[width=\linewidth]{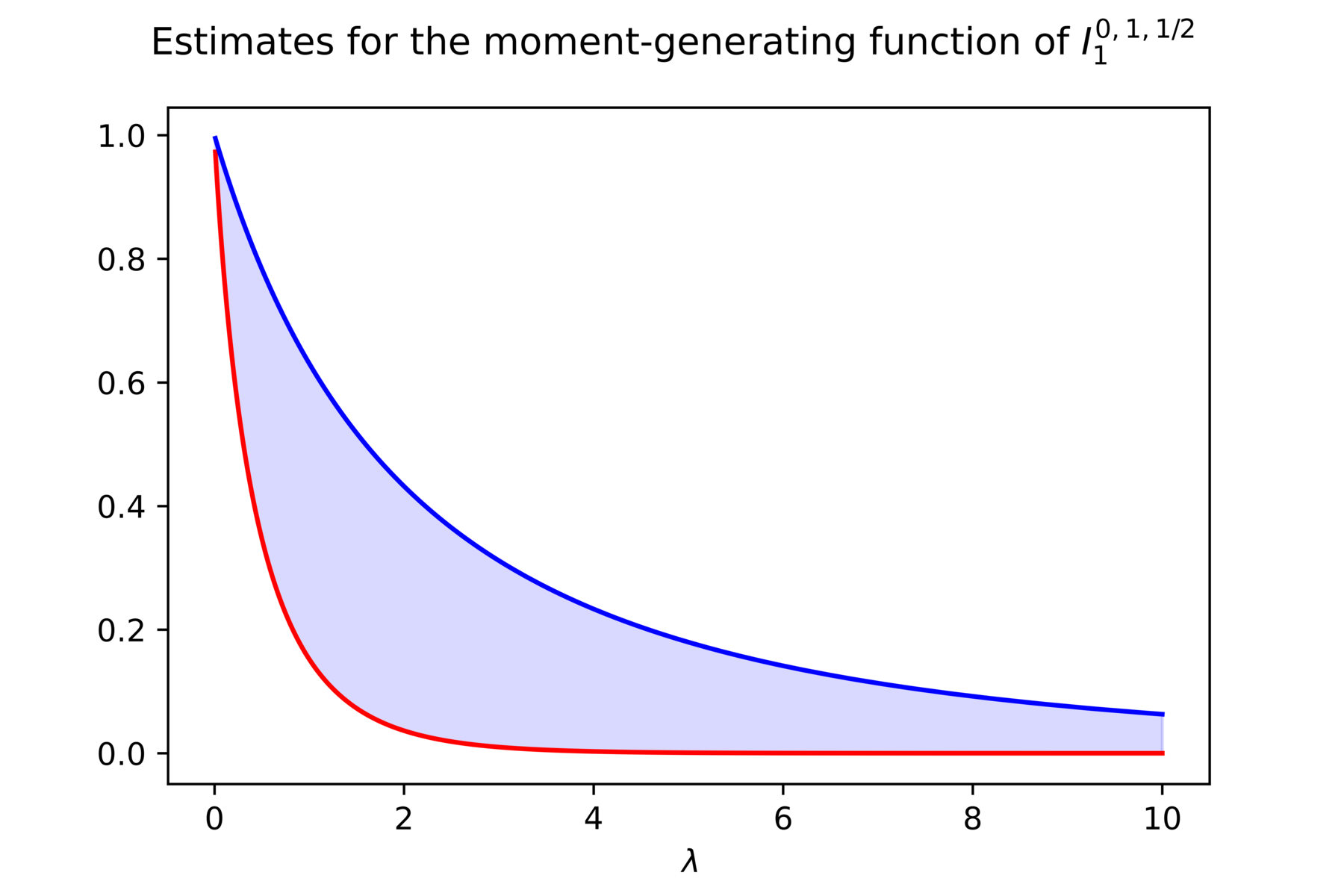}
		\end{subfigure}\hfill
		\begin{subfigure}[t]{0.26\textwidth}
			\includegraphics[width=\linewidth]{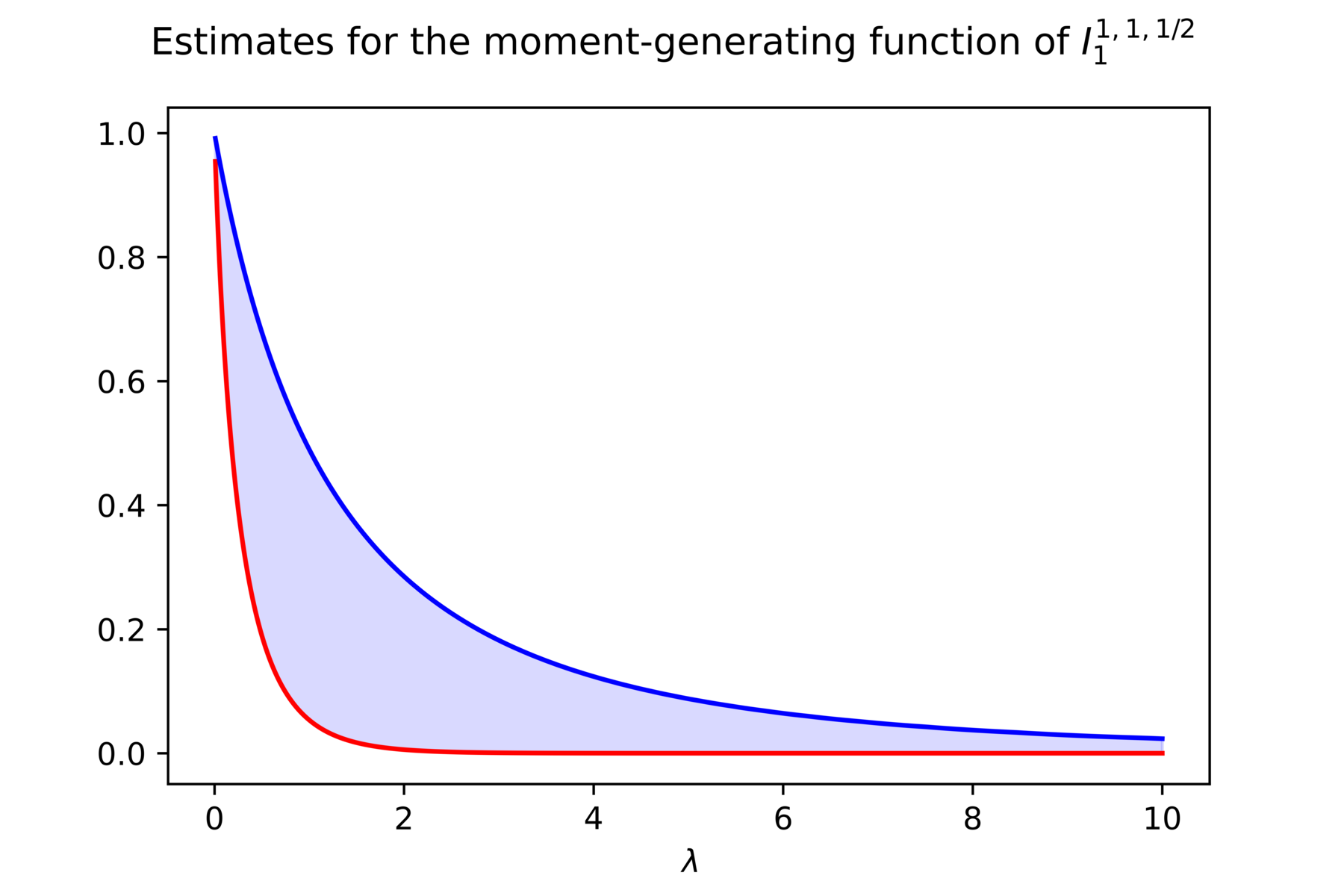}
		\end{subfigure}

	\begin{subfigure}[t]{0.26\textwidth}
		\includegraphics[width=\linewidth]{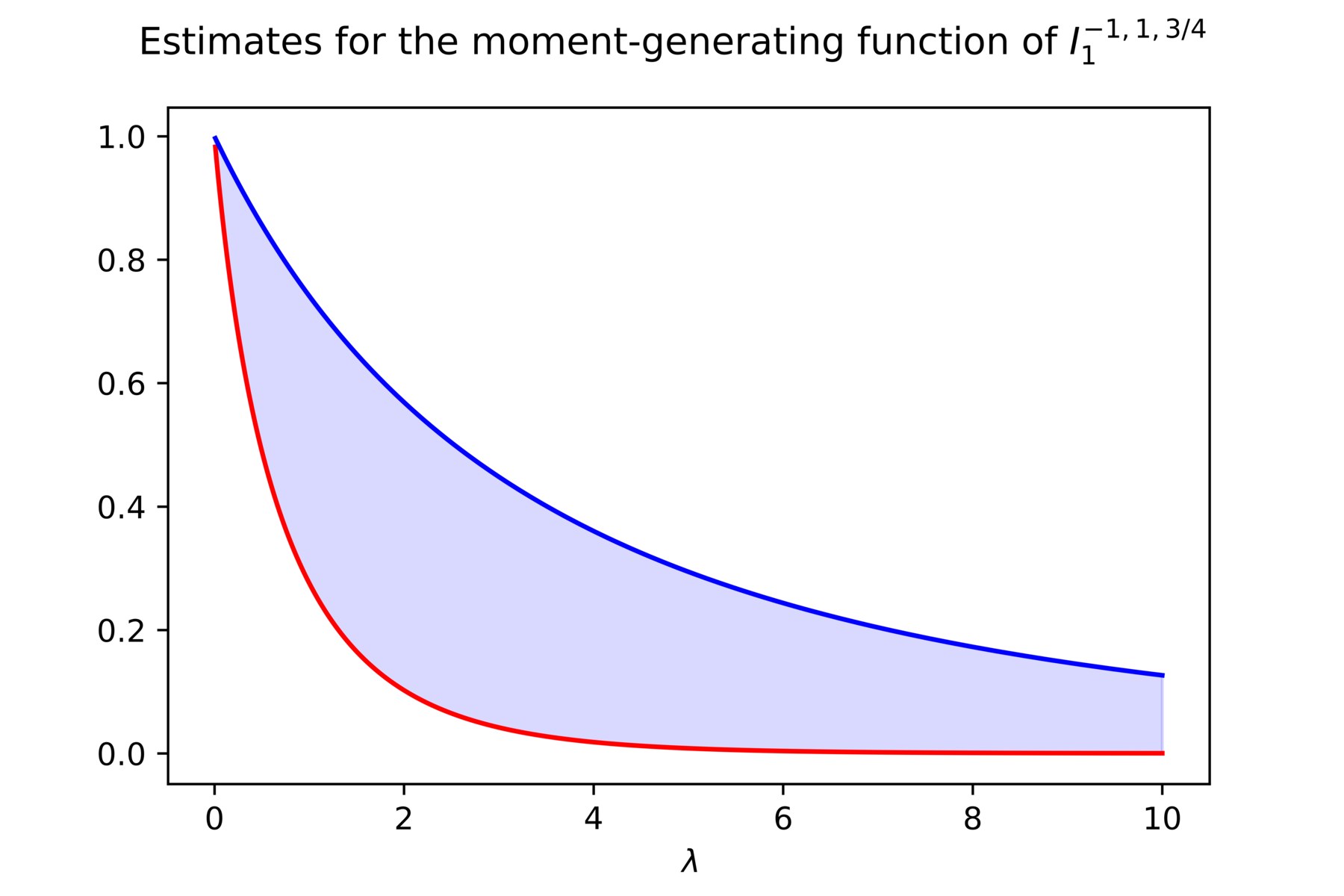}
	\end{subfigure}\hfill
	\begin{subfigure}[t]{0.26\textwidth}
		\includegraphics[width=\linewidth]{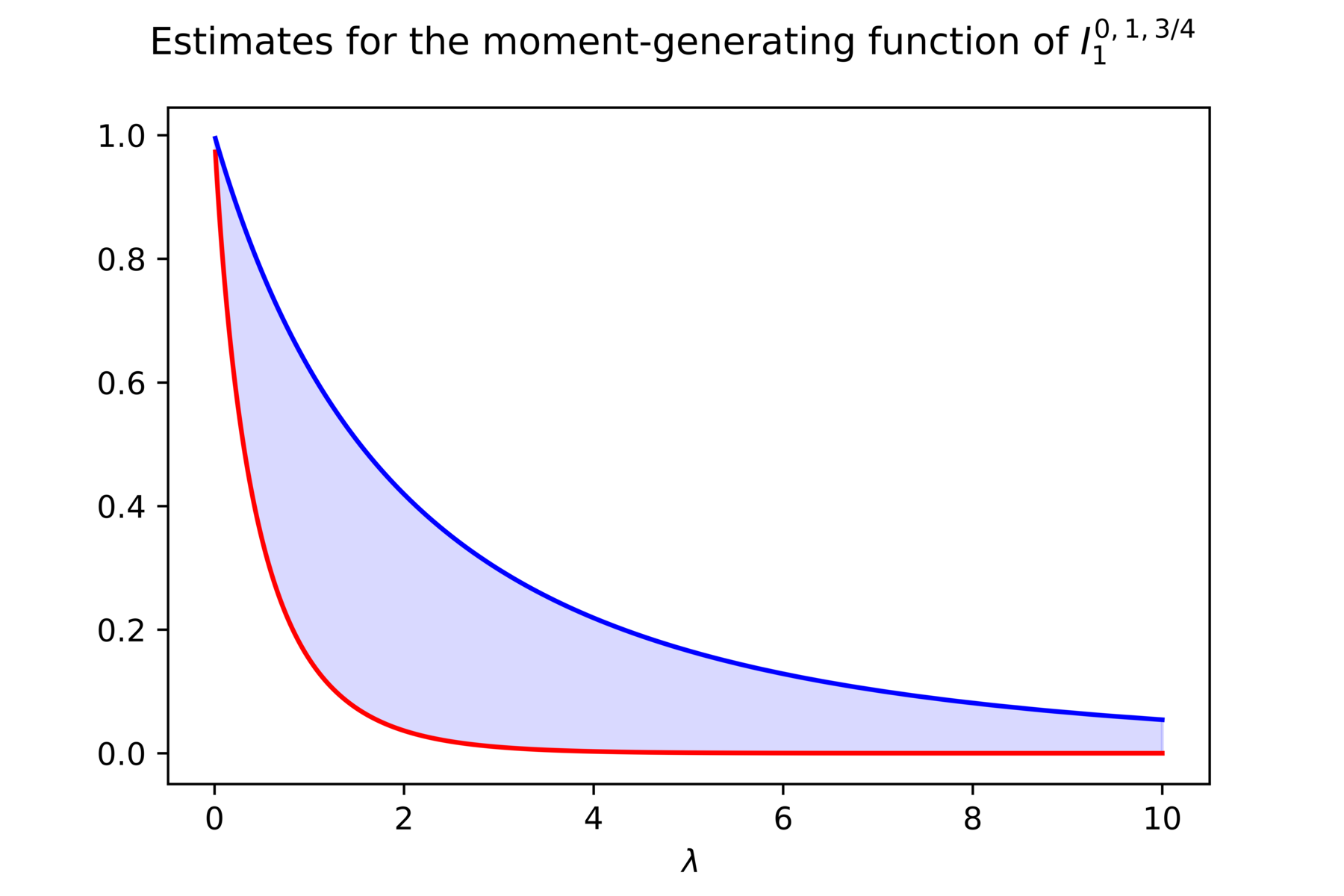}
	\end{subfigure}\hfill
	\begin{subfigure}[t]{0.26\textwidth}
		\includegraphics[width=\linewidth]{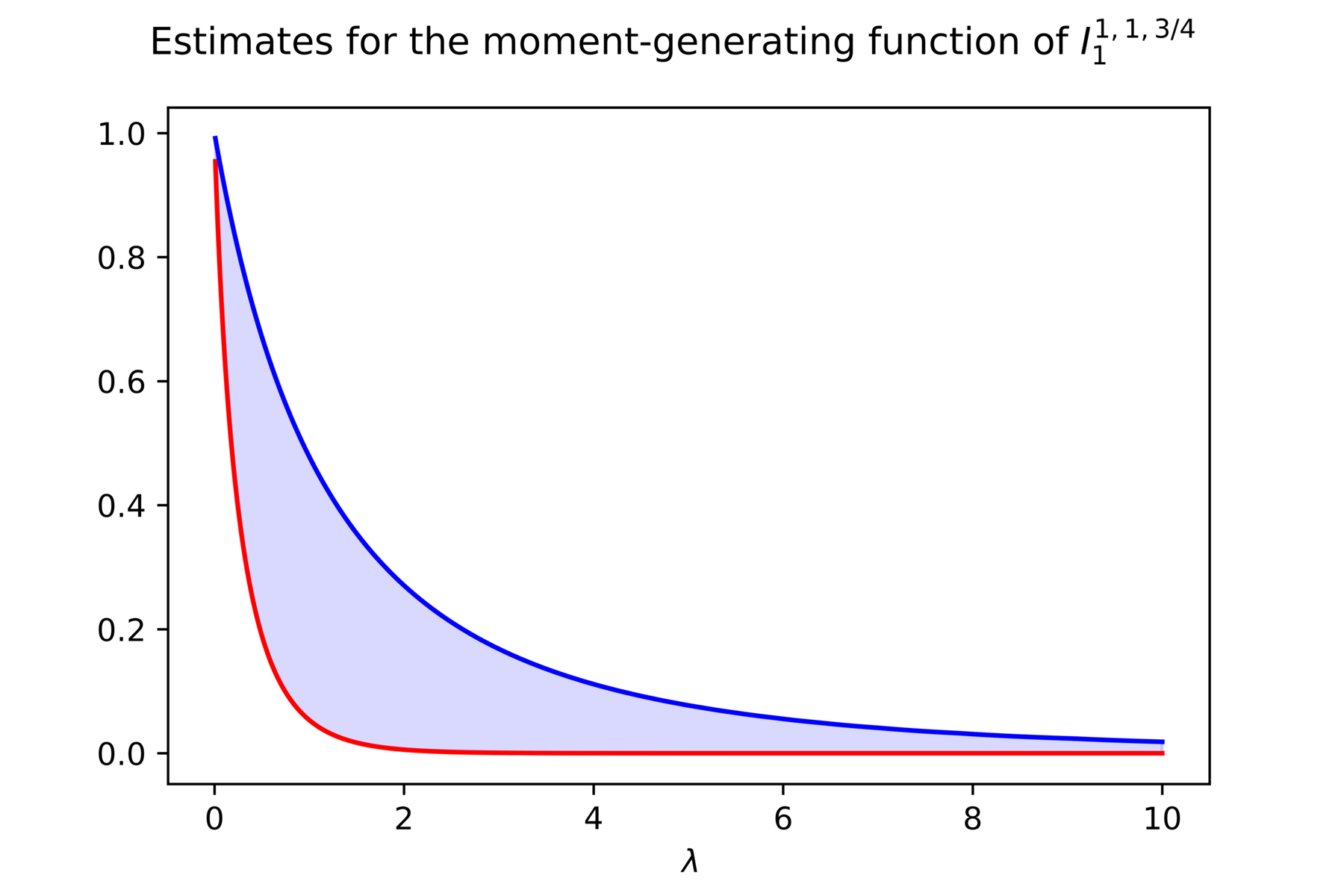}
	\end{subfigure}
	
		\caption{Upper bounds (blue lines) and  lower bounds (red lines) for $\mathbb{E}[\exp(-\lambda I_T^{\mu,\sigma,H})]$ are derived from Corollary \ref{cor-gmf-finite}.}
		\label{figure:finite_3}
	\end{figure}

	\begin{figure}[!htbp]
		\begin{subfigure}[t]{0.26\textwidth}
			\includegraphics[width=\linewidth]{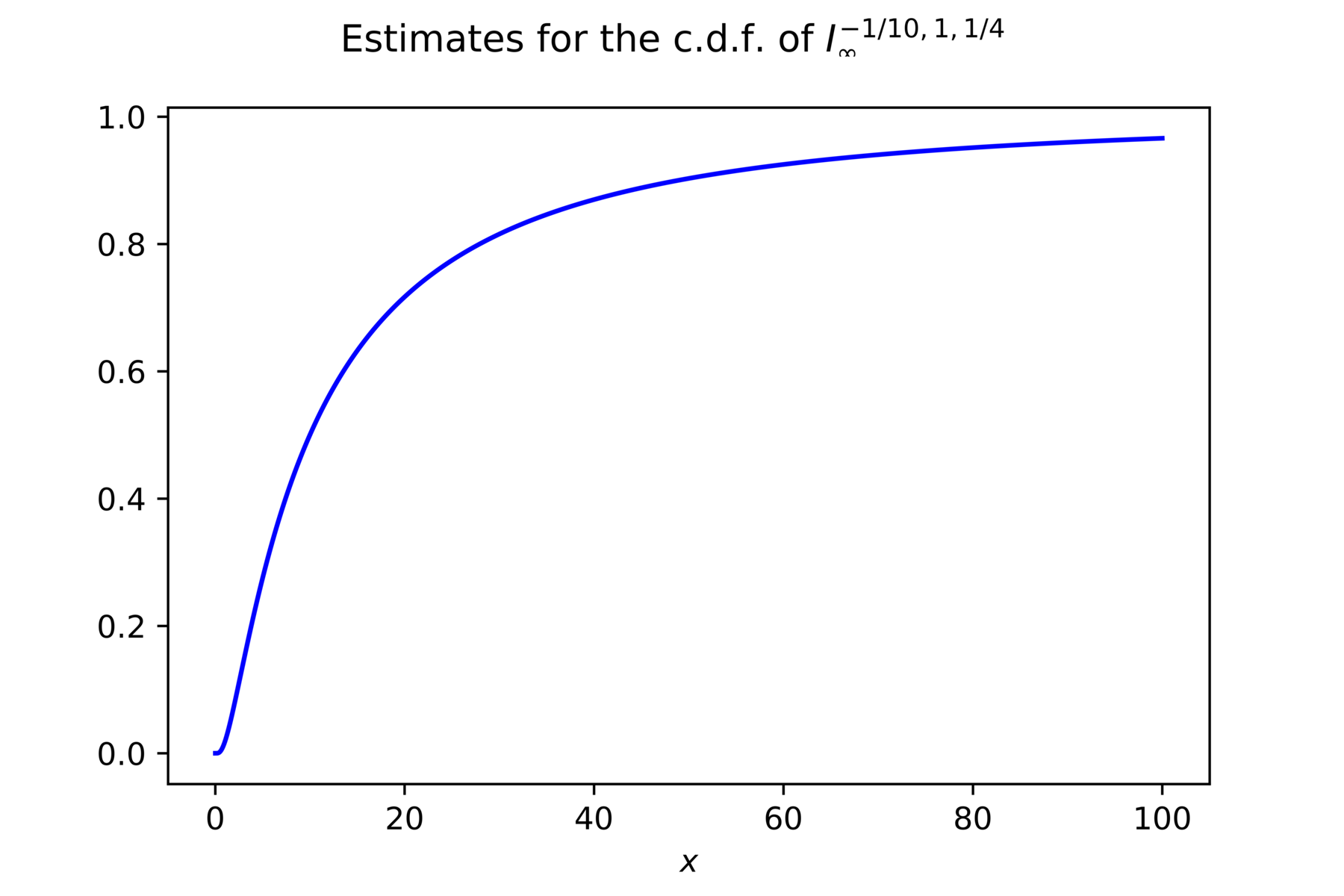}
		\end{subfigure}\hfill
		\begin{subfigure}[t]{0.26\textwidth}
			\includegraphics[width=\linewidth]{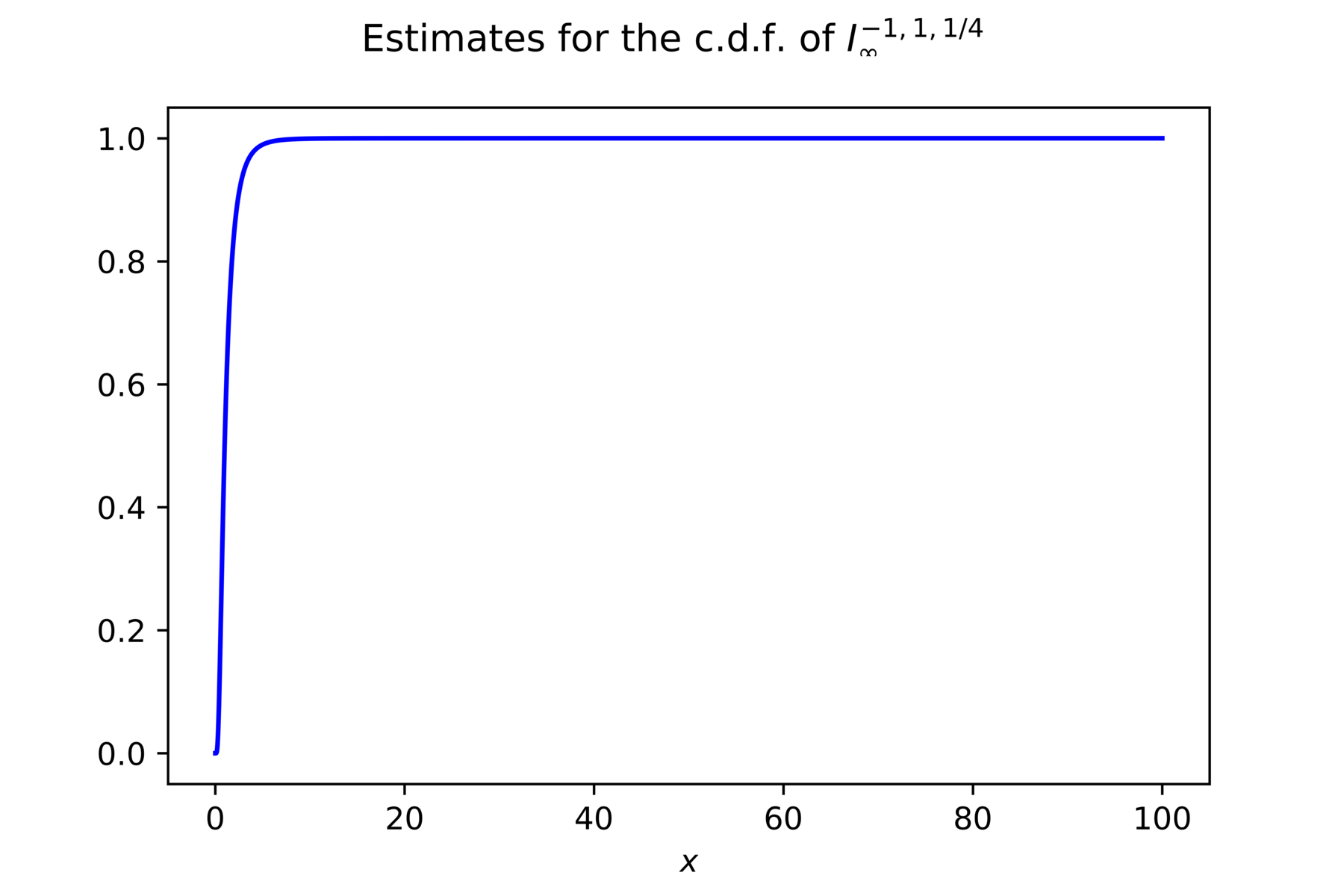}
		\end{subfigure}\hfill
		\begin{subfigure}[t]{0.26\textwidth}
			\includegraphics[width=\linewidth]{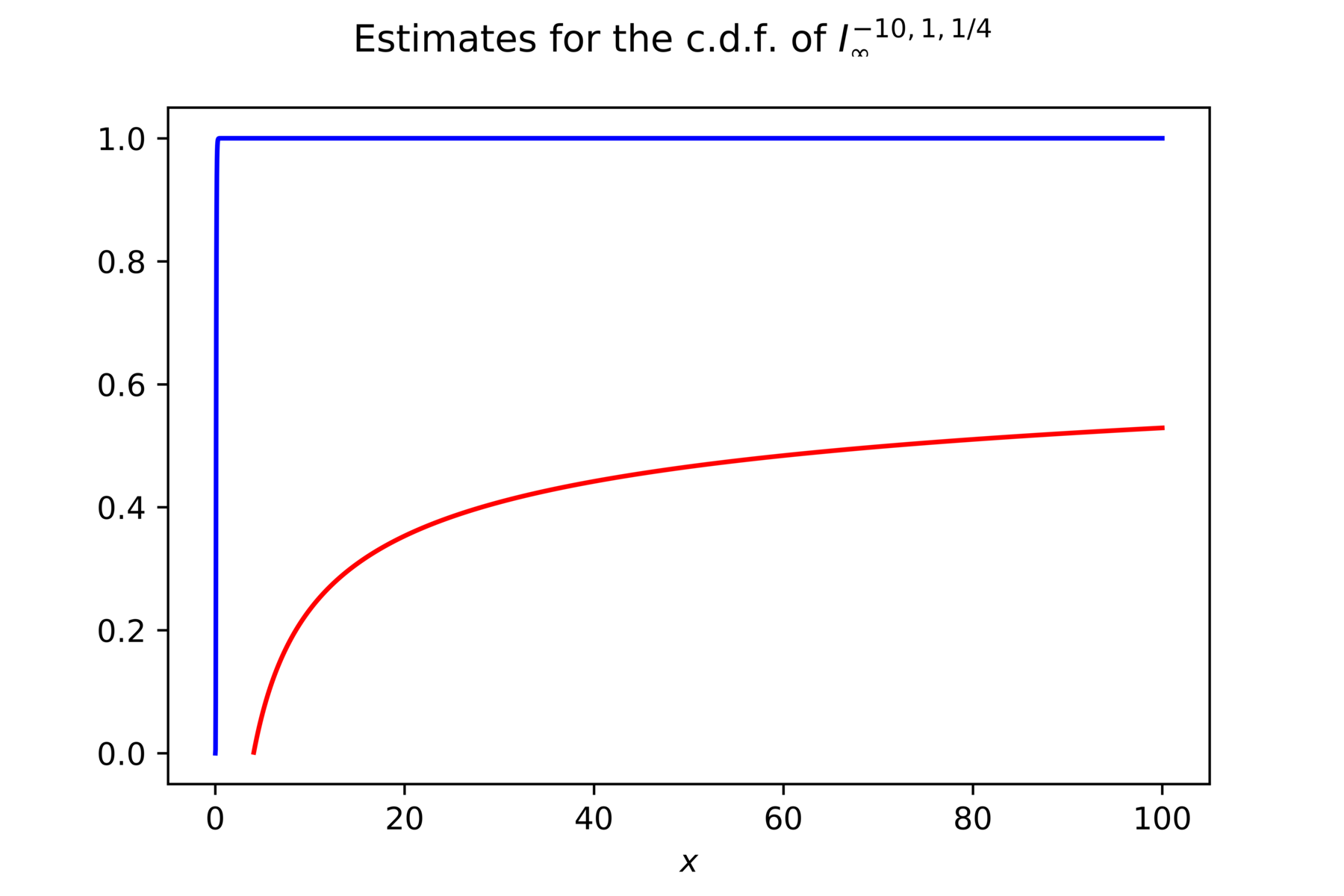}
		\end{subfigure}
		
		\begin{subfigure}[t]{0.26\textwidth}
			\includegraphics[width=\linewidth]{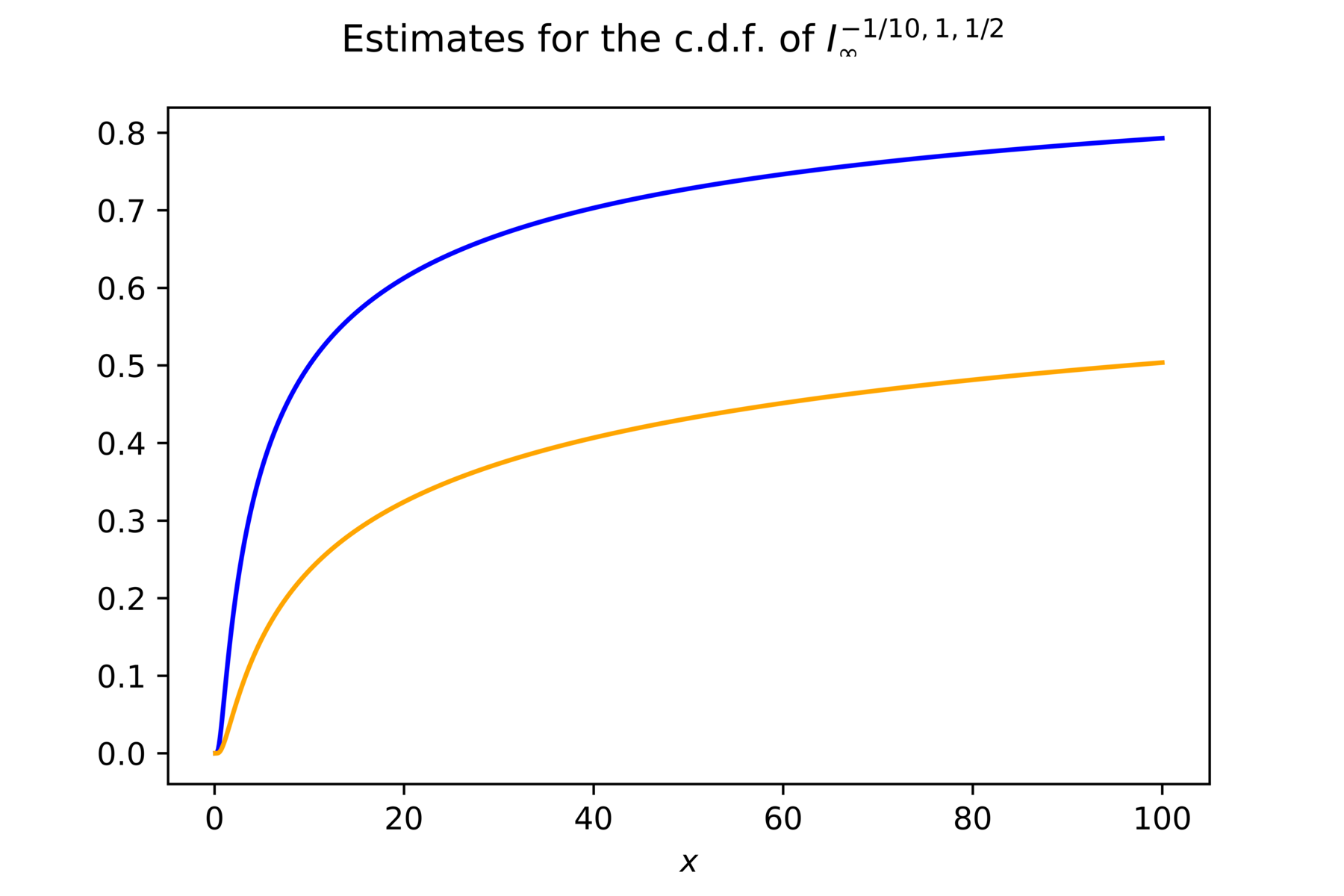}
		\end{subfigure}
		\hfill
		\begin{subfigure}[t]{0.26\textwidth}
			\includegraphics[width=\linewidth]{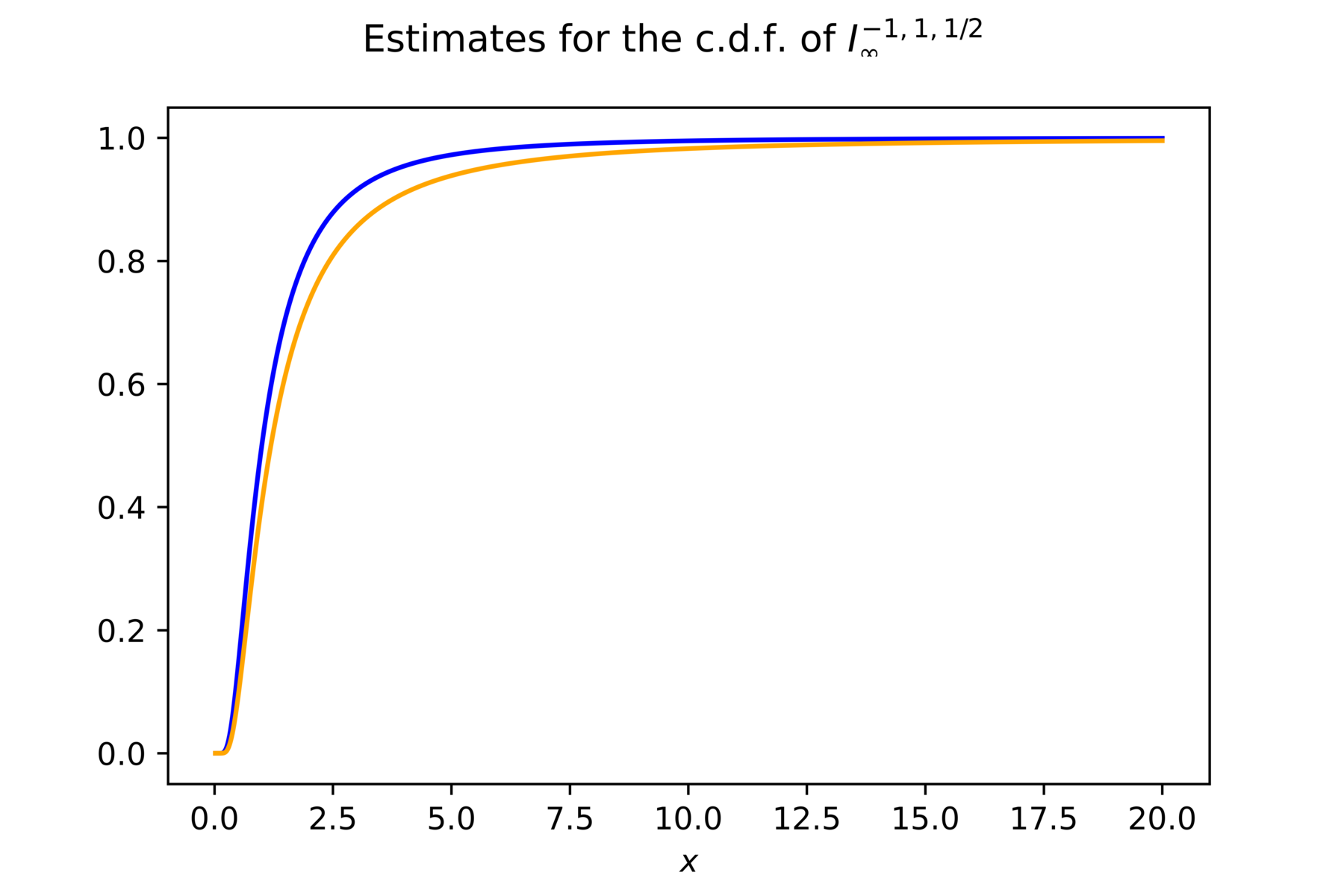}
		\end{subfigure}\hfill
		\begin{subfigure}[t]{0.26\textwidth}
			\includegraphics[width=\textwidth]{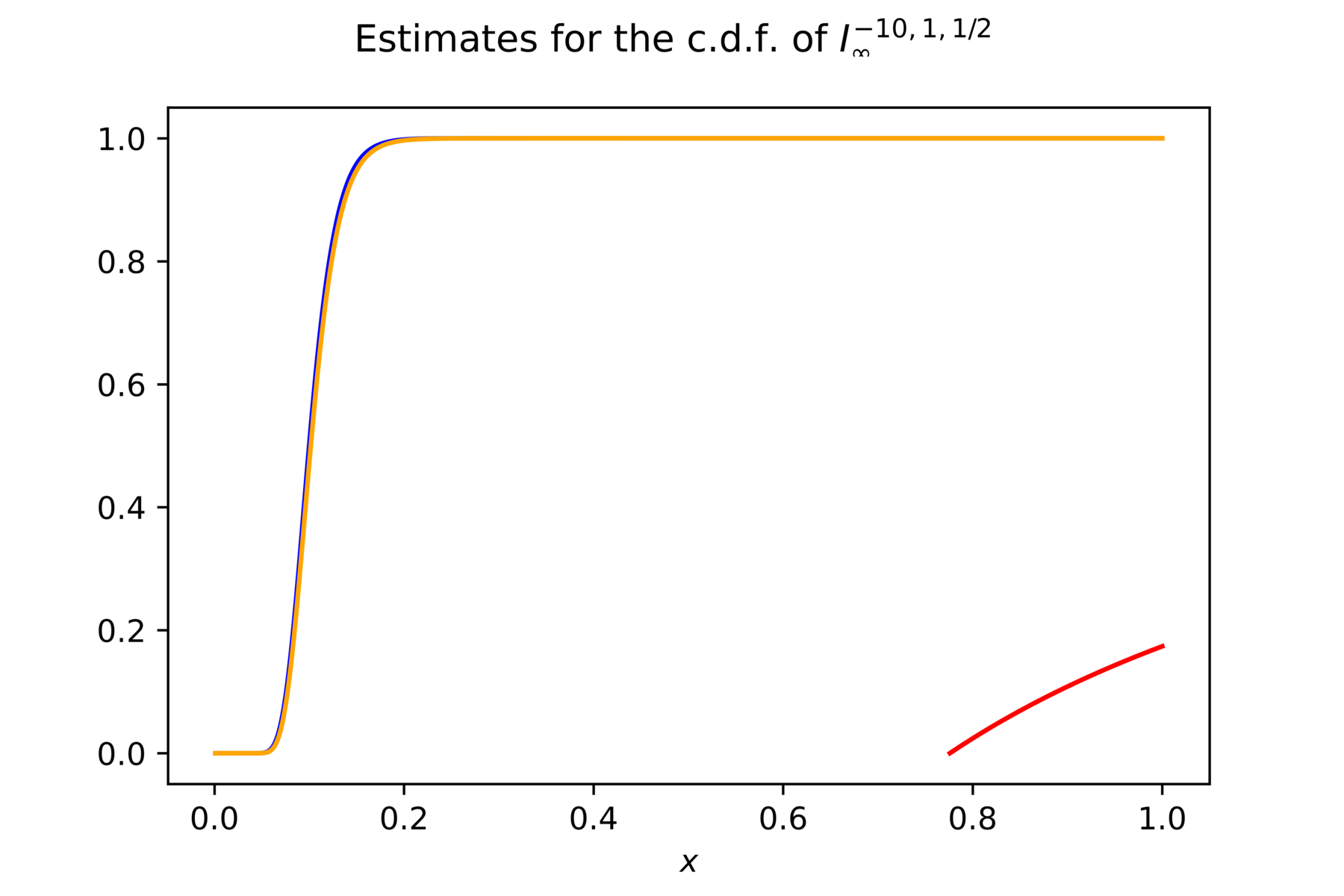}
		\end{subfigure}
		
		\begin{subfigure}[t]{0.26\textwidth}
			\includegraphics[width=\linewidth]{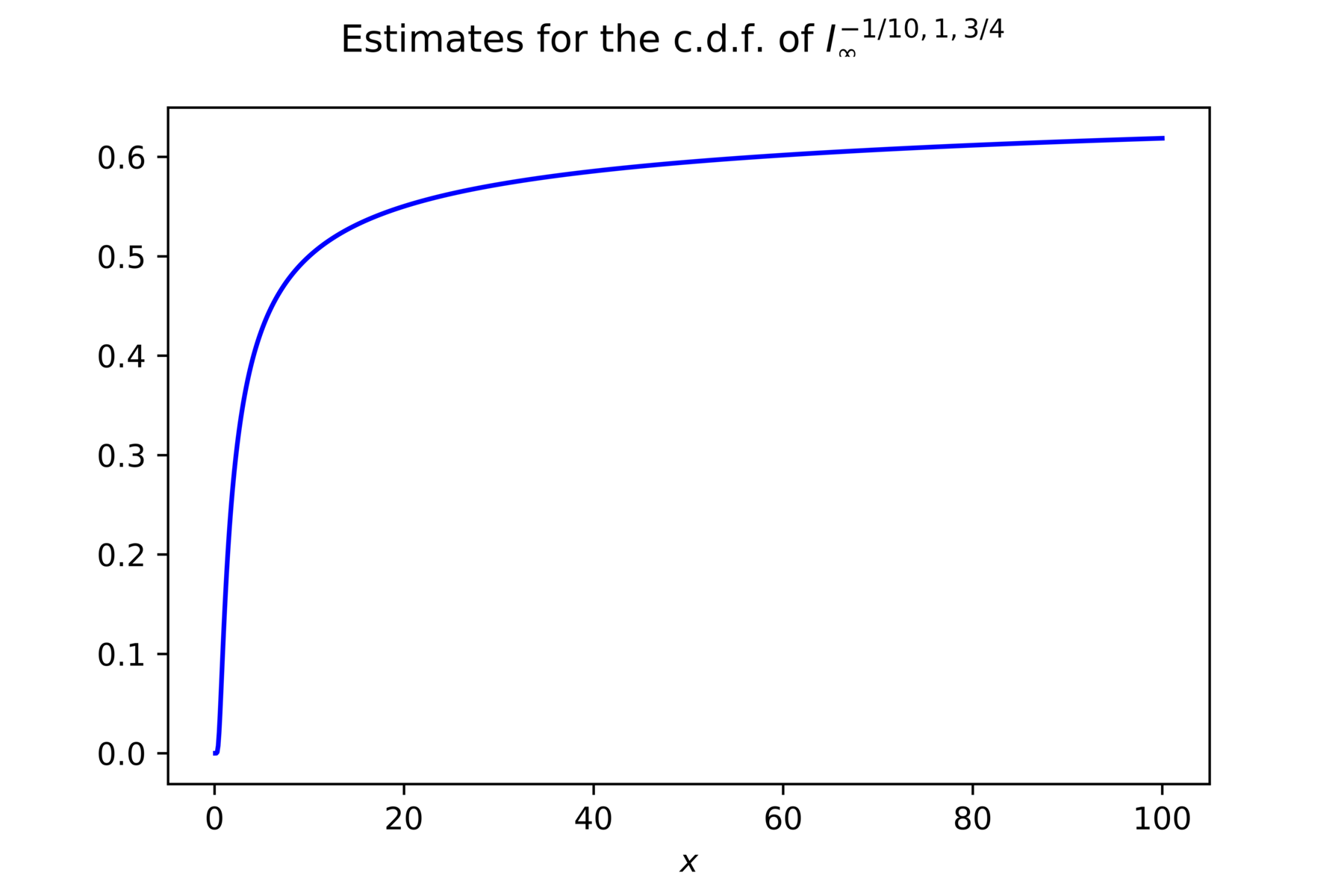}
		\end{subfigure}\hfill
		\begin{subfigure}[t]{0.26\textwidth}
			\includegraphics[width=\linewidth]{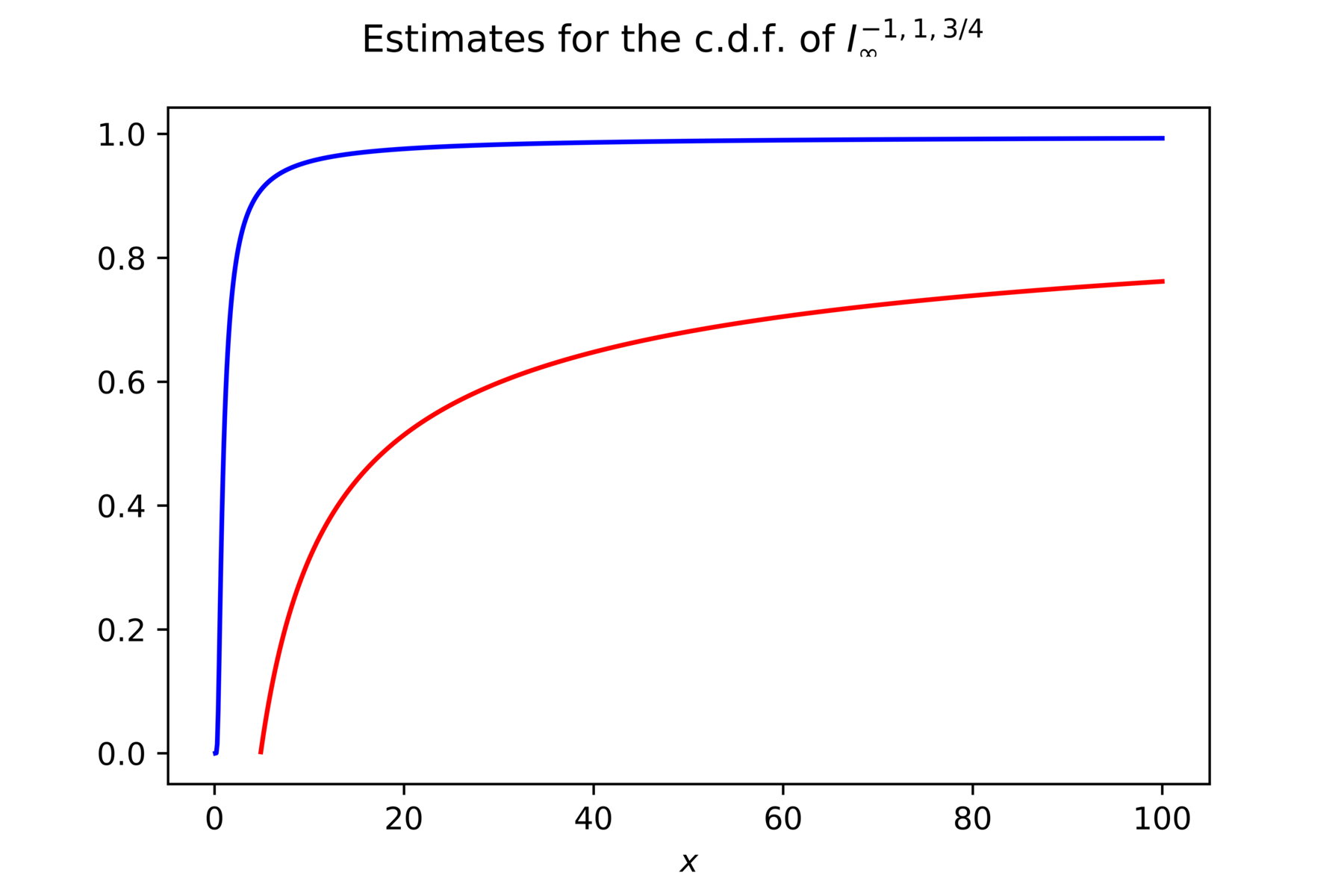}
		\end{subfigure}\hfill
		\begin{subfigure}[t]{0.26\textwidth}
			\includegraphics[width=\linewidth]{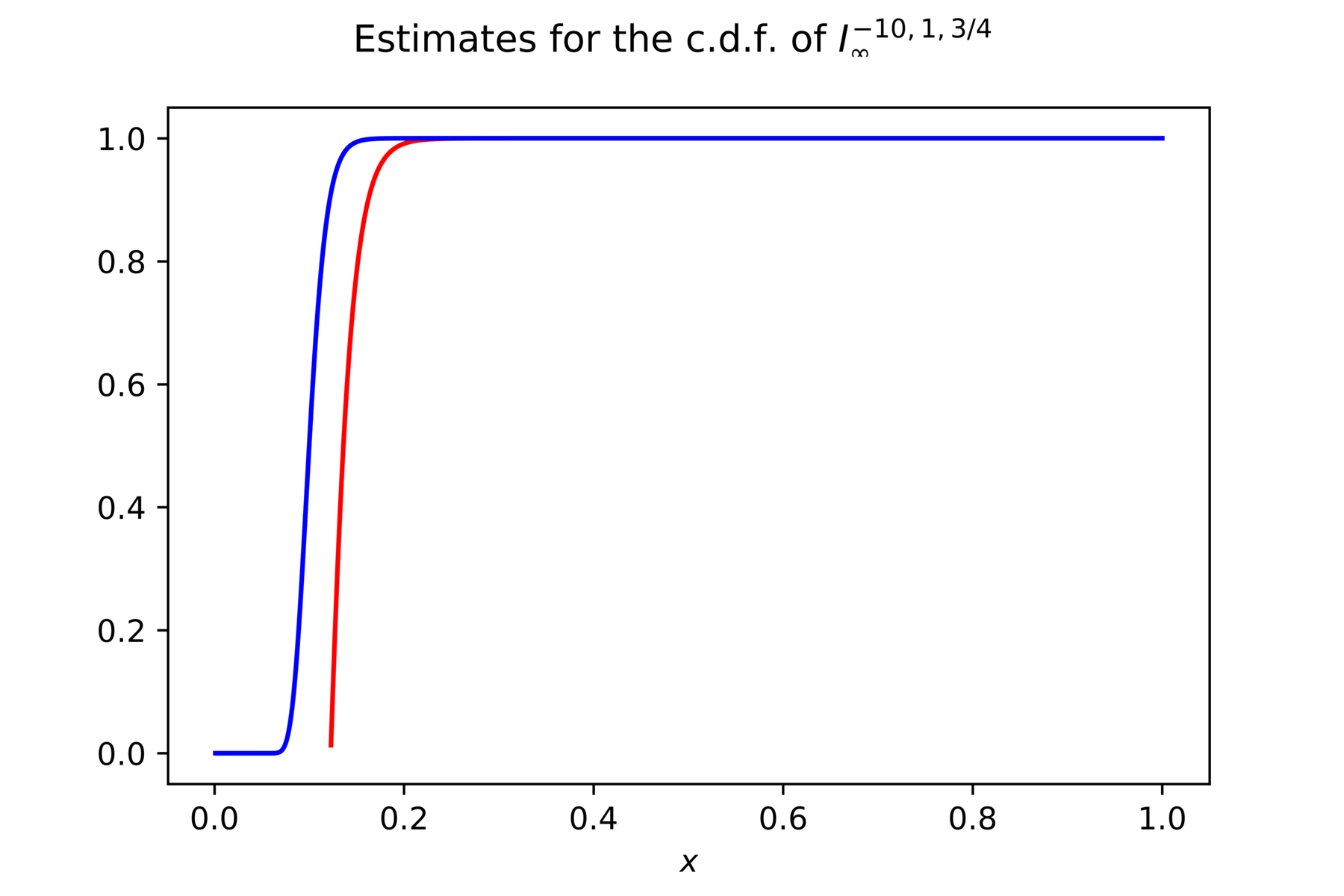}
		\end{subfigure}
		
		\caption{Upper bounds (blue lines) are derived from Theorem \ref{bounds_infinity} (iii). 
			Lower bounds (red lines) are derived from Theorem \ref{pro-lowe-bound-func-infinite}.
			The exact c.d.f.\ of $I_\infty^{\mu,\sigma,1/2}$ (orange lines) is given by (\ref{func_exp_brow}).}
		\label{figure:infinite}
	\end{figure}
	
	Now, for the case where $T=\infty$, we consider the values $\mu=-1/10,-1,-10$, $\sigma=1$ and $H=1/4,1/2,3/4$.
	We observe in Figure \ref{figure:infinite} that $I_\infty^{\mu,\sigma,H}$ has a heavy-tailed distribution.
	This is especially notorious in the case $\mu=-1/10$ and $H=3/4$.
	Since our lower bounds are non-trivial for very large values of $x$,
	they do not appear in most of the plots.
	So our lower bounds do not provide us with insightful information at most of the considered cases.	
	Nevertheless, in the case $H=3/4$, we notice that both bounds appear for large negative values of $\mu$, such as $\mu=-1,-10$.
	Let us recall from Proposition \ref{convergence-result-infinite-case} that $I_\infty^{\mu,\sigma,H}$ converges to zero a.s.\ as $\mu\rightarrow -\infty$,  which justifies the shape of the upper bound in these  cases.
	In contrast to the case where $T<\infty$, we are not able to simulate the random variable $I_\infty^{\mu,\sigma,H}$.
	Although we know that $I_T^{\mu,\sigma,H}\rightarrow I_\infty^{\mu,\sigma,H} $ a.s.\ as $T\rightarrow\infty$, it is not clear how large we should take $T$.
	For the case  $H<1/2$, we know from  Corollary \ref{cor-moment-func-infinte} (iv) that all the moments of $ I_\infty^{\mu,\sigma,H}$ are finite. 
	Therefore, the tail of $ I_\infty^{\mu,\sigma,H}$, with $H<1/2$, decays faster than any power function $1/x^{p}$, $p>0$, as $x\rightarrow\infty$.
	In the special case $H=1/2$, we know the exact c.d.f.\ of $I_\infty^{\mu,\sigma,H}$, which is given by (\ref{func_exp_brow}).
	We observe that our upper bound is close to the exact c.d.f.\ of $I_\infty^{\mu,\sigma,1/2}$ as $\mu\rightarrow-\infty$.
	However, for small values of $\mu$, such as $\mu=-1/10$, we observe that there is a notorious discrepancy between our upper bound and the exact c.d.f.\ of $I_\infty^{\mu,\sigma,1/2}$.
	This is due to the fact that the inverse Gamma distribution has, in general, a heavier tail than the log-normal distribution.
	The Kolmogorov distance between the c.d.f.\ of $I_\infty^{\mu,\sigma,1/2}$ and the upper bound obtained in Theorem \ref{bounds_infinity} (iii) may be bounded as
	\begin{align}
		&\sup_{x>0} \left(F(x)-	P\left[I_\infty^{\mu,\sigma,1/2}\leq x\right]\right)\nonumber\\
		&\leq \max\left\lbrace\sup_{x\in [0,T]}\left( F(x)-	P\left[I_\infty^{\mu,\sigma,1/2}\leq x\right]\right), \sup_{x>T}\left( 1-	P\left[I_\infty^{\mu,\sigma,1/2}\leq x\right]\right)\right\rbrace \nonumber \\
		&\leq \max\left\lbrace\sup_{x\in [0,T]}\left( F(x)-	P\left[I_\infty^{\mu,\sigma,1/2}\leq x\right]\right), P\left[I_\infty^{\mu,\sigma,1/2}\geq T\right]\right\rbrace\label{kolmogorov-distance}
	\end{align}
	for any $T>0$, where $F(x)\coloneqq\Phi\left( \frac{2\sqrt{-2\mu }}{\sigma}\left(\sqrt{W(-\mu x e)}-1/\sqrt{W(-\mu x e)}\right)\right)$ for $\mu<0$ and $\sigma,x>0$.
	In Figure \ref{figure:kolmogorov}, we appreciate that this distance decreases quickly to zero as $\mu\rightarrow-\infty$.  
	For the case $H>1/2$, from Corollary \ref{cor-moment-func-infinte} (ii) we know that the $p$-th order moment of $I_\infty^{\mu,\sigma,H}$ is infinite for any $p>0$.
	This tell us that $I_\infty^{\mu,\sigma,H}$ has a "very heavy tail" when $H>1/2$.
	In this case, the c.d.f.\ of $I_\infty^{\mu,\sigma,H}$ can not be lower bounded by a log-normal c.d.f.   
	Nevertheless, our lower bound allows us to  have an idea about how fast the tail  of $I_\infty^{\mu,\sigma,H}$ decays as $x\rightarrow\infty$ when $H\in(1/2,1)$.
	Replacing $\lambda^{\ast}$ by $\mu/2$ in the proof of Theorem  \ref{pro-lowe-bound-func-infinite} and using Lemma 2.1 in \cite{debicki}, we have 
	\begin{align*}
		P[I_\infty^{\mu,\sigma,H}>x]
		&\leq l_{H}\left(\frac{1}{\sigma^{\frac{1}{1-H}}}\left(-\frac{\mu}{2}\right)^{\frac{H}{1-H}}\log\left(-\frac{\mu x}{2}\right)\right)\\
		&\sim \left(\frac{H}{1-H}\right)^{1/2}\exp\left(-\frac{1}{2\sigma^{2}}\left(\frac{1}{H}\right)^{2H}\left(\frac{1}{1-H}\right)^{2-2H}\left(-\frac{\mu}{2}\right)^{2H}\left(\log\left(-\frac{\mu x}{2}\right)\right)^{2-2H}\right)
	\end{align*}
	as $x\rightarrow\infty$.
	In conclusion, our bounds allow us to estimate $P\left[I_\infty^{\mu,\sigma,H}\leq x\right]$ for very small  or very large values of $x$.
	However, further research is needed to estimate $P\left[I_\infty^{\mu,\sigma,H}\leq x\right]$  for intermediate values of $x$, especially for the case where $\mu$ is small.
	
	\begin{figure}[htbp]
		\centering
		\includegraphics[width=0.4\textwidth]{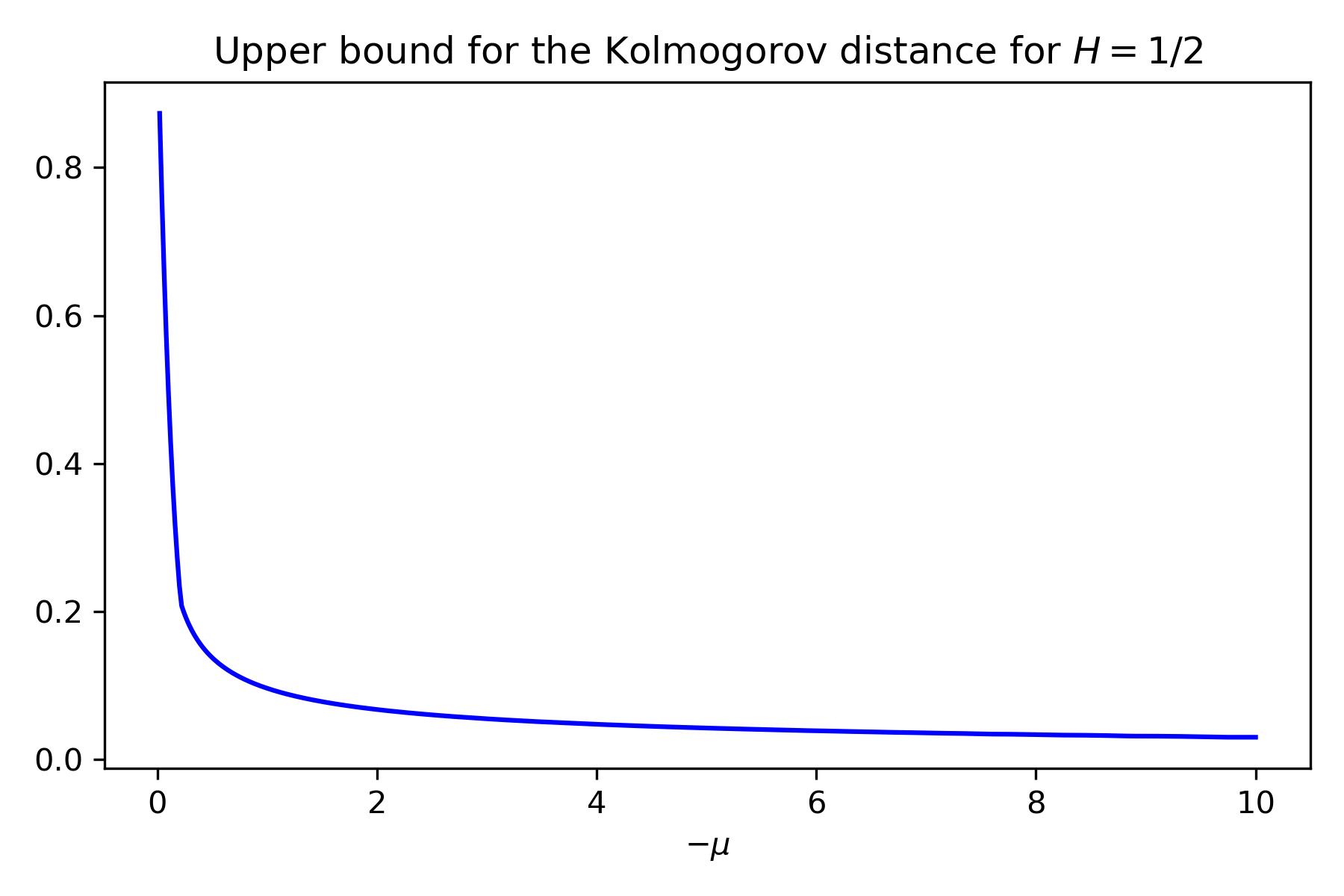}
		\caption{Upper bound for the Kolmogorov distance between the c.d.f.\ of $I_{\infty}^{\mu,1,1/2}$ and its upper bound given by \eqref{kolmogorov-distance}.
			We considered $T=100$ and $\mu\in [-10,0)$.	
			For $\mu=-10$ the distance reported is 0.029.}
		\label{figure:kolmogorov}
	\end{figure}

	\subsection{Comparison with previous results}

	To the best of our knowledge, the estimates for the law of exponential functionals of fBM reported in the literature are due to N.T. Dung and his coauthors.
	In \cite{dung}, the tail of a general class of exponential functionals was investigated.
	As a particular case, an upper bound for the tail of $I_T^{\mu,\sigma,H}$, with $T<\infty$, was obtained.
	As we mentioned in Remark \ref{improved-bound}, for $H<1/2$ such estimate and the bound we proved in Theorem \ref{lower-bound-finite-ii} are similar as $x\rightarrow \infty$. 
	However, the estimate obtained in \cite{dung} is informative for a larger set of values of $x$.
	This is illustrated in the left plot in Figure \ref{figure:comparison}.
	On the other hand, for $H\geq 1/2$ our estimate is sharper.
	This is shown in the right plot in Figure  \ref{figure:comparison}.
	
	\begin{figure}[!htbp]
		\begin{subfigure}[t]{0.4\textwidth}
			\includegraphics[width=\linewidth]{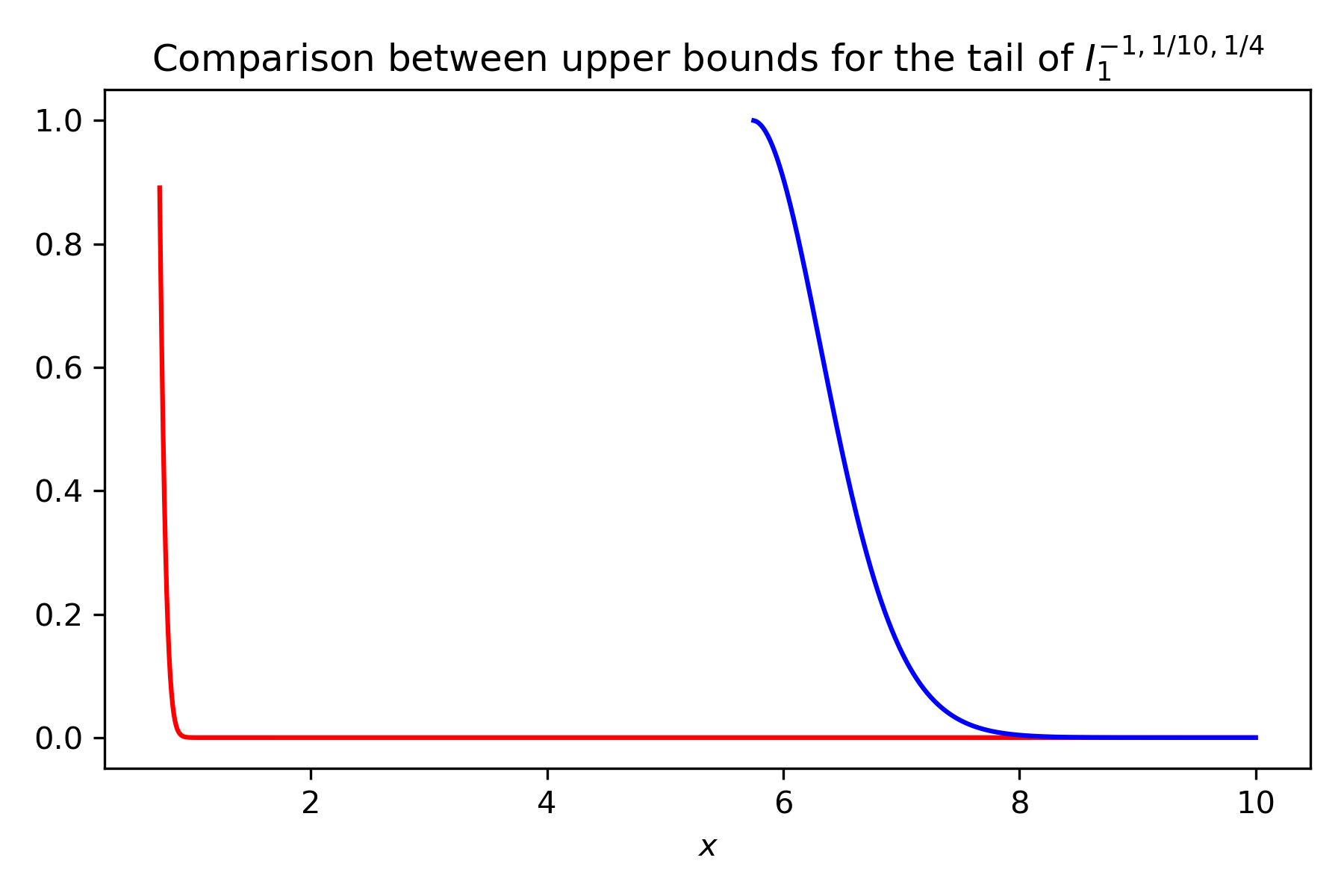}
		\end{subfigure}\hfill
		\begin{subfigure}[t]{0.4\textwidth}
			\includegraphics[width=\linewidth]{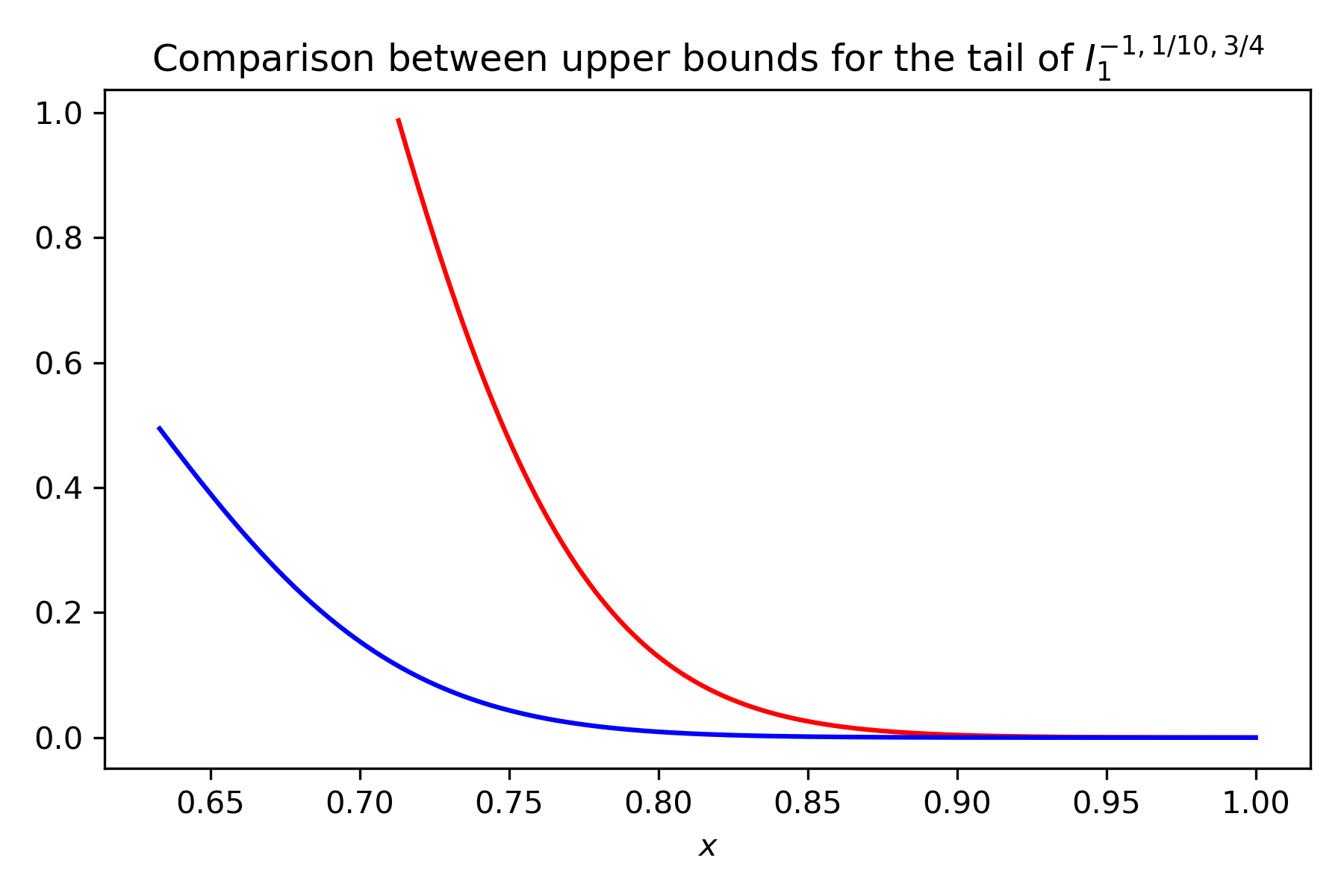}
		\end{subfigure}
			\caption{The red lines are upper bounds for the tail of $I_T^{\mu,\sigma,H}$ obtained in \cite{dung}. The blue lines are upper bounds for the tail of $I_T^{\mu,\sigma,H}$ proved in Theorem  \ref{lower-bound-finite-ii}. }
		\label{figure:comparison}
	\end{figure}

		In \cite{dung1}, an upper bound for the c.d.f.\ of $I_\infty^{-\mu,\sigma,H}$ was proved for any $\mu,\sigma>0$ and $H\in(0,1)$.
		It states that
		\begin{align*}
			P\left[I_\infty^{-\mu,\sigma,H}<x\right]\leq  \exp\left(-\frac{1}{2\sigma^{2}}\left(\frac{\alpha}{\alpha-H}\right)^{2-2H/\alpha}(\log(x+1))^{2-2H/\alpha}(m_x(\alpha)-1)^{2}\right), \quad x>0,
		\end{align*}
		for any $\alpha>H$, where
		\begin{align*}
			m_x(\alpha)\coloneqq \mathbb{E}\left[\sup_{t\geq 0} \frac{\log\left(\int_0^{t}e^{-\mu s-\frac{1}{2}\sigma^{2}s^{2H}+\sigma B_s^{H}}\,\mathrm{d}s\right)+t^{\alpha}}{\log(x+1)+t^{\alpha}}\right].
		\end{align*}
		However, it is not clear how to numerically compute this expression.
		In contrast, we have obtained explicit computable upper bounds for $P\left[I_\infty^{-\mu,\sigma,H}\leq x\right]$ for any $\mu\in\mathbb{R}$, $\sigma>0$ and $H\in(0,1]$ in Theorem \ref{bounds_infinity}.

		The Kolmogorov distance between $I_T^{\mu,\sigma,H_1}$ and $I_T^{\mu,\sigma,H_2}$ is studied in \cite{dung-dist} for any $H_1,H_2\in(0,1)$ and $T<\infty$.
		The main result states that 
		\begin{align*}
			\sup_{x\geq 0} \left|P\left[I_T^{\mu,\sigma,H_1}\leq x\right]-P\left[I_T^{\mu,\sigma,H_2}\leq x\right]\right|\leq C|H_1-H_2|,
		\end{align*}
		where $C$ is a positive constant depending on $\mu,\sigma,T,H_1$ and $H_2$.
		We proved in Corollary \ref{corollary:kolmogorov_distance} that $\sup_{x\geq 0} \left|P\left[I_T^{\mu,\sigma,H_1}\leq x\right]-P\left[I_T^{\mu,\sigma,H_2}\leq x\right]\right|\rightarrow 0$ as $H_1\rightarrow H_2$ for $T<\infty$.
		We remark that this convergence is not implied by the result in \cite{dung-dist} since the constant $C$ depends on $H_1,H_2$.
		The Kolmogorov distance between $I_\infty^{\mu,\sigma,H_1}$ and  $I_\infty^{\mu,\sigma,H_2}$ has not been investigated in the literature.

	Finally, in \cite{dung2} a log-normal upper bound for the density function  of $I_T^{\mu,\sigma,H}$ was proved for $H\in(0,1)$ and $T<\infty$.
	This was achieved by exploiting the fact that $I_T^{\mu,\sigma,H}$ is  Malliavin differentiable for $T<\infty$.
	However, this bound depends on unknown constants.
	Our present results do not address the density of $I_T^{\mu,\sigma,H}$.
	Nevertheless, both the upper bound obtained in \cite{dung2} and our estimates yield that the c.d.f.\ of $I_T^{\mu,\sigma,H}$ is upper bounded by a log-normal c.d.f.
	The existence of the density  of $I_\infty^{\mu,\sigma,H}$ has not been proved yet.\newline

	\noindent \textbf{Acknowledgement} This work was partially supported by CONACyT grant No. 652255.
	The authors are grateful to two anonymous referees for their constructive comments and suggestions,
	which greatly improved this paper.


\bigskip
	
	\begin{thebibliography}{99}
		
		\bibitem{bert}
		J. Bertoin, M. Yor. Exponential functionals of Lévy processes. \textit{Probab. Surv.} 2 (2005), 191–212.  
		
		\bibitem{billi}
		P. Billingsley. \emph{Convergence of Probability Measures}. John Wiley \&
		Sons Inc., New York, 1968.
		
		\bibitem{bise}
		K. Bisewski, K. Dębicki, M. Mandjes. Bounds for expected supremum of fractional Brownian motion with drift. \emph{J. Appl. Probab.} 58 (2021), no. 2,  411-427.
		
		\bibitem{borodin}
		A. N. Borodin, P. Salminen.
		\emph{Handbook of Brownian motion—facts and formulae.
		Second edition.}
		Probab. Appl.
		Birkh\"{a}user Verlag, Basel, 2002.
		
		\bibitem{boro}
		K. Borovkov, Y. Mishura, A. Novikov, M. Zhitlukhin.
		Bounds for expected maxima of Gaussian processes and their discrete approximations. \emph{Stochastics} 89 (2017), no. 1, 21–37. 
		
		\bibitem{che}
		P. Cheridito. Mixed fractional Brownian motion. \emph{Bernoulli} 7 (2001), no. 6, 913–934. 
		
		\bibitem{chow}
		Y. S. Chow, H. Teicher. \emph{Probability theory. Independence, interchangeability, martingales. Third edition}. Springer Texts in Statistics. Springer-Verlag, New York, 1997. 
		
		\bibitem{comte}
		A. Comtet, C. Monthus, M. Yor. Exponential functionals of Brownian motion and disordered systems. \emph{J. Appl. Probab.} 35 (1998), no. 2, 255–271.
		
		\bibitem{debicki}
		K. D\k{e}bicki, Z. Michna, T. Rolski. On the supremum from Gaussian processes over infinite horizon. \emph{Probab. Math. Statist.} 18 (1998), no. 1, Acta Univ. Wratislav. No. 2045, 83–100. 
		
		\bibitem{dieker}
		A. B. Dieker.
		Extremes of Gaussian processes over an infinite horizon. \emph{Stochastic Process. Appl.} 115 (2005), no. 2, 207–248. 
		
		\bibitem{dozzi2}  M. Dozzi, E. T. Kolkovska, J. A. López-Mimbela. Finite-time blowup and existence of global positive solutions of a semi-linear stochastic partial differential equation with fractional noise. \emph{Modern stochastics and applications}, 95–108, Springer Optim. Appl., 90, Springer, Cham, 2014.
		
		\bibitem{dozzi3}  M. Dozzi, E. T. Kolkovska, J. A. López-Mimbela. Global and non-global solutions of a fractional reaction-diffusion equation perturbed by a fractional noise. \emph{Stoch. Anal. Appl.} 38 (2020), no. 6, 959–978.
		
		\bibitem{dozzi}
		M. Dozzi, J. A. López-Mimbela.
		Finite-time blowup and existence of global positive solutions of a semi-linear SPDE. \textit{Stochastic Process. Appl.} 120 (2010), no. 6, 767–776.
		
		\bibitem{duf}
		D. Dufresne. The distribution of a perpetuity, with applications to risk theory and pension funding. \textit{Scand. Actuar. J.} 1990, no. 1-2, 39–79.
		
		\bibitem{dung}
		T. N.  Dung. Tail estimates for exponential functionals and applications to SDEs. \textit{Stochastic Process. Appl.} 128 (2018), no. 12, 4154–4170.
		
		\bibitem{dung1}
		N. T. Dung. The probability of finite-time blowup of a semi-linear SPDE with fractional noise. \textit{Statist. Probab. Lett.} 149 (2019), 86–92.
		
		\bibitem{dung-dist}
		N. T. Dung. Kolmogorov distance between the exponential functionals of fractional Brownian motion. \emph{C. R. Math. Acad. Sci. Paris} 357 (2019), no. 7, 629–635.
		
		\bibitem{dung2}
		N. T. Dung, N. T. Hang, P. T. P. Thuy. Density estimates for the exponential functionals of fractional Brownian motion. \emph{Comptes Rendus Mathématique}, Volume 360 (2022), pp. 151-159. 
		
		\bibitem{guerrero}
		E. Guerrero, J. A. López-Mimbela. Perpetual integral functionals of Brownian motion and blowup of semilinear systems of SPDEs. \emph{Commun. Stoch. Anal.} 11 (2017), no. 3, Article 5, 335–356.
		
		\bibitem{kallenberg}
		O. Kallenberg. \emph{Foundations of modern probability. Third edition}. Probability Theory and Stochastic Modelling, 99. Springer, Cham, 2021. 
		
		
		\bibitem{jose}
		J. A. López-Mimbela, A. Murillo-Salas, J. H. Ramírez-González.
		Occupation time fluctuations of an age-dependent critical binary branching particle system.  	
		Preprint, 2023.
		https://doi.org/10.48550/arXiv.2301.11540
		
		\bibitem{mishu}
		Y. Mishura. \textit{Stochastic calculus for fractional Brownian motion and related processes}. Lecture Notes in Mathematics, 1929. Springer-Verlag, Berlin, 2008. 
		
		\bibitem{mishu2}
		Y. Mishura, M. Zili. \emph{Stochastic analysis of mixed fractional Gaussian processes}. ISTE Press, London; Elsevier Ltd, Oxford, 2018.
		
		\bibitem{nourd}
		I. Nourdin. \textit{Selected aspects of fractional Brownian motion.} Bocconi \& Springer Series, 4. Springer, Milan; Bocconi University Press, Milan, 2012.
		
		\bibitem{nualart}
		D. Nualart. \emph{The Malliavin calculus and related topics. Second edition}. Probability and its Applications (New York). Springer-Verlag, Berlin, 2006.
		
		\bibitem{orey}
		S. Orey. Growth rate of certain Gaussian processes. \emph{Proceedings of the Sixth
		Berkeley Symposium on Mathematical Statistics and Probability Vol. II: Probability Theory}, pp.
		 443–451. University of California Press, Berkeley, 1972.
		
		\bibitem{priv}
		C. Pintoux, N. Privault. A direct solution to the Fokker-Planck equation for exponential
		Brownian functionals. \emph{Anal. Appl. (Singap.)} 8 (2010), no. 3, 287–304
		
		\bibitem{salminen}		
		P. Salminen, M. Yor. Perpetual integral functionals as hitting and occupation times. \emph{Electron. J. Probab.} 10 (2005), no. 11, 371–419.
		
		\bibitem{shao}
		Q. M. Shao. Bounds and estimators of a basic constant in extreme value theory of Gaussian
		processes. \emph{Statist. Sinica} 6 (1996), pp. 245–257.
		
		\bibitem{tudor}
		C. A. Tudor. \emph{Analysis of variations for self-similar processes. A stochastic calculus approach}. Probability and its Applications (New York). Springer, Cham, 2013.
		
		\bibitem{vos}
		L. Vostrikova. On distributions of exponential functionals of the processes with independent increments. \emph{Modern Stoch. Theory Appl.} 7(2020), no. 3, 291-313,
		
		\bibitem{whitt}
		W. Whitt. Weak convergence of probability measures on the function space $C[0,\infty)$. \emph{Ann. Math. Statist.} 41 (1970), 939–944. 
		
		\bibitem{yor} M. Yor.
		\emph{Exponential functionals of Brownian motion and related processes}. Springer Finance. Springer-Verlag, Berlin, 2001.
		
	\end{thebibliography}
\end{document}